\documentclass[11pt]{amsart}
\usepackage{amsmath,amsfonts,amssymb,mathrsfs}
\usepackage{amssymb,mathrsfs,graphicx,enumerate,mathabx}
\usepackage{amssymb,amscd,amsthm,bbm}
\usepackage{mathtools}
\mathtoolsset{showonlyrefs}
\usepackage[retainorgcmds]{IEEEtrantools}
\usepackage{colortbl}
\usepackage{lipsum}
\usepackage{graphicx, subfigure}
\usepackage{graphicx}
\usepackage{hyperref}
\usepackage{comment}
\mathtoolsset{showonlyrefs}

\title[Aggregation model with fast diffusion on sphere]{Ground states and phase transitions for an aggregation model with fast diffusion on sphere}

\author[Fetecau]{Razvan C. Fetecau}
\address[Razvan C. Fetecau]{\newline Department of Mathematics, Simon Fraser University, 8888 University Dr., Burnaby, BC V5A 1S6, Canada}
\email{van@math.sfu.ca}

\author[Park]{Hansol Park}
\address[Hansol Park]{\newline Department of Mathematics, National Tsing Hua University, Section 2, Kuang-Fu Road, Hsinchu 30013, Taiwan}
\email{hansolpark@math.nthu.edu.tw}

\newtheorem{theorem}{Theorem}[section]
\newtheorem{lemma}{Lemma}[section]

\newtheorem{proposition}{Proposition}[section]
\newtheorem{remark}{Remark}[section]

\newcommand{\bbr}{\mathbb R}

\newcommand{\bbs}{\mathbb S}

\newcommand{\calA}{\mathcal{A}}
\newcommand{\calF}{\mathcal{F}}

\newcommand{\calP}{\mathcal{P}}

\newcommand{\Psym}{\mathcal{P}_{\mathrm{sym}}}

\def\d{\mathrm{d}}
\newcommand{\dS}{\mathrm{d}S} 
\newcommand{\dx}{\mathrm{d}S (x)}
\newcommand{\dy}{\mathrm{d}S (y)}

\newcommand{\rhou}{\rho_{\text{uni}}}
\newcommand{\rhoeps}{\rho^\epsilon}

\newcommand{\alphaone}{\tilde{\alpha}_{\kappa}}
\newcommand{\alphatwo}{\alpha_{\kappa}}
\newcommand{\brho}{\bar{\rho}}
\newcommand{\srho}{\bar{s}}

\newcommand{\dm}{d} 

\begin{document}


\begin{abstract} We consider a free energy on the sphere that contains an entropy associated to nonlinear fast diffusion, and a nonlocal interaction energy. The two components of the free energy compete with each other, as one favours spreading and the other promotes concentration, respectively. The model is a generalization  of the Onsager free energy with dipolar potential, used to study polymer orientation. We study the global energy minimizers of the energy functional, and in particular the various phase transitions that occur with respect to the strength of the nonlocal attractive interactions. In the considered regime, diffusion reduces as the density increases, for which reason the global energy minimizers can contain Dirac mass concentrations. We identify various ranges of the fast diffusion exponent and of the interaction strength, which give qualitatively different equilibria and ground states. The theoretical results are supported by numerical illustrations.
\end{abstract}

\maketitle 


\section{Introduction}
\label{sect:intro}
In this paper we investigate the minimizers of the free energy functional
\begin{equation}
\label{energy-sphere}
E[\mu] = \frac{1}{m-1} \int_{\bbs^\dm} \rho(x)^m \dx + \frac{\kappa}{4} \iint_{\bbs^\dm \times \bbs^\dm} \| x-y \|^2 \d \mu (x) \d \mu (y)
\end{equation}
defined on the space $\mathcal{P}(\bbs^\dm)$ of probability measures on the unit sphere $\bbs^\dm$, where the integration is with respect to the  standard Lebesgue measure $\dS$ of the sphere, and $\rho \, \dS$ represents the absolutely continuous part of the measure $\mu$ (the notion of absolute continuity is with respect to the measure $\dS$). Also, $0<m<1$ is the diffusion exponent, $\kappa>0$ represents the interaction strength, and $\|\cdot\|$ denotes the Euclidean norm in $\bbr^{\dm+1}$. 

The energy functional \eqref{energy-sphere} consists of two terms: the first is the entropy and the second is a nonlocal interaction energy, with an interaction potential $W:\bbs^\dm \times \bbs^\dm \to \bbr$ given by $W(x,y) = \frac{1}{2} \| x-y \|^2$. The two terms in \eqref{energy-sphere} have competing effects, as the first favours spreading, while the second promotes aggregation. In its most general form, the energy \eqref{energy-sphere} can be set up on arbitrary Riemannian manifolds $M$, with general interaction potentials $W:M\times M \to \bbr$. There is an extensive literature on this class of aggregation-diffusion energies and associated equations, due to their wide range of applications in sciences and engineering; it is beyond our scope to give a comprehensive review of this literature, we simply refer here to \cite{BailoCarrilloCastro2024, CarrilloCraigYao2019, GomezCastro2024} for some recent reviews on the topic. For example, nonlinear diffusion was considered in the Keller-Segel model for chemotaxis \cite{CarrilloCalvez2006}, in swarming models \cite{BurgerDiFrancescoFranek, BuFeHu14, TBL}, granular media \cite{CaMcVi2006}, machine learning \cite{PeletierShalova2025}, and opinion formation \cite{FagioliRadici2021}.

Particularly relevant to the current work, the free energy \eqref{energy-sphere} for the special case $m=1$ (in this case the entropic term is given by $\int_{\bbs^\dm} \rho(x) \log \rho(x) \dx$), known as the Onsager free energy  \cite{Onsager1949}, has been used to study isotropic-nematic phase transitions in rod-like polymers \cite{ConstantinKevrekidisTiti2004, fatkullin2005critical}. In this application, $\rho$ represents the probability distribution function for the orientation of a polymer interpreted as a rigid rod of unit vector $x \in \bbs^2$. The quadratic interaction potential in \eqref{energy-sphere} (in fact, its equivalent expression $-\kappa \langle x, y\rangle $) is called the dipolar potential.  As shown in  \cite{fatkullin2005critical,FrouvelleLiu2012, DegondFrouvelleLiu2014}, there exists a critical value of the interaction strength $\kappa$ at which the isotropic state (provided by the uniform distribution) loses stability to a nematic equilibrium density. Similar phase transitions are investigated in the present paper, but with nonlinear diffusion replacing the linear diffusion. The current work complements the recent research on phase transitions with nonlinear diffusion in the porous media regime ($m>1$) done in \cite{FePaVa2025}.

The range $0<m<1$ of the diffusion exponent corresponds to the fast diffusion regime; this is in contrast with the range $m>1$ for the slow diffusion (porous media) regime. The terminology comes from the evolution equation (gradient flow) associated to the entropic term in \eqref{energy-sphere}. Indeed, in a suitable Wasserstein space, the gradient flow of the entropy component (the interaction term in \eqref{energy-sphere} is ignored here) leads to the nonlinear diffusion equation $\displaystyle \rho_t=\Delta \rho^m$ \cite{AGS2005}. The nonlinear diffusion $ \Delta \rho^m $ can be written as $\nabla \cdot (m\rho^{m-1}\nabla\rho)$, and the coefficient $m\rho^{m-1}$ in front of $\nabla\rho$ is the (density-dependent) diffusion coefficient. As $0<m<1$, the diffusion coefficient becomes larger as $\rho$ gets smaller. Conversely, as $\rho$ increases, the diffusion coefficient becomes smaller, which suppresses the diffusion rate. Hence, regions of low density correspond to faster diffusion, while in regions of high density the diffusion is slower. As an extreme example, if $\rho$ contains a singular part at some point (i.e., $\rho = \infty$ at that point), then $\rho^{m-1}=0$, and the diffusion vanishes there. For example, fast diffusion behaviour can be observed in dilute plasmas or rarefied gases, where particles in low-density regions move more freely and spread rapidly \cite{berryman1978nonlinear, king1988extremely}.

This behaviour of the diffusion coefficient in fast diffusion is opposite to that in the slow diffusion (porous media) regime, where $m>1$. The two types of diffusion (fast and slow) are very different, and the fast diffusion is by far the one that is less studied and understood. In the Euclidean setting, aggregation–diffusion dynamics in the fast diffusion regime was studied in \cite{CaDeFrLe2022, CaFeAlGo2024}, where the so-called partial mass concentration phenomenon was introduced. This phenomenon occurs as the aggregation interactions lead to high densities, while diffusion is too weak in order to prevent the blow-up. Mass concentration occurs only in the fast diffusion regime, and does not appear for slow or linear diffusion.

An alternative viewpoint on fast versus slow diffusion comes from the variational perspective. Indeed, write the entropy term in the energy as $\frac{1}{m-1} \int \rho(x)^{m-1} \d \rho(x)$, where the integration is with respect to the measure $\d \rho = \rho \,\d S$. In the slow diffusion regime ($m>1$), as $\rho$ concentrates, $\rho^{m-1}$ diverges to infinity, which makes the entropy itself infinite. Hence, energy minimizers in this regime cannot have mass concentrations. However, for fast diffusion ($0<m<1$), $\rho^{m-1}$ tends to zero as $\rho$ increases, and the entropy remains finite even when $\rho$ develops a singular part. Due to this fundamental difference between the two regimes, in the present research we allow singular measures as admissible energy minimizers, whereas in previous work on the slow diffusion case only absolutely continuous measures were considered \cite{FePaVa2025}. 

Steady states and energy minimization for aggregation models with fast diffusion in Euclidean space have been studied recently in several papers. In \cite{carrillo2019reverse}, energy minimizers are analyzed using a reverse HLS (Hardy--Littlewood--Sobolev) inequality that relates the fast-diffusion entropy and the interaction energy. In \cite{CaHiVoYa2019}, the existence of a global minimizer is proved by a variational method, and radial symmetry is obtained via a rearrangement inequality. The partial mass concentration in the steady states of aggregation equations with fast diffusion is explored in \cite{CaDeFrLe2022} and \cite{CaFeAlGo2024}. The minimization problem for aggregation-diffusion energies has also been studied on Cartan--Hadamard manifolds \cite{CaFePa2025a, CaFePa2025b, FePa2024b, FePa2024a} - for example, the fast diffusion range is considered in \cite{CaFePa2025b}. On Cartan–Hadamard manifolds, the faster volume growth of geodesic balls effectively strengthens the diffusion, so the conditions ensuring the existence of an energy minimizer differ from the flat case. Specifically, such conditions depend on the behaviour of the manifold's curvature at infinity. We also note that the fast diffusion equation (with no aggregation terms) has been extensively studied in the literature, with particular interest on the well-posedness and long time behaviour of its solutions \cite{BlBoDoGrVa2009, bonforte2010sharp, BonforteFigalli2024, HerreroPierre1985, vazquez2006porous, vazquez2006smoothing}, including works on manifolds \cite{bonforte2008fast, grillo2021fast, lu2009local}.

In our study we consider the free energy set up on the unit sphere, with fast diffusion entropy and interactions modelled by a quadratic potential. As noted above, one of the main motivations for the present work is to extend the Onsager free energy  \cite{Onsager1949} to nonlinear diffusion in the fast regime, and investigate the phase transitions (in terms of the interaction strength $\kappa$) of its ground states. In our analysis we have identified three different subranges of $0<m<1$ (specifically, i) $1- \frac{2}{\dm} <m<1$, ii) $1- \frac{2}{\dm-1}<m< 1- \frac{2}{\dm}$, and iii) $0<m<1- \frac{2}{\dm-1}$), where the global minimizers are qualitatively different, and the phase transitions are different too. For example, although the entropy remains finite for all $0<m<1$, even when $\rho$ has a singular part, the global energy minimizer of \eqref{energy-sphere} contains a singular part only for $0<m<1-\frac{2}{\dm}$ (ranges ii) and iii) -- see Theorem~\ref{thm:global-min}), and any such singular part must be a single Dirac delta concentration. This result is consistent with the fast diffusion theory on the Euclidean space. For the fast diffusion equation on the Euclidean space, it is known that the Dirac delta mass is stable for $0<m<1-\frac{2}{\dm}$ and unstable for $1-\frac{2}{\dm}<m<1$; see \cite{BrezisFriedman1983} and Remarks \ref{rmk:rho-integr} and \ref{rmk:scaling-blowup}. Since concentration is a local phenomenon, this consistency of the ranges of $m$ with the Euclidean case is natural.
 
We also mention that there has been an increased interest recently on aggregation-diffusion models on sphere and other manifolds, due to applications of these models to machine learning \cite{huilearning, Maggioni-etal21}. We note in particular the application to large language models, specifically to transformers \cite{Geshkovski_etal2025}. Various recent works have been motivated by such application, where a common theme is phase transitions from the uniform distribution, similar to the focus of this paper \cite{Balasubramanian_etal2025, Gerber_etal2025, PeletierShalova2025, ShalovaSchlichting2025}. 

The rest of the paper is organized as follows. In Section \ref{sect:cp} we establish the variational framework by setting up the admissible class of minimizers, and derive the possible critical points from the Euler-Lagrange equations.  In Section \ref{sect:bifurcations-ml1} we find the critical values of $\kappa$ at which equilibria undergo phase transitions, and we characterize these equilibria for various ranges of $\kappa$ and $m$ (Proposition \ref{prop:bif}). Finally, in Section \ref{sect:minimizers} we compare the energies associated with the found critical points, and determine which ones yield the global minimizer (Theorem \ref{thm:global-min}).

\section{Admissible class and critical points of the energy}
\label{sect:cp}
In this section we present the admissible set of energy minimizers and derive various expressions for the equilibrium points. For any two points on the unit sphere $x,y\in \bbs^\dm$, $\|x-y\|^2$ can be expressed as 
\[
\|x-y\|^2=\|x\|^2+\|y\|^2-2\langle x, y\rangle=2-2\langle x, y\rangle,
\] 
which can be used to rewrite the energy \eqref{energy-sphere} as 
\begin{equation}
\label{eqn:energy-cmform}
E[\mu]= \frac{1}{m-1}\int_{\bbs^\dm}\rho(x)^m \dx-\frac{\kappa}{2}\iint_{\bbs^{\dm}\times \bbs^{\dm}}\langle x, y\rangle \, \d \mu(x)\d \mu(y)+\frac{\kappa}{2}.
\end{equation}
Here, $\langle \cdot, \cdot \rangle$ denotes the inner product in $\bbr^{\dm+1}$.

The centre of mass $c_\mu \in \bbr^{\dm+1}$ of a measure $\mu\in \mathcal{P}(\bbs^\dm)$ is defined as
\begin{equation}
\label{eqn:cm}
c_\mu=\int_{\bbs^{\dm}}x \, \d \mu(x).
\end{equation}
It holds that
\begin{equation}
\begin{aligned}
\|c_\mu\|^2 &=\left\langle
\int_{\bbs^{\dm}}x \, \d \mu(x), \int_{\bbs^{\dm}}y \, \d \mu(y)
\right\rangle \\[2pt]
&=\iint_{\bbs^\dm\times \bbs^\dm}\langle x, y\rangle \, \d \mu(x) \d \mu(y),
\end{aligned}
\end{equation}
and hence, the energy functional can be written as
\begin{equation}
\label{eqn:energy-s}
E[\mu]= \frac{1}{m-1}\int_{\bbs^\dm}\rho(x)^m \dx -\frac{\kappa}{2}\|c_\mu\|^2+\frac{\kappa}{2}.
\end{equation}

\subsection{Admissible class}
\label{subsect:class}
We first need to set up the set of admissible energy minimizers. Note that there are infinitely many critical points of the energy functional. For instance, any symmetric distribution of Dirac delta measures (e.g., a regular $n$-gon on the great circle or a regular $\dm$-simplex on $\bbs^\dm$) are equilibrium states of the energy. The following lemma excludes a large class of such equilibria from being candidates for energy minimizers.

\begin{lemma}
\label{Lemma:Dirac-delta-concentration}
Let $\mu\in\mathcal{P}(\bbs^\dm)$ be a critical point of the energy $E[\mu]$. If the singular part of $\mu$ with respect to $\dS$ is not concentrated at one point, then $\mu$ cannot be a global energy minimizer.
\end{lemma}
\begin{proof} A generic probability measure $\mu$ on $\bbs^\dm$ can be written as
\begin{equation}
\label{eqn:mu-decomp}
\mu=\alpha\mu_s+(1-\alpha)\rho \, \dS,
\end{equation}
where $0\leq \alpha\leq 1$, and $\alpha\mu_s$ and $(1-\alpha)\rho\,\dS$ are the singular and the absolutely continuous parts of $\mu$, respectively. Here, we assumed that both $\mu_s$ and $\rho \, \dS$ are in $\mathcal{P}(\bbs^\dm)$. 
Note that the regular part of $\mu$ here is $(1-\alpha) \rho \dS$, and the extra factor of $1-\alpha$ has to be considered when computing $E[\mu]$ (cf., \eqref{energy-sphere} or \eqref{eqn:energy-s}). We will show that provided the singular component of $\mu$ is not concentrated at one point, then a change of $\mu_s$ can decrease the energy $E[\mu]$. See Appendix \ref{appendix:Dirac-delta-concentration} for the rest of the proof.
\end{proof}

By Lemma \ref{Lemma:Dirac-delta-concentration}, we can assume that the singular part of a ground state $\mu$ is a single Dirac delta mass $\alpha\delta_{x_0}$ concentrated at a fixed point $x_0\in\bbs^\dm$. In the next lemma we show that a delta concentration with no regular part ($\alpha =1$) cannot be a global minimizer either.

\begin{lemma}
\label{lem:single-delta}
The Dirac delta concentration $\delta_{x_0}$ is not a global energy minimizer. 
\end{lemma}
\begin{proof}
As $\delta_{x_0}$ has no regular part and its centre of mass is at $x_0$ with $\|x_0\| =1$, by \eqref{eqn:energy-s} we have $E[\delta_{x_0}] =0.$ Consider the following perturbation of $\delta_{x_0}$:
\[
\mu^t=(1-t)\delta_{x_0}+\frac{t}{|\bbs^\dm|}\dS,
\]
where $0\leq t<1$. One can easily check that $\int_{\bbs^\dm}\d\mu^t=1$, and that $c_{\mu^t} = (1-t)x_0$. Then, by \eqref{eqn:energy-s} the corresponding energy can be expressed as 
\begin{align*}
E[\mu^t]&=\frac{t^m}{(m-1)|\bbs^\dm|^{m-1}}-\frac{\kappa}{2}(1-t)^2+\frac{\kappa}{2}\\
&=-\left(\frac{1}{(1-m)|\bbs^\dm|^{m-1}}\right)t^m+\kappa\left(t-\frac{t^2}{2}\right).
\end{align*}

As $0<m<1$, we have $E[\mu^t]<E[\mu^0]=E[\delta_{x_0}]$ for small $t>0$. In fact, we have
\[
\frac{\d}{\d t}E[\mu^t]\bigg|_{t=0+}=-\infty.
\]
We infer that $\delta_{x_0}$ is not a global energy minimizer.
\end{proof}

Next, we will investigate the absolutely continuous part of $\mu$. We denote by $\calP_{ac}(\bbs^\dm)\subset \calP(\bbs^\dm)$ the space of probability measures that are absolutely continuous with respect to $\d S$. By an abuse of notation we will often refer to an absolutely continuous measure $\rho \d S$ directly by its density $\rho$. 

\begin{lemma}[Radial symmetry of global minimizers]
\label{lem:symmetry-steady-states}
Let $\mu=\alpha\delta_{x_0}+(1-\alpha)\rho\dS$ be a global energy minimizer, where $0\leq \alpha<1$ and $\rho \in \calP_{ac}(\bbs^\dm)$. Then, 

\noindent a) The density $\rho$ is radially symmetric with respect to some $z \in\bbs^\dm$, i.e., for any two points $y_1, y_2 \in \bbs^\dm$ satisfying $\langle z,y_1\rangle=\langle z, y_2\rangle$, we have $\rho(y_1)=\rho(y_2)$. 

\noindent b) Furthermore, if $0<\alpha<1$, then $z=x_0$. That is, $\rho$ is symmetric with respect to $x_0$, and $c_\rho=\|c_\rho\|x_0$.
\end{lemma}

\begin{proof}
\textit{a}) Take a rotation $R:\bbs^\dm\to\bbs^\dm$ that preserves the centre of mass $c_\rho$ of $\rho$. Note that the centre of $\rho':=R_\#\rho$ is also $c_\rho$ (i.e., $c_{\rho'}=c_{\rho}$). Define
\[
\mu'=\alpha\delta_{x_0}+(1-\alpha)\rho'\dS.
\]
Then, 
\[
c_{\mu'}=\alpha x_0+(1-\alpha)c_{\rho'}=\alpha x_0+(1-\alpha)c_{\rho}=c_\mu,
\]
and by \eqref{eqn:energy-s}, this yields
\begin{align*}
E[\mu']&=\frac{1}{m-1}\int_{\bbs^\dm}\rho'(x)^m\dx-\frac{\kappa}{2}\|c_{\mu'}\|^2+\frac{\kappa}{2}\\
&=\frac{1}{m-1}\int_{\bbs^\dm}\rho(x)^m\dx-\frac{\kappa}{2}\|c_{\mu}\|^2+\frac{\kappa}{2} \\
&=E[\mu],
\end{align*}
where for the second equality we used that $\rho'$ is a rotation of $\rho$, and that $c_{\mu'}=c_\mu$.

Consider a convex combination of $\rho$ and $\rho'$:
\[
\rho_\lambda=(1-\lambda)\rho+\lambda\rho', \qquad 0<\lambda<1.
\]
Since $\frac{1}{m-1}\int_{\bbs^\dm}\rho(x)^m$ is a strictly convex functional, we have
\[
\frac{1}{m-1}\int_{\bbs^\dm}\rho_\lambda(x)^m\dx\leq (1-\lambda)\left(\frac{1}{m-1}\int_{\bbs^\dm}\rho(x)^m\dx\right)+\lambda\left(\frac{1}{m-1}\int_{\bbs^\dm}\rho'(x)^m\dx\right),
\]
where equality holds only for $\rho=\rho'$. Together with the fact that $c_{\mu_\lambda} = (1-\lambda) c_\mu + \lambda c_{\mu'} = c_\mu$, we get from \eqref{eqn:energy-s} that
\[
E[\mu_\lambda]\leq (1-\lambda)E[\mu]+\lambda E[\mu']=E[\mu],
\]
with equality holding only for $\rho=\rho'$. Since $E[\mu]$ is the global minimum, then $E[\mu_\lambda]=E[\mu]$ and hence, $\rho=\rho'$ for any rotation $R$ that preserves $c_\rho$. Let $z\in \bbs^\dm$ satisfy $c_\rho=\|c_\rho\|z$. Then, $\rho=R_\#\rho$ for any rotation $R$ that preserves $z$. It implies that $\rho$ is symmetric with respect to $z$ (or equivalently symmetric with respect to $-z$).
\medskip

\textit{b}) Now assume that $0<\alpha<1$, i.e., $\mu$ has both a Dirac delta and a regular component. From part a), we know that $\rho$ is symmetric with respect to $z$, where $c_\rho=\|c_\rho\|z$. If $\|c_\rho\|=0$ then $z$ can be any vector and we can set $z=x_0$. If  $\|c_\rho\|>0$, then express the energy as
\[
E[\mu]=\frac{1}{m-1}\int_{\bbs^\dm} (1-\alpha)^m \rho(x)^m\dx-\frac{\kappa}{2} \bigl\|(1-\alpha)x_0+ \alpha \|c_\rho\| z \bigr\|^2+\frac{\kappa}{2}.
\]
Finally, note that this energy is minimized if and only if $z=x_0$, since both $1-\alpha$ and $\alpha \|c_\alpha\|$ are positive. This concludes the proof.
\end{proof}

Based on Lemmas \ref{Lemma:Dirac-delta-concentration}-\ref{lem:symmetry-steady-states}, we consider the following admissible set $\Psym(\bbs^\dm) \subset \mathcal{P}(\bbs^\dm)$ for the energy minimizers:
\begin{equation}
\label{eqn:Psym}
\Psym(\bbs^\dm)=\left\{\alpha\delta_{x_0}+(1-\alpha)\rho\dx:~0\leq\alpha < 1,~\rho\in\mathcal{P}_{ac}(\bbs^\dm),~ c_\rho=\|c_\rho\|x_0\right\},
\end{equation}
where $x_0$ is a fixed point on $\bbs^\dm$. Note that when $\alpha=0$, $\mu=\rho\dS$ and we can assume that $\rho$ is symmetric with respect to $x_0$ without loss of generality. Indeed, by Lemmas \ref{Lemma:Dirac-delta-concentration}-\ref{lem:symmetry-steady-states}, it holds that
\[
\min_{\mu\in\Psym(\bbs^\dm)}E[\mu]=\min_{\mu\in\mathcal{P}(\bbs^\dm)}E[\mu].
\]


\subsection{Critical points of the energy}
\label{subsect:cp}

We now investigate the equilibria that are in the set $\Psym(\bbs^\dm)$ of admissible minimizers. We will distinguish between equilibria that are absolutely continuous with respect to $\d S$, and equilibria that contain a singular part. The interaction strength $\kappa>0$ is arbitrary, but fixed. In this section we do not yet indicate, however, the dependence of the equilibria on $\kappa$; hence, the equilibria computed here depend on $\kappa$, unless otherwise noted.
\medskip

{\em a. Absolutely continuous equilibria.} Consider first an equilibrium $\mu = \rho \dS$ which has no singular component. The Euler--Lagrange equation for the functional \eqref{eqn:energy-s} is given by
\begin{equation}
\label{eqn:equil}
\left(\frac{m}{m-1}\right)\rho(x)^{m-1}-\kappa \langle c_\rho, x\rangle = \lambda, \qquad x\in \mathrm{supp}(\rho),
\end{equation}
for some constant $\lambda$. Here, $\lambda$ has the role of a Lagrange multiplier associated with the unit mass constraint.

The equilibria satisfying \eqref{eqn:equil} are fully supported on the entire sphere. Indeed, since the exponent of $\rho$ in \eqref{eqn:equil} is negative, if the support was strictly contained in $\bbs^\dm$, then $\lambda$ would be infinite. From \eqref{eqn:equil} and $0<m<1$, it must necessarily hold that
\begin{equation}
\label{eqn:ml1:cond}
\lambda + \kappa \langle c_\rho, x \rangle \leq 0, \qquad \forall x \in \bbs^\dm.
\end{equation}

Recall that $c_\rho = \|c_\rho\| x_0$, with $x_0 \in \bbs^\dm$ fixed. Define $\theta_x = \mathrm{arccos}\, \langle x_0, x \rangle \in [0,\pi]$, to write the equilibria as
\begin{equation}
\label{eqn:equil-fs}
\rho(x) = \left(\frac{m}{1-m}\right)^{\frac{1}{1-m}} \left(-\lambda- \kappa\|c_\rho\|\cos\theta_x\right)^{\frac{1}{m-1}}, \qquad\forall~ x\in \bbs^\dm.
\end{equation}
Also, since $-\lambda- \kappa\|c_\rho\|\cos\theta_x \geq 0$, for all $\theta_x \in [0,\pi]$, we have  $\lambda \leq -\kappa \|c_\rho\|$. If the inequality is strict, i.e., $\lambda < -\kappa \|c_\rho\|$, then  $-\lambda- \kappa\|c_\rho\|\cos\theta_x >0$ for all $x \in \bbs^\dm$, and hence $\rho \in L^\infty(\bbs^\dm)$. However, $\lambda = -\kappa \|c_\rho\|$ is a limiting case, where $\rho$ becomes infinite at $\theta=0$; see part b, in particular Remark \ref{rmk:rho-integr}.

To find $\lambda$ and $\|c_\rho\|$, and hence determine the equilibrium \eqref{eqn:equil-fs}, we ask that $\rho$ has mass $1$ and centre of mass at $c_\rho$. Using hyperspherical coordinates, we arrive at the following system of equations:
\begin{equation}
\label{eqn:system-fs-ml1}
\begin{aligned}
1&=\dm w_\dm  \int_0^\pi \left(\frac{m}{1-m}\right)^{\frac{1}{1-m}} \left(-\lambda- \kappa\|c_\rho\|\cos\theta \right)^{\frac{1}{m-1}}\sin^{\dm-1}\theta \, \d\theta, \\
\|c_\rho\|&=\dm w_\dm\int_0^\pi \left(\frac{m}{1-m}\right)^{\frac{1}{1-m}} \left(-\lambda- \kappa\|c_\rho\|\cos\theta \right)^{\frac{1}{m-1}} \sin^{\dm-1}\theta \cos\theta \, \d\theta,
\end{aligned}
\end{equation}
where $w_{\dm}$ denotes the volume of the $\dm$-dimensional unit ball. In the integration above we used $|\bbs^{\dm-1}|=\dm w_{\dm}$, where $|\bbs^{\dm-1}|$ denotes the area of the $(\dm-1)$-dimensional unit sphere. System \eqref{eqn:system-fs-ml1} needs to be solved for $\lambda$ and $\|c_\rho\|$, for a given $\kappa$.
\medskip

{\em b. Equilibria that have a singular component.} By \eqref{eqn:Psym}, we consider probability measures of the form
\begin{equation}
\label{eqn:mu-alpha}
\mu_\alpha = \alpha \delta_{x_0} + (1-\alpha) \rho \, \d S,
\end{equation}
where $\alpha \in (0,1)$ and $x_0$ is the normalized centre of mass of the density $\rho \in \calP_{ac}(\bbs^\dm)$:
\begin{equation}
\label{eqn:x0}
x_0 = \frac{c_\rho}{\| c_\rho\|} = \frac{\int_{\bbs^\dm} x \rho(x) \d S}{\left \| \int_{\bbs^\dm} x \rho(x)  \d S\right \|}.
\end{equation}

Similar to how \eqref{eqn:energy-s} was derived, we can write
\begin{equation}
\label{eqn:energy-mu-s}
E[\mu_\alpha] = \frac{1}{m-1} \int_{\bbs^\dm} (1-\alpha)^m \rho(x)^m \dx - \frac{\kappa}{2} \| c_{\mu_\alpha} \|^2 + \frac{\kappa}{2},
\end{equation}
where the centre of mass of $\mu_\alpha$ is given by 
\begin{equation}
\label{eqn:cm-mu}
\begin{aligned}
c_{\mu_\alpha} &= \int_{\bbs^\dm} x \d \mu_\alpha(x)= \alpha x_0  + (1-\alpha) c_\rho.
\end{aligned}
\end{equation}

Also, using \eqref{eqn:x0}, we can further write
\begin{equation}
\label{eqn:energy-mu-s2}
\begin{aligned}
E[\mu_\alpha] =& \frac{1}{m-1} \int_{\bbs^\dm} (1-\alpha)^m \rho(x)^m \dx \\
& \quad - \frac{\kappa}{2} \left( \alpha^2 + 2 \alpha (1-\alpha) \| c_\rho \| + (1-\alpha)^2 \| c_\rho \|^2 \right) + \frac{\kappa}{2}.
\end{aligned}
\end{equation}

In computing the Euler-Lagrange equations for the energy $E[\mu_\alpha]$, we consider variations in both the parameter $\alpha$ and density $\rho$. First, by setting $\displaystyle \frac{\partial}{\partial \alpha} E[\mu_\alpha]=0$,
we find 
\begin{align*}
\frac{m(1-\alpha)^{m-1}}{1-m} \int_{\bbs^\dm}  \rho(x)^m \dx -\kappa \left( \alpha + (1 -2 \alpha) \| c_\rho \| + (\alpha-1) \|c_\rho \|^2 \right) = 0.
\end{align*}
After factorizing the expression in the round brackets in the l.h.s. above, we arrive at the following Euler-Lagrange equation (by perturbing with respect to $\alpha$ only):
\begin{equation}
\label{eqn:EL-alpha}
\frac{m(1-\alpha)^{m-1}}{1-m} \int_{\bbs^\dm}  \rho(x)^m \dx - \kappa \left ( 1- \| c_\rho \|\right) \left( \alpha + (1-\alpha) \|c_\rho \|)\right) =0.
\end{equation}

Now consider perturbations with respect to $\rho$. Note again that perturbations of $\rho$ must have mass $1$, which introduces a Lagrange multiplier. 

Omitting the calculation, we arrive at the following Euler--Lagrange equation (by perturbating with respect to $\rho$):
\begin{equation}
\label{eqn:EL-rho}
\frac{m}{m-1} (1-\alpha)^m \rho(x)^{m-1} - \kappa(1-\alpha) \langle c_\rho, x \rangle \left( \frac{\alpha}{\| c_\rho \|} + 1-\alpha \right) = \lambda, \qquad x \in \text{supp}(\rho),
\end{equation}
for some constant $\lambda$.

The support of $\rho$ is the entire sphere; if we assume otherwise, $\lambda$ would be infinite. From \eqref{eqn:EL-rho}, by writing in hyperspherical coordinates, we have
\begin{equation}
\label{eqn:equil-fsmu}
\rho(x) = \left( \frac{m(1-\alpha)^m}{1-m} \right)^{\frac{1}{1-m}} \bigl (-\lambda - \kappa(1-\alpha) \cos \theta_x \left (\alpha + (1 -\alpha)\| c_\rho \|  \right) \bigr)^{\frac{1}{m-1}}, \qquad \forall x \in \bbs^\dm.
\end{equation}
Note that for $\alpha=0$, \eqref{eqn:equil-fsmu} reduces to \eqref{eqn:equil-fs}.

An equilibrium $\mu_\alpha$ in the form \eqref{eqn:mu-alpha} must satisfy both \eqref{eqn:EL-alpha} and \eqref{eqn:EL-rho}. Multiply \eqref{eqn:EL-rho} by $\frac{\rho(x)}{\alpha-1}$ and integrate over $x \in \bbs^\dm$ to get
\begin{equation}
\label{eqn:EL-int}
\frac{m(1-\alpha)^{m-1}}{1-m} \int_{\bbs^\dm}\rho(x)^m \dx + \kappa  \|c_\rho \|\left( \alpha + (1-\alpha) \|c_\rho \| \right) = \frac{\lambda}{\alpha-1}.
\end{equation}
Now subtract \eqref{eqn:EL-int} from \eqref{eqn:EL-alpha} to find the following relationship between $\lambda, \| c_\rho \|$ and $\alpha$:
\begin{equation}
\label{eqn:lambda-kappa}  
-\frac{\lambda}{1-\alpha} = \kappa( \alpha + (1 -\alpha)\| c_\rho \|).
\end{equation}
We remark here that $\alpha=0$ corresponds to the limiting case $\lambda = - \kappa \| c_\rho \| $ pointed out in part a. 

Finally, by using \eqref{eqn:lambda-kappa}, we write \eqref{eqn:equil-fsmu} as
\begin{equation}
\label{eqn:equil-fsmu-s}  
\rho(x) = \left( \frac{m(1-\alpha)^m}{1-m} \right)^{\frac{1}{1-m}} (-\lambda)^{\frac{1}{m-1}} \bigl( 1 -  \cos \theta_x \bigr)^{\frac{1}{m-1}}, \qquad \forall x \in \bbs^\dm.
\end{equation}
Also, we find the compatibility conditions for the mass and centre of mass:
\begin{equation}
\label{eqn:system-mu}
\begin{aligned}
1&=\dm w_\dm  \int_0^\pi \left( \frac{m(1-\alpha)^m}{1-m} \right)^{\frac{1}{1-m}} (-\lambda)^{\frac{1}{m-1}} \bigl (1 -  \cos \theta \bigr)^{\frac{1}{m-1}} \sin^{\dm-1}\theta \, \d\theta, \\[5pt]
\|c_\rho\|&=\dm w_\dm\int_0^\pi \left( \frac{m(1-\alpha)^m}{1-m} \right)^{\frac{1}{1-m}} (-\lambda)^{\frac{1}{m-1}} \bigl (1 -  \cos \theta \bigr)^{\frac{1}{m-1}}  \sin^{\dm-1}\theta \cos\theta \, \d\theta.
\end{aligned}
\end{equation}
The constants $\lambda, \|c_\rho\|$ and $\alpha$ are then found by solving \eqref{eqn:lambda-kappa} and \eqref{eqn:system-mu}.

First note that by dividing the two equations in \eqref{eqn:system-mu}, we find
\begin{equation}
\label{eqn:crho-kgk2}
\|c_\rho\| =\frac{\int_0^\pi (1-\cos\theta)^{\frac{1}{m-1}}\sin^{\dm-1}\theta \cos\theta\d\theta}{\int_0^\pi (1-\cos\theta)^{\frac{1}{m-1}}\sin^{\dm-1}\theta\d\theta},
\end{equation}
so the centre of mass of $\rho$ does not depend on $\kappa$. Furthermore, by the first equation in \eqref{eqn:system-mu}, we have
\begin{equation}
\label{eqn:mlambda}
(-\lambda)^{\frac{1}{m-1}} = \frac{1}{\dm w_\dm} \left( \frac{m(1-\alpha)^m}{1-m} \right)^{\frac{1}{m-1}} \frac{1}{\int_0^\pi\bigl (1 -  \cos \theta \bigr)^{\frac{1}{m-1}} \sin^{\dm-1}\theta \, \d\theta},
\end{equation}
which upon substitution in \eqref{eqn:equil-fsmu-s} yields
\begin{equation}
\label{eqn:equil-rho-reg}
\rho(x) = \frac{\bigl( 1 -  \cos \theta_x \bigr)^{\frac{1}{m-1}}}{\dm w_\dm \int_0^\pi\bigl (1 -  \cos \theta \bigr)^{\frac{1}{m-1}} \sin^{\dm-1}\theta \d \theta} , \qquad \forall x \in \bbs^\dm.
\end{equation}
Hence, we conclude that the density $\rho$ of the regular component of an equilibrium in the form \eqref{eqn:mu-alpha} is {\em fixed}, in the sense that it does not depend on $\kappa$. The only variable that needs to be determined is $\alpha$, which can be done by solving \eqref{eqn:lambda-kappa} and \eqref{eqn:mlambda} for $\alpha$ and $\lambda$, with $\|c_\rho\|$ given by \eqref{eqn:crho-kgk2}.
\begin{remark}
\label{rmk:rho-integr} 
The density $\rho$ from \eqref{eqn:equil-rho-reg} is not in $L^\infty(\bbs^\dm)$, as it blows up at $\theta_x=0$. We need however that $\rho \in L^1(\bbs^\dm)$ ($\rho$ is unit mass in fact). By hyperspherical coordinates (see first equation in \eqref{eqn:system-mu}), this reduces to checking whether $(1-\cos \theta)^{\frac{1}{m-1}} \sin^{\dm-1} \theta$ has an integrable singularity at $\theta =0$. As $1-\cos \theta$ behaves as $\theta^2/2$ near $\theta=0$, we need to check the integrability of $\theta^{\frac{2}{m-1}+\dm-1}$. We find that $\rho$ is integrable when $0<m<1-\frac{2}{\dm}$, and not integrable for $1-\frac{2}{\dm} < m <1$. Also note that for $\dm =1$ and $\dm=2$, $\rho$ is not integrable for any $0 < m <1$. We infer from these observations that equilibria in the form \eqref{eqn:mu-alpha} (i.e., with a singular component) do not occur for $\dm =1$ and $\dm=2$, and only occur when $0<m<1-\frac{2}{\dm}$ for $\dm \geq 3$.
\end{remark}

\begin{remark}
\label{rmk:scaling-blowup} 
The critical value $m=1-\frac{2}{\dm}$ also appears in distinguishing different behaviours of solutions to the nonlinear diffusion equation $\rho_t = \Delta \rho^m$ in the Euclidean space $\bbr^\dm$. For $0<m<1-\frac{2}{\dm}$, it was shown that the Dirac delta remains stable in a certain sense, for all times \cite{BrezisFriedman1983}. In informal terms, diffusion does not spread out very localized initial densities when $0<m<1-\frac{2}{\dm}$; we also note that this result applies only in dimensions $\dm\geq 3$. On the other hand, also shown in \cite{BrezisFriedman1983}, diffusion does spread out a Dirac delta when $1-\frac{2}{\dm}<m<1$. Even though we work here on the sphere, as opposed to the flat Euclidean space, these consideration are very local in space, and the effects or curvature do not seem important. Hence, as noted in Remark \ref{rmk:rho-integr}, equilibria of \eqref{energy-sphere} that contain a Dirac delta can only occur for $0<m<1-\frac{2}{\dm}$ and $\dm \geq 3$, which is consistent with the case when the Dirac delta is ``stable" for the nonlinear diffusion evolution equation.
\end{remark}


\subsection{Uniform distribution}
\label{subsect:unif-distr}

The uniform distribution on the sphere, given by
\begin{equation}
\label{eqn:rho-uni}
\rhou(x) = \frac{1}{|\bbs^\dm|}, \qquad \forall x \in \bbs^\dm,
\end{equation}
where $|\bbs^\dm|$ denotes the area of $\bbs^\dm$, is a solution of \eqref{eqn:equil} for all $\kappa>0$. Indeed, in this case, $c_{\rhou} =0$, and the constant $\lambda$ is given by
\begin{equation}
\label{eqn:lambda-uni}
\lambda_{\text{uni}} = \left(\frac{m}{m-1}\right) \frac{1}{|\bbs^\dm|^{m-1}}.
\end{equation}
In applications to polymer orientation, this equilibrium corresponds to the isotropic state. 

The full nonlinear stability of the uniform distribution can be established, in terms of the strength $\kappa$ of the attractive interactions. The result is the following. 
\begin{proposition}[Stability of the uniform distribution]
\label{prop:stab-unif}
The uniform distribution $\rhou$ is a stable critical point of the energy functional $E$ if and only if $\kappa \leq \kappa_1$, where
\begin{equation}
\label{eqn:kappa1}
\kappa_1:= \frac{m (\dm+1)}{|\bbs^\dm|^{m-1}}.
\end{equation}
\end{proposition}
\begin{proof}
We consider two types of perturbations of the uniform distribution: i) absolutely continuous perturbations, and ii) perturbations that contain a singular part.

{\em Case i) Absolutely continuous perturbations.} First consider perturbations $\{ \rho^\epsilon \} \subset \calP_{ac}(\bbs^\dm) $ of the uniform distribution, that are absolutely continuous:
\begin{equation}
\label{rho-eps-cond}
\rho^0 = \rhou, \qquad \int_{\bbs^\dm} \rhoeps(x) \dx  = 1, \quad \text{ for all } \epsilon.
\end{equation}
The energy of $\rho^\epsilon$ is given by (use \eqref{eqn:energy-cmform}):
\[
E[\rhoeps]= \frac{1}{m-1}\int_{\bbs^\dm}\rhoeps(x)^m \dx-\frac{\kappa}{2}\iint_{\bbs^{\dm}\times \bbs^{\dm}}\langle x, y\rangle \rhoeps(x)\rhoeps(y)\dS(x)\dS(y)+\frac{\kappa}{2}.
\]
The uniform distribution is a local minimizer with respect to absolutely continuous perturbations provided 
\begin{equation}
\label{eqn:localmin}
\frac{\d^2}{\d \epsilon^2} E[\rhoeps]_{\big|_{\epsilon = 0}}\geq 0,
\end{equation}
for all perturbations $\rho^\epsilon$ that satisfy \eqref{rho-eps-cond}.

The stability calculations in this case are identical to those carried out in \cite{FePaVa2025} for slow diffusion, and we will only summarize them here. Denote 
\[
\delta \rhoeps = \frac{\d }{\d \epsilon} \rhoeps, \qquad \text{ and } \qquad \delta^2 \rhoeps = \frac{\d^2 }{\d \epsilon^2} \rhoeps.
\]
Using the unit mass condition in \eqref{rho-eps-cond}, we note that 
\begin{equation}
\label{eqn:zero-int}
\int_{\bbs^\dm} \delta \rhoeps(x) \dx  = 0, \qquad \text{ and } \qquad \int_{\bbs^\dm} \delta^2 \rhoeps(x) \dx  = 0, \qquad \textrm{ for all } \epsilon.
\end{equation}

By direct calculations, one finds
\begin{equation}
\label{eqn:d1Eeps}
\frac{\d}{\d \epsilon} E[\rhoeps] 
= \frac{1}{m-1}\int_{\bbs^\dm} m \rhoeps(x)^{m-1} \delta \rhoeps(x) \dx -\kappa \iint_{\bbs^\dm\times \bbs^\dm} \langle x,y \rangle \, \rhoeps(x) \delta \rhoeps(y) \dx \dy,
\end{equation}
and
\begin{equation}
\label{eqn:d2Eeps}
\begin{aligned}
\frac{\d^2}{\d \epsilon^2} E[\rhoeps] &= m \int_{\bbs^\dm} \rhoeps(x)^{m-2} \left( \delta \rhoeps(x)\right)^2 \dx + \frac{m}{m-1}\int_{\bbs^\dm} \rhoeps(x)^{m-1} \delta^2 \rhoeps(x) \dx \\[5pt]
& \quad -\kappa \iint_{\bbs^\dm\times \bbs^\dm} \langle x,y \rangle (\delta \rhoeps(x) \delta \rhoeps(y) + \rhoeps(x) \delta^2 \rhoeps(y) ) \dx \dy.
\end{aligned}
\end{equation}

Evaluate \eqref{eqn:d1Eeps} and \eqref{eqn:d2Eeps} at $\epsilon = 0$.  Using that $\rho^0=\rhou$ is constant, that $\delta \rhoeps(x)$ integrates to $0$, and also that the centre of mass of $\rho^0$ is at the origin, we find $\displaystyle \frac{\d}{\d \epsilon} E[\rhoeps]_{\big| {\epsilon = 0}} = 0$ -- as expected, since $\rhou$ is a critical point of the energy. Similarly, using that $\rho^0=\rhou$ is constant, that $\delta^2 \rhoeps(x)$ integrates to $0$, and that the centre of mass of $\rho^0$ is at the origin, we find that two of the terms in the r.h.s. of \eqref{eqn:d2Eeps} vanish:
\[
\int_{\bbs^\dm} \rho^0(x)^{m-1} \delta^2 \rho^0(x) \dx = \frac{1}{|\bbs^\dm |^{m-1}}  \int_{\bbs^\dm} \delta^2 \rho^0(x) \dx = 0.
\]
and
\begin{align*}
\iint_{\bbs^\dm\times \bbs^\dm} \langle x,y \rangle \rho^0(x) \delta^2 \rho^0(y) \dx \dy &=\bigg \langle \int_{\bbs^\dm} x \rho^0(x) \dx, \int_{\bbs^\dm} y \, \delta^2 \rho^0(x) \dy \bigg \rangle \\
&=0.
\end{align*}

Therefore, by dropping the superscript $0$ and by denoting $\delta \rho = \delta \rhoeps_{\, | \epsilon = 0}$, we find from \eqref{eqn:d2Eeps}:
\begin{equation}
\label{eqn:d2Eeps0}
\begin{aligned}
\frac{\d^2}{\d \epsilon^2} E[\rhoeps]_{\big| {\epsilon = 0}}&= \frac{m}{|\bbs^\dm|^{m-2}} \int_{\bbs^\dm}  \left( \delta \rho(x) \right)^2 \dx -\kappa \left \| \int_{\bbs^\dm} x \delta \rho(x) \dx \right \|^2.
\end{aligned}
\end{equation}
For \eqref{eqn:localmin} to hold, the r.h.s. above has to be non-negative, i.e., 
\[
\kappa \leq \frac{m}{|\bbs^\dm|^{m-2}} \cdot \frac{\int_{\bbs^\dm} \left( \delta \rho(x) \right)^2 \dx}{\left | \int_{\bbs^\dm} x \delta \rho(x) \dx \right|^2},
\]
for all perturbations $\delta \rho$ that have zero mass.

Define $\calA = \{ \psi \in L^2(\bbs^\dm): \int_{\bbs^\dm} \psi(x) \dx = 0\}$ and the following functional on $\calA$:
\begin{equation}
\label{eqn:calF}
\calF [\psi] = \frac{\int_{\bbs^\dm} \left( \psi(x) \right)^2 \dx}{\left | \int_{\bbs^\dm} x \psi(x) \dx \right|^2}.
\end{equation}
The uniform distribution is a local minimum (and hence, stable) with respect to absolutely continuous perturbations, provided
\begin{equation}
\label{eqn:cond-kappa}
\kappa \leq \frac{m}{|\bbs^\dm|^{m-2}} \cdot \inf_{\psi \in \calA} \calF[\psi],
\end{equation}
and unstable otherwise. 

The following lemma is proved in \cite{FePaVa2025}, and we simply list it here.
\begin{lemma}[Lemma 3.1, \cite{FePaVa2025}]
\label{lem:minF}
Let $\mathcal{F}: \calA \to \bbr$ be given by \eqref{eqn:calF}. Then,
\[
\inf_{\psi \in \calA} \calF[\psi] =\frac{\dm+1}{|\bbs^\dm|}.
\]  
\end{lemma}

By Lemma \ref{lem:minF}, \eqref{eqn:cond-kappa} holds only for $\kappa \leq \kappa_1$, and this shows the claimed stability with respect to absolutely continuous perturbations.
\medskip

{\em Case ii) Singular perturbations.} Consider a singular perturbation $\{\mu^\epsilon\}_{\epsilon>0} \subset \calP(\bbs^\dm)$ of the form
\[
\mu^\epsilon = \alpha^\epsilon \nu + (1-\alpha^\epsilon) \rhou \d S,
\]
where $\nu=\nu_s+\tilde{\rho}\d S\in \mathcal{P}(\bbs^\dm)$, with $\nu_s$ and $\tilde{\rho}\d S$ being the singular and the absolutely continuous parts of $\nu$, respectively. Also, $\{\alpha^\epsilon\}_{\epsilon>0}$ satisfy
\[
\alpha^0 = 0, \qquad \text{ and } \qquad \delta \alpha := \frac{\d }{\d \epsilon} \alpha^\epsilon_{\; \big|_{\epsilon = 0+}}> 0.
\]
Note that these are one-sided perturbations with respect to $\alpha$, as $\alpha=0$ is a boundary value of the interval $(0,1)$.

Then, by \eqref{eqn:energy-s} we have
\[
E[\mu^\epsilon]=\frac{1}{m-1}\int_{\bbs^{\dm}}\left(\alpha^\epsilon \tilde{\rho}(x)+(1-\alpha^\epsilon)\rho_{\text{uni}}(x)\right)^m \d S-\frac{\kappa}{2}(\alpha^\epsilon\|c_{\nu}\|)^2+\frac{\kappa}{2},
\]
and its derivative at $\epsilon=0+$ is
\[
\frac{\d}{\d \epsilon} E[\mu^\epsilon]_{\big|_{\epsilon = 0+}} = \frac{m}{1-m} \delta \alpha \int_{\bbs^\dm} \rho^{m-1}_{\text{uni}}(x)\left(\rho_{\text{uni}}(x)-\tilde{\rho}(x)\right) \dS.
\]
Since $\rho_{\text{uni}}$ is constant, we can simplify the above as
\[
\frac{\d}{\d \epsilon} E[\mu^\epsilon]_{\big|_{\epsilon = 0+}} = \frac{m(\delta \alpha)\rho_{\text{uni}}^{m-1}}{1-m} \int_{\bbs^\dm} \left(\rho_{\text{uni}}(x)-\tilde{\rho}(x)\right) \dS= \frac{m(\delta \alpha)\rho_{\text{uni}}^{m-1} \int_{\bbs^\dm}\d \nu_s }{1-m},
\]
where we used $\rhou, \nu\in\mathcal{P}(\bbs^\dm)$. 

As we considered singular perturbations $\nu$, we have $\int_{\bbs^\dm}\d \nu_s  >0$. Therefore, 
\begin{equation}
\label{eqn:dEmueps}
\frac{\d}{\d \epsilon} E[\mu^\epsilon]_{\big|_{\epsilon = 0+}}>0,
\end{equation}
and we conclude that the energy of the uniform distribution increases with the formation of a singular measure.

More generally, if we consider perturbations of the form
\begin{equation}
\label{eqn:mueps-gen}
\mu^\epsilon = \alpha^\epsilon \nu + (1-\alpha^\epsilon) \rho^\epsilon \d S,
\end{equation}
where
\[
\alpha^0 = 0, \quad \rho^0 = \rhou, \quad \delta \alpha := \frac{\d }{\d \epsilon} \alpha^\epsilon_{\; \big|_{\epsilon = 0+}}> 0, \quad \text{ and } \quad \int_{\bbs^\dm} \rhoeps(x) \dx  = 1,
\]
then the considerations follow the previous calculations. Indeed, by taking a derivative of $E[\mu^\epsilon]$ with respect to $\epsilon$, and evaluate at $\epsilon=0+$, we find again \eqref{eqn:dEmueps}, as the contributions from the variations $\rhoeps$ lead to $0$.

We conclude that the energy increases by perturbing $\rhou$ with a singular measure. Note that this fact applies to all values of $\kappa$. Hence, the stability of $\rhou$ for different ranges of $\kappa$ is as found using absolutely continuous perturbations.
\end{proof} 

We will find in the next section that for $\kappa<\kappa_1$, the uniform distribution is the only critical point of the energy (and hence, a global minimizer), and at $\kappa=\kappa_1$ a bifurcation occurs.


\section{Critical values of $\kappa$ and phase transitions}
\label{sect:bifurcations-ml1}
In this section, we will identify various critical values of $\kappa$ at which equilibria experience phase transitions, for several ranges of $m$ within $0<m<1$.

\subsection{Bifurcation from the uniform distribution} 
\label{subsect:fs}
 We first study the bifurcation that occurs at $\kappa = \kappa_1$ (cf., Proposition \ref{prop:stab-unif}). We consider the fully supported equilibrium \eqref{eqn:equil-fs}, with $\lambda$ and $\|c_\rho\|$ that solve \eqref{eqn:system-fs-ml1}, where in this case, $-\lambda \geq \kappa \|c_\rho\|$.

For a simpler notation, we denote
\begin{equation}
\label{eqn:slambda}
  s=\|c_\rho\|, \qquad \lambda=-\kappa s\eta.
\end{equation}
To guarantee the condition on $\lambda$, we assume $\eta\geq 1$. The compatibility conditions \eqref{eqn:system-fs-ml1} for mass and centre of mass become
\begin{align*}
1&=\dm w_\dm \left(\frac{m}{(1-m)\kappa s}\right)^{\frac{1}{1-m}}  \int_0^\pi \left(\eta - \cos\theta \right)^{\frac{1}{m-1}}\sin^{\dm-1}\theta \, \d\theta, \\[5pt]
s&=\dm w_\dm \left(\frac{m}{(1-m) \kappa s}\right)^{\frac{1}{1-m}}  \int_0^\pi \left( \eta- \cos\theta \right)^{\frac{1}{m-1}} \sin^{\dm-1}\theta \cos\theta \, \d\theta.
\end{align*}
From the above, we find $s$ and $\kappa$ in terms of $\eta$ as
\begin{subequations}
\label{eqn:skappa-eta}
\begin{align}
 s&=\frac{\int_0^\pi (\eta-\cos\theta)^{\frac{1}{m-1}}\sin^{\dm-1}\theta \cos\theta\d\theta}{\int_0^\pi (\eta-\cos\theta)^{\frac{1}{m-1}}\sin^{\dm-1}\theta\d\theta},\label{eqn:s-eta} \\[2pt]
 \kappa^{-1}&=\displaystyle \frac{1-m}{m} (\dm w_\dm)^{m-1} \left(\int_0^\pi (\eta-\cos\theta)^{\frac{1}{m-1}}\sin^{\dm-1}\theta \cos\theta\d\theta\right)
\left(\int_0^\pi (\eta-\cos\theta)^{\frac{1}{m-1}}\sin^{\dm-1}\theta\d\theta\right)^{m-2}. \label{eqn:kappa-eta}
\end{align}
\end{subequations}

We define the following function in the range $\eta \geq 1$: 
\begin{equation}
\label{eqn:H}
H(\eta)=  \frac{1-m}{m} (\dm w_\dm)^{m-1} \left(\int_0^\pi (\eta-\cos\theta)^{\frac{1}{m-1}}\sin^{\dm-1}\theta \cos\theta\d\theta\right)
\left(\int_0^\pi (\eta-\cos\theta)^{\frac{1}{m-1}}\sin^{\dm-1}\theta \d\theta\right)^{m-2}.
\end{equation}
Equation \eqref{eqn:kappa-eta} can then be written as
\[
H(\eta) = \kappa^{-1}.
\]

\begin{lemma}
\label{lem:H-monotone} 
The function $H(\eta)$ defined in \eqref{eqn:H} for $\eta \geq 1$, is strictly increasing for $1 - \frac{2}{d - 1} < m < 1$, strictly decreasing for $0 < m < 1 - \frac{2}{d - 1}$, and constant when $m = 1 - \frac{2}{d - 1}$.
\end{lemma}
\begin{proof}
The proof is quite lengthy and nontrivial. It involves investigating the sign of $H'(\eta)$ using associated Legendre functions of the first kind and the Chebyshev inequality; see Appendix \ref{appendix:H-monotone}.
\end{proof}

\begin{remark}
\label{rmk:d12}
Note that $\dm=1$ and $\dm=2$ in Lemma \ref{lem:H-monotone} are special, in the sense that for these cases, $H(\eta)$ is strictly decreasing for all $0<m<1$.
\end{remark}

\begin{lemma}
\label{lem:Hlim}
The following holds:
\[
\Bigl( \lim_{\eta \to \infty } H(\eta)\Bigr)^{-1} = \kappa_1,
\]
where $\kappa_1$ is the critical $\kappa$ identified in Proposition \ref{prop:stab-unif}. 
\end{lemma}
\begin{proof}
The proof follows a direct calculation -- see Appendix \ref{appendix:Hlim}.  
\end{proof}

\begin{remark}
\label{rmk:kappac-ml1}
In the limit $\eta\to \infty$ (along with $s \to 0$ and $\lambda \to \lambda_{\text{uni}}$), the equilibria approach the uniform distribution $\rho_{\text{uni}}$. As discussed below (Remark \ref{rmk:kappact}, Proposition \ref{prop:bif}), there is a bifurcation at $\kappa = \kappa_{1}$, where a fully supported equilibrium $\rho_\kappa$ gets born from the uniform distribution $\rhou$.
\end{remark}

The function $H(\eta)$ is plotted in Figure \ref{fig:H}; for plot (a) we used  $m=0.5$, $\dm=2$, for (b) we have $m=0.25$, $\dm=3$, and for plot (c) we used $m=0.3$, $\dm=5$. For both cases (a) and (b), $H(\eta)$ is increasing, but the cases are different in the behaviour $\eta \searrow 1$ -- see discussion below. Case (c) corresponds to $H(\eta)$ being decreasing; see Lemma \ref{lem:H-monotone}. 

{\em Limit $\eta \searrow 1$.} We investigate the convergence of the integrals that appear in the function $H$ -- note the singularity of $(1-\cos \theta)^{\frac{1}{m-1}}$ at $\theta =0.$ Near $\theta =0$, $(1-\cos\theta)^{\frac{1}{m-1}}\sin^{\dm-1} \theta$, as well as $(1-\cos\theta)^{\frac{1}{m-1}}\sin^{\dm-1} \theta \cos \theta$, behave as $\theta^{\frac{2}{m-1}+\dm-1}$. The integrals are convergent provided $\frac{2}{m-1}+\dm-1 >-1$, which is equivalent to $m< 1- \frac{2}{\dm}$; see also Remark \ref{rmk:rho-integr}.

Based on Lemma \ref{lem:H-monotone} and the limit $\eta \searrow 1$, we distinguish three cases: i) $1- \frac{2}{\dm} < m <1$, ii) $1- \frac{2}{\dm-1}<m< 1- \frac{2}{\dm}$, and iii) $0<m<1- \frac{2}{\dm-1}$. Note that in dimensions $\dm=1$ and $\dm=2$, only case i) applies (see also Remark \ref{rmk:d12}), while in dimension $\dm=3$, only cases i) and ii) apply.

\begin{figure}[!htbp]
 \begin{center}
\includegraphics[width=0.48\textwidth]{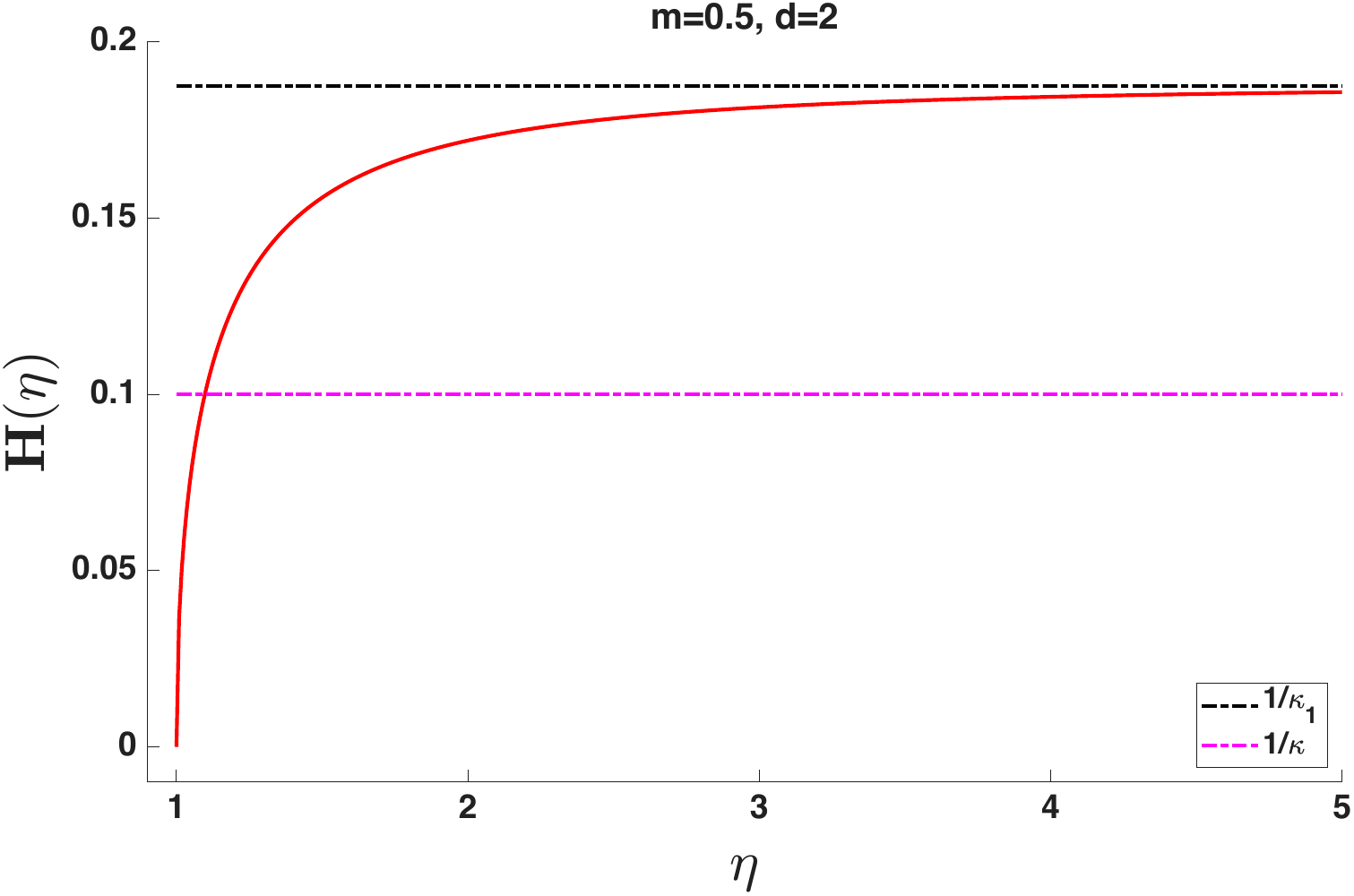} \\
  (a)
 \begin{tabular}{cc} 
 \includegraphics[width=0.48\textwidth]{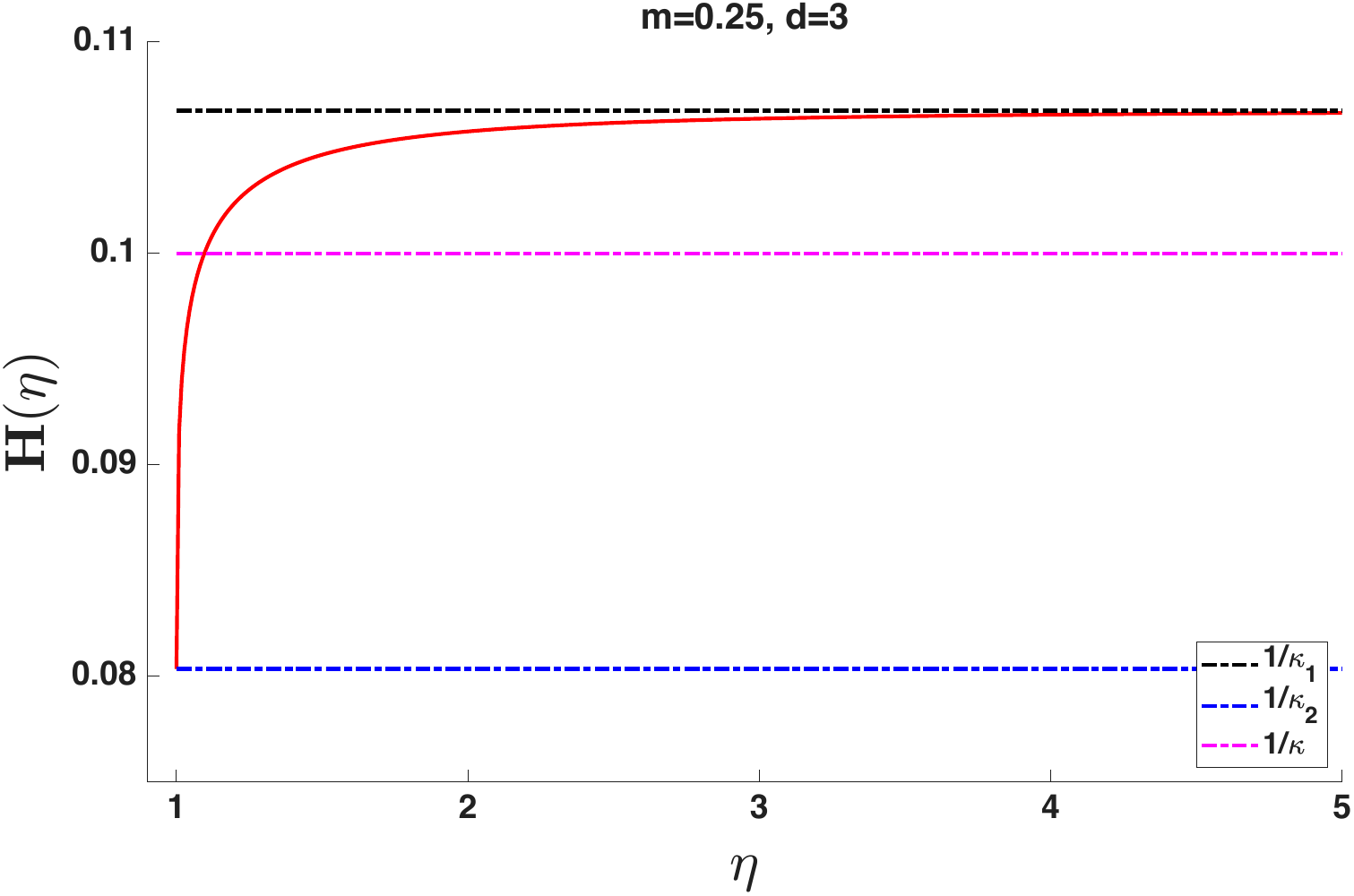} &
\includegraphics[width=0.48\textwidth]{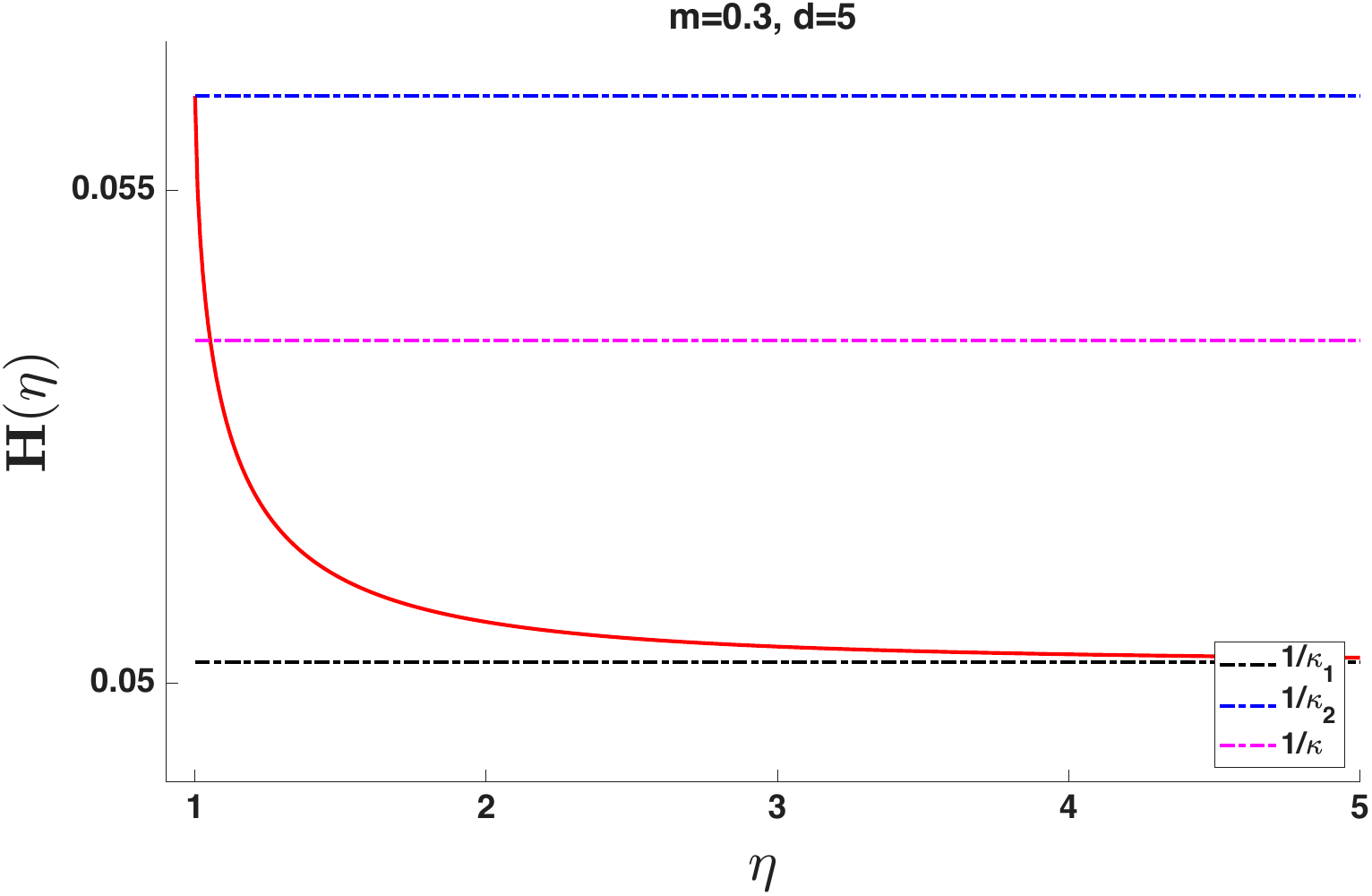}
 \\
 (b) & (c)
\end{tabular}
\caption{Plot of function $H$ defined in \eqref{eqn:H}. (a) $1-2/\dm<m<1$. For this range of $m$, $H(1) = 0$. For any $\kappa>\kappa_1$, there exists a unique $\eta_\kappa >1$ such that $\kappa^{-1} = H(\eta_\kappa)$ -- see equation \eqref{eqn:kappa-eta}. (b) $1-2/(\dm-1)<m<1-2/\dm$. In this case, $H(1) = 1/\kappa_2$, with $\kappa_2>\kappa_1$. For any $\kappa_1<\kappa<\kappa_2$, there exists a unique $\eta_\kappa>1$ that solves $\kappa^{-1} = H(\eta_\kappa)$. (c) $0<m<1-2/(\dm-1)$. In this case, $H(1) = 1/\kappa_2$, with $\kappa_2<\kappa_1$. For any $\kappa_2<\kappa<\kappa_1$, there exists a unique $\eta_\kappa>1$ that solves $\kappa^{-1} = H(\eta_\kappa)$. For plot (a) we used $m=0.5$ and $\dm =2$, for plot (b), $m=0.25$ and $\dm =3$, and for plot (c), $m=0.3$ and $\dm =3$.}
\label{fig:H}
\end{center}
\end{figure}

\medskip
{\em Case i) $1- \frac{2}{\dm} < m <1$. } Write 
\[
H(\eta)=  \frac{1-m}{m} (\dm w_\dm)^{m-1} \frac{\int_0^\pi (\eta-\cos\theta)^{\frac{1}{m-1}}\sin^{\dm-1}\theta \cos\theta\d\theta}
{\left(\int_0^\pi (\eta-\cos\theta)^{\frac{1}{m-1}}\sin^{\dm-1}\theta \d\theta\right)^{2-m}}.
\]
Due to the non-integrability of the integrands, the limit $\eta \searrow 1$ is of the type $\infty/\infty$. Write the above as 
\begin{align*}
H(\eta) =  \frac{1-m}{m} (\dm w_\dm)^{m-1} \frac{\int_0^\pi (\eta-\cos\theta)^{\frac{1}{m-1}}\sin^{\dm-1}\theta \cos\theta\d\theta}
{\int_0^\pi (\eta-\cos\theta)^{\frac{1}{m-1}}\sin^{\dm-1}\theta \d\theta } \times \frac{1}{\left( \int_0^\pi (\eta-\cos\theta)^{\frac{1}{m-1}}\sin^{\dm-1}\theta \d\theta\right)^{1-m}}.
\end{align*}
The quotient of the two integrals in the r.h.s. is bounded in absolute value by $1$, as
\[
\left | \int_0^\pi (\eta-\cos\theta)^{\frac{1}{m-1}}\sin^{\dm-1}\theta \cos\theta\d\theta \right | \leq \int_0^\pi (\eta-\cos\theta)^{\frac{1}{m-1}}\sin^{\dm-1}\theta \d\theta,
\]
and also,
\[
\lim_{\eta \searrow 1} \frac{1}{\left( \int_0^\pi (\eta-\cos\theta)^{\frac{1}{m-1}}\sin^{\dm-1}\theta \d\theta\right)^{1-m}} = 0,
\]
due to the non-integrability of the integrand in the limit.  

Hence, in case i) we have
\[
\lim_{\eta \searrow 1} H(\eta)=0.
\]
Figure \ref{fig:H}(a) corresponds to this case (for dimension $\dm=2$, all $0<m<1$ fall in this regime in fact). Together with Lemmas \ref{lem:H-monotone}
and \ref{lem:Hlim}, we infer that for any $\kappa>\kappa_{1}$, there is a unique $\eta_\kappa >1$ that solves $\kappa^{-1} = H(\eta_\kappa)$ -- see equation \eqref{eqn:kappa-eta}. This value of $\eta_\kappa$ determines uniquely the corresponding $s_\kappa$ from \eqref{eqn:s-eta}, and then set $\lambda_\kappa = -\kappa s_\kappa \eta_\kappa$. Together they determine a unique (fully supported) equilibrium in the form \eqref{eqn:equil-fs}. For the range of $m$ in case i), there is no second critical value of $\kappa$.
\medskip

{\em Case ii) $1- \frac{2}{\dm-1}<m< 1- \frac{2}{\dm}$.} In this case, the integrals that appear in $H(\eta)$ are convergent in the limit $\eta \searrow 1$, and hence, 
\[
\lim_{\eta \searrow 1} H(\eta)= H(1) >0.
\]
Denote by 
\begin{equation}
\label{eqn:kappa2}
\kappa_2: = \frac{1}{H(1)}.
\end{equation}
Since $H(\eta)$ is increasing on $\eta \in (1,\infty)$, we have $\kappa_2 >\kappa_1$ (see Lemma \ref{lem:Hlim}). We also infer that for $ 1- \frac{2}{\dm-1}<m< 1- \frac{2}{\dm}$ fixed, where $\dm \geq 3$, for any $\kappa$ in the range $\kappa \in (\kappa_1,\kappa_2)$, there exists a unique $\eta_\kappa>1$ that solves \eqref{eqn:kappa-eta} -- see Figure \ref{fig:H}(b) for an illustration that corresponds to this case. This value of $\eta_\kappa$ determines $s_\kappa$ by \eqref{eqn:s-eta} and then $\lambda_\kappa$, uniquely. This determines uniquely a fully supported equilibrium $\rho_\kappa$ in the form \eqref{eqn:equil-fs}. 
\medskip

{\em Case iii) $0<m< 1- \frac{2}{\dm-1}$.} Similar to case ii), the integrals in \eqref{eqn:H} are convergent. Define again $\kappa_2$ by \eqref{eqn:kappa2}; however, for this range of $m$, $H(\eta)$ is decreasing, so $\kappa_2 <\kappa_1$. In this case, for 
$0<m< 1- \frac{2}{\dm-1}$ fixed, with $\dm \geq 3$, for any $\kappa \in (\kappa_2,\kappa_1)$, there exists a unique $\eta_\kappa>1$ that solves \eqref{eqn:kappa-eta} -- see Figure \ref{fig:H}(c). Then, by \eqref{eqn:s-eta} one determines $s_\kappa$, and finally set $\lambda_\kappa = - \kappa s_\kappa \eta_\kappa$. Hence, a unique equilibrium $\rho_\kappa$ in the form \eqref{eqn:equil-fs} is determined.

\begin{remark}
\label{rmk:kappact} For any $0<m<1$, a fully supported equilibrium $\rho_\kappa$ gets born at $\kappa = \kappa_1$, from the uniform distribution $\rhou$. For $1- 2/\dm < m<1$ (case i), this equilibrium exists for all $\kappa > \kappa_1$, and no further bifurcation occurs. However, when $0<m< 1- 2/\dm$ (cases ii) and iii)), a transition occurs at $\kappa = \kappa_2$, where an equilibrium of the form \eqref{eqn:mu-alpha} (a measure with a singular component) gets born from $\rho_\kappa$.
\end{remark}

\begin{remark}\label{kappa2-exp}
The explicit expression of $\kappa_2$ can be obtained as
\[
\kappa_2=\frac{m \left(\frac{1}{m-1}+\dm \right)}{2^{1+(\dm-1)(m-1)}|\bbs^\dm|^{m-1}}\left(\frac{\Gamma\left(\frac{1}{2}\right)\Gamma\left(\frac{1}{m-1}+\dm\right)}{\Gamma\left(\frac{1}{m-1}+\frac{\dm}{2}\right)\Gamma\left(\frac{\dm+1}{2}\right)}\right)^{m-1}.
\]
The derivation is provided in Appendix \ref{appendix:kappa2}. Note that $0<m<1-\frac{2}{\dm}$ here, so $\frac{1}{m-1}+\dm > \frac{\dm}{2} >0$. As $m\nearrow1-\frac{2}{\dm}$, $\kappa_2\to+\infty$ since $\frac{1}{m-1}+\frac{\dm}{2}\searrow0+$, and this yields $\Gamma\left(\frac{1}{m-1}+\frac{\dm}{2}\right)\to\infty$ while $m-1<0$. This limiting behaviour is consistent with the fact that there is no second critical value $\kappa_2$ when $1-\frac{2}{\dm}\leq m<1$. 
\end{remark}

\subsection{Second critical value $\kappa = \kappa_2$} 
\label{subsect:mu}

We only consider here cases ii) and iii) from Section \ref{subsect:fs} and investigate the ranges $\kappa >\kappa_2$ and $\kappa<\kappa_2$, respectively (see \eqref{eqn:kappa2}).   

At $\kappa=\kappa_2$  we have $\eta_{\kappa_2}=1$ and $-\lambda_{\kappa_2} = \kappa_2 s_{\kappa_2} $, with $s_{\kappa_2}$ given by (see \eqref{eqn:s-eta}) 
\begin{equation}
\label{eqn:s-kgk2}
s_{\kappa_2} =\frac{\int_0^\pi (1-\cos\theta)^{\frac{1}{m-1}}\sin^{\dm-1}\theta \cos\theta\d\theta}{\int_0^\pi (1-\cos\theta)^{\frac{1}{m-1}}\sin^{\dm-1}\theta\d\theta}.
\end{equation}

Correspondingly, the equilibrium in the form \eqref{eqn:equil-fs} is given by
\begin{equation}
\label{eqn:rho-k2}
\rho_{\kappa_2}(x) = \left(\frac{m}{1-m}\right)^{\frac{1}{1-m}} (-\lambda_{\kappa_2})^{\frac{1}{m-1}}\left(1- \cos\theta_x\right)^{\frac{1}{m-1}}, \qquad\forall x\in \bbs^\dm.
\end{equation}
Using the first equation in \eqref{eqn:system-fs-ml1}, we find
\begin{equation}
\label{eqn:lambda-k2}
(-\lambda_{\kappa_2})^{\frac{1}{m-1}} = \frac{1}{\dm w_\dm} \left( \frac{m}{1-m} \right)^{\frac{1}{m-1}} \frac{1}{\int_0^\pi\bigl (1 -  \cos \theta \bigr)^{\frac{1}{m-1}} \sin^{\dm-1}\theta \, \d\theta},
\end{equation}
which used in \eqref{eqn:rho-k2} yields
\[
\rho_{\kappa_2}(x) = \frac{\bigl( 1 -  \cos \theta_x \bigr)^{\frac{1}{m-1}}}{\dm w_\dm \int_0^\pi\bigl (1 -  \cos \theta \bigr)^{\frac{1}{m-1}} \sin^{\dm-1}\theta \d \theta} , \qquad \forall x \in \bbs^\dm.
\]

Remarkably, this is the density (see \eqref{eqn:equil-rho-reg}) of the regular component of the measure-valued equilibrium in the form \eqref{eqn:mu-alpha}; also compare \eqref{eqn:s-kgk2} with \eqref{eqn:crho-kgk2} and note that \eqref{eqn:lambda-k2} corresponds to \eqref{eqn:mlambda} for $\alpha =0$. We infer that at $\kappa=\kappa_2$, an equilibrium $\mu_\alpha$ of the form \eqref{eqn:mu-alpha} gets born from the fully supported equilibrium $\rho_{\kappa_2}$. 

We investigate now the measure-valued equilibria in the form \eqref{eqn:mu-alpha}, for a fixed $\kappa>\kappa_2$ (case ii) or $\kappa<\kappa_2$ (case iii). Recall from Section \ref{sect:cp} that the density $\rho$ given by \eqref{eqn:equil-rho-reg} does not depend on $\kappa$. To avoid confusing notation, we redenote the density from \eqref{eqn:equil-rho-reg} by $\brho$:
\begin{equation}
\label{eqn:equil-brho}
\brho(x) = \frac{\bigl( 1 -  \cos \theta_x \bigr)^{\frac{1}{m-1}}}{\dm w_\dm \int_0^\pi\bigl (1 -  \cos \theta \bigr)^{\frac{1}{m-1}} \sin^{\dm-1}\theta \d \theta} , \qquad \forall x \in \bbs^\dm.
\end{equation}
Here, $\theta_x$ represents the angle measured from a fixed (but arbitrary) direction $x_0 \in \bbs^\dm$, which represents the normalized centre of mass of $\brho$:
\begin{equation}
\label{eqn:x_0}
x_0= \frac{c_{\brho}}{\|c_{\brho}\|}.
\end{equation}
We also denote
\begin{equation}
\label{eqn:srho}  
\srho = \|c_{\brho}\|,
\end{equation}
and corresponding to \eqref{eqn:crho-kgk2} we have
\begin{equation}
\label{eqn:srho-kgk2}
\srho =\frac{\int_0^\pi (1-\cos\theta)^{\frac{1}{m-1}}\sin^{\dm-1}\theta \cos\theta\d\theta}{\int_0^\pi (1-\cos\theta)^{\frac{1}{m-1}}\sin^{\dm-1}\theta\d\theta}.
\end{equation}

With notation \eqref{eqn:equil-brho}, the measure-valued equilibria \eqref{eqn:mu-alpha} has the form 
\begin{equation}
\label{eqn:mu-alpha-mod}
\mu_\alpha = \alpha \delta_{x_0} + (1-\alpha) \brho \, \d S.   
\end{equation}
The only parameter that needs to be determined (in terms of $\kappa$) is $\alpha$, the strength of the singular part. Using \eqref{eqn:mlambda} in \eqref{eqn:lambda-kappa} -- with $\brho$ instead of $\rho$, along with notation \eqref{eqn:srho}, we arrive at the following equation for $\alpha$:
\begin{equation}
\kappa  (\alpha + (1-\alpha)\srho) = (\dm w_\dm) ^{1-m} \, \frac{m(1-\alpha)^{m-1}}{1-m}\left( \int_0^\pi \bigl (1 -  \cos \theta \bigr)^{\frac{1}{m-1}} \sin^{\dm-1}\theta \, \d\theta \right)^{1-m}.
\label{eqn:find-alpha}
\end{equation}

By the considerations above, at $\kappa=\kappa_2$, a solution of \eqref{eqn:find-alpha} is $\alpha=0$. Hence, \eqref{eqn:find-alpha} can be simplified into
\begin{equation}
\label{eqn:alpha-fg}
\kappa(\srho+\alpha(1-\srho))=(1-\alpha)^{m-1}\kappa_2 \srho.
\end{equation}

Denote the l.h.s. and the r.h.s. of \eqref{eqn:alpha-fg} by
\begin{equation}
\label{eqn:fg}
f_\kappa(\alpha)=\kappa(\srho+\alpha(1-\srho)), \quad \text{ and } \quad g(\alpha)=(1-\alpha)^{m-1}\kappa_2 \srho.
\end{equation}
We are now concerned with finding the number of solutions $\alpha \in (0,1)$ to $f_\kappa(\alpha) = g(\alpha)$, for a given value of $\kappa$. Note that $f_\kappa$ is a linear function of $\alpha$, while $g$ is an increasing function of $\alpha$ which becomes infinite as $\alpha \to 1$. We also note that the expression for $\srho$ can be simplified using the Gamma function (see Appendix \ref{appendix:s}) to
\begin{equation}
\label{eqn:s-simple}
\srho = \frac{1}{(1-m)\dm-1}.
\end{equation}
Since $0<m<1-\frac{2}{\dm}$, it holds indeed that $0<\srho<1$.

\begin{lemma}
\label{lem:alpha}
Assume $\dm \geq 3$, and consider the ranges of $m$ from cases ii) and iii) in Section \ref{subsect:fs}; note that for $\dm =3$ only case ii) applies. Then, 
\begin{itemize}
\item Case $1-\frac{2}{\dm-1}<m<1-\frac{2}{\dm}$. For any $\kappa > \kappa_2$, equation \eqref{eqn:find-alpha} has a unique solution $\alpha_\kappa \in (0,1)$.

\item Case $0<m<1-\frac{2}{\dm-1}$. There exists a critical value $\kappa_3<\kappa_2$ such that: 
\begin{itemize}
    \item for $\kappa_3 <\kappa <\kappa_2$, there exist two solutions $0<\alphaone<\alphatwo<1$ of \eqref{eqn:find-alpha}. At $\kappa=\kappa_3$, the two solutions $\alphaone$ and $\alphatwo$ coincide. Also, as $\kappa \nearrow \kappa_2$ we have $\alphaone \searrow 0$.
    \item for $\kappa>\kappa_2$, \eqref{eqn:find-alpha} has a unique solution $\alphatwo \in (0,1)$.
\end{itemize}
\end{itemize}
\end{lemma}
\begin{proof} We consider the two cases separately. 

{\em $\bullet$ Case $1-\frac{2}{\dm-1} < m<1-\frac{2}{\dm}$.} This is case ii) from Section \ref{subsect:fs}.  We note that both $f_\kappa(0)$ and $f_\kappa'(0)$ are increasing functions of $\kappa$. Therefore, if we show that for $\kappa=\kappa_2$, $\alpha =0$ is the only solution of \eqref{eqn:alpha-fg}, then we can infer the uniqueness of the solution for $\kappa>\kappa_2$ -- see Figure \ref{fig:fgplot}(a) for an illustration of this case. 

We know that $f_{\kappa_2}(0) = g(0)$, and we will show that $f_{\kappa_2}'(0)<g'(0)$. Given that $f_{\kappa_2}$ is linear and $g$ is increasing, this shows that $\alpha =0$ is the only solution of \eqref{eqn:alpha-fg} for $\kappa = \kappa_2$. Indeed, the slope of the tangent line of $g(\alpha)$ at $\alpha=0$ can be calculated as
\[
g'(0) = (1-m)\kappa_2 \srho.
\]
Also, $f_{\kappa_2}'(0) = \kappa_2(1-\srho)$. By a simple calculation we find that
\[
\kappa_2(1-\srho) < (1-m)\kappa_2 \srho,
\]
as this is equivalent to
\[
1 < (2-m)\srho,
\]
which after using \eqref{eqn:s-simple}, is equivalent to
\[
m> 1- \frac{2}{\dm-1}.
\]
The latter inequality is satisfied by the assumption on $m$.

Therefore, for any $\kappa > \kappa_2$ there exists a unique solution $0<\alpha_\kappa<1$ of $f_\kappa(\alpha)=g(\alpha)$ for this case.
\medskip

{\em $\bullet$ Case $0<m<1-\frac{2}{\dm-1}$.} From a similar argument (reverse the inequality sign in the calculation above), in this case we have 
\[
f_{\kappa_2}'(0)>g'(0).
\]
Therefore, in addition to $\alpha=0$, the equation $f_{\kappa_2}(\alpha) = g(\alpha)$ has an additional solution in the interval $(0,1)$ -- see Figure \ref{fig:fgplot}(b) for an illustration. 

Recall that both $f_\kappa(0)$ and $f_\kappa'(0)$ are increasing functions of $\kappa$. Hence, for $\kappa>\kappa_2$, there exists a unique solution in $(0,1)$ of \eqref{eqn:find-alpha}. Also, by decreasing the values of $\kappa$ down from $\kappa_2$, we infer that there exists $\kappa_3<\kappa_2$ such that $f_{\kappa_3}(\alpha)$ and $g(\alpha)$ are tangent to each other at some $\overline{\alpha} \in (0,1)$. To find $\kappa_3$ and $\overline{\alpha}$ one needs to solve
\[
\begin{cases}
f_{\kappa_3}(\overline{\alpha})=g(\overline{\alpha}),\\[2pt]
f_{\kappa_3}'(\overline{\alpha})=g'(\overline{\alpha}).
\end{cases}
\]
Using the expressions of $f_\kappa$ and $g$, the system above becomes 
\[
\begin{cases}
\kappa_3(\srho+\overline{\alpha}(1-\srho))=(1-\overline{\alpha})^{m-1}\kappa_2 \srho,\\[2pt]
\kappa_3(1-\srho)=(1-m)(1-\overline{\alpha})^{m-2}\kappa_2 \srho,
\end{cases}
\]
and solving the system we find
\begin{equation}
\label{eqn:alphac}
\overline{\alpha} = \frac{1-2\srho+m \srho}{(1-\srho)(2-m)}, 
\end{equation}
and
\begin{equation}
\label{eqn:kappa3}
\kappa_3 = \kappa_2 \frac{(1-m)\srho}{1-\srho} (1-\overline{\alpha})^{m-2},
\end{equation}
with $\srho$ given by \eqref{eqn:s-simple}.

For $\kappa_3<\kappa<\kappa_2$, there exist two solutions $0<\alphaone<\alphatwo<1$ of \eqref{eqn:find-alpha}. At $\kappa=\kappa_3$, the two solutions $\alphaone$ and $\alphatwo$ are equal to $\overline{\alpha}$. For $\kappa<\kappa_3$ there exist no solutions.
\end{proof}

\begin{figure}[!htbp]
\begin{center}
 \begin{tabular}{cc}
 \includegraphics[width=0.48\textwidth]{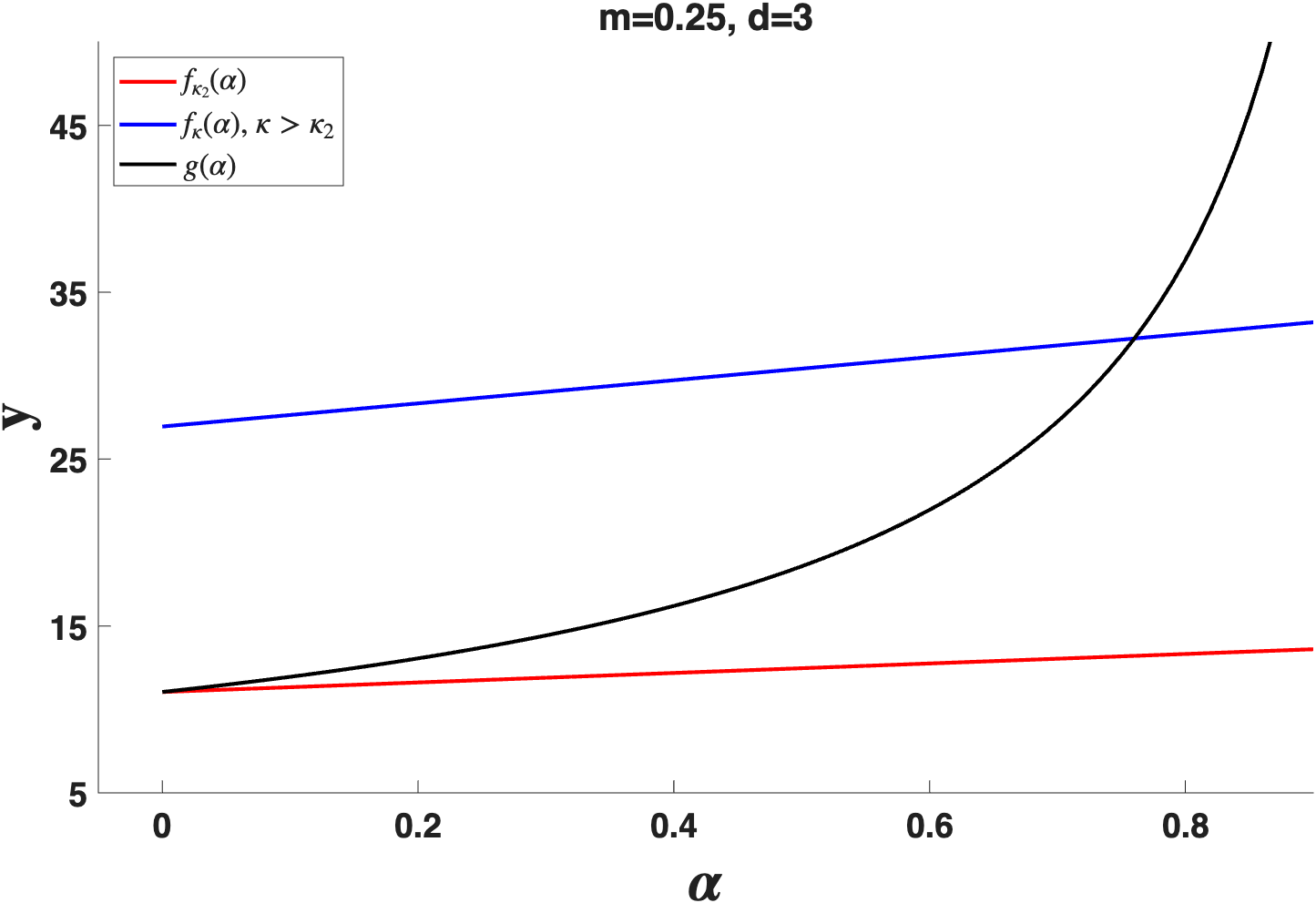} & 
 \includegraphics[width=0.48\textwidth]{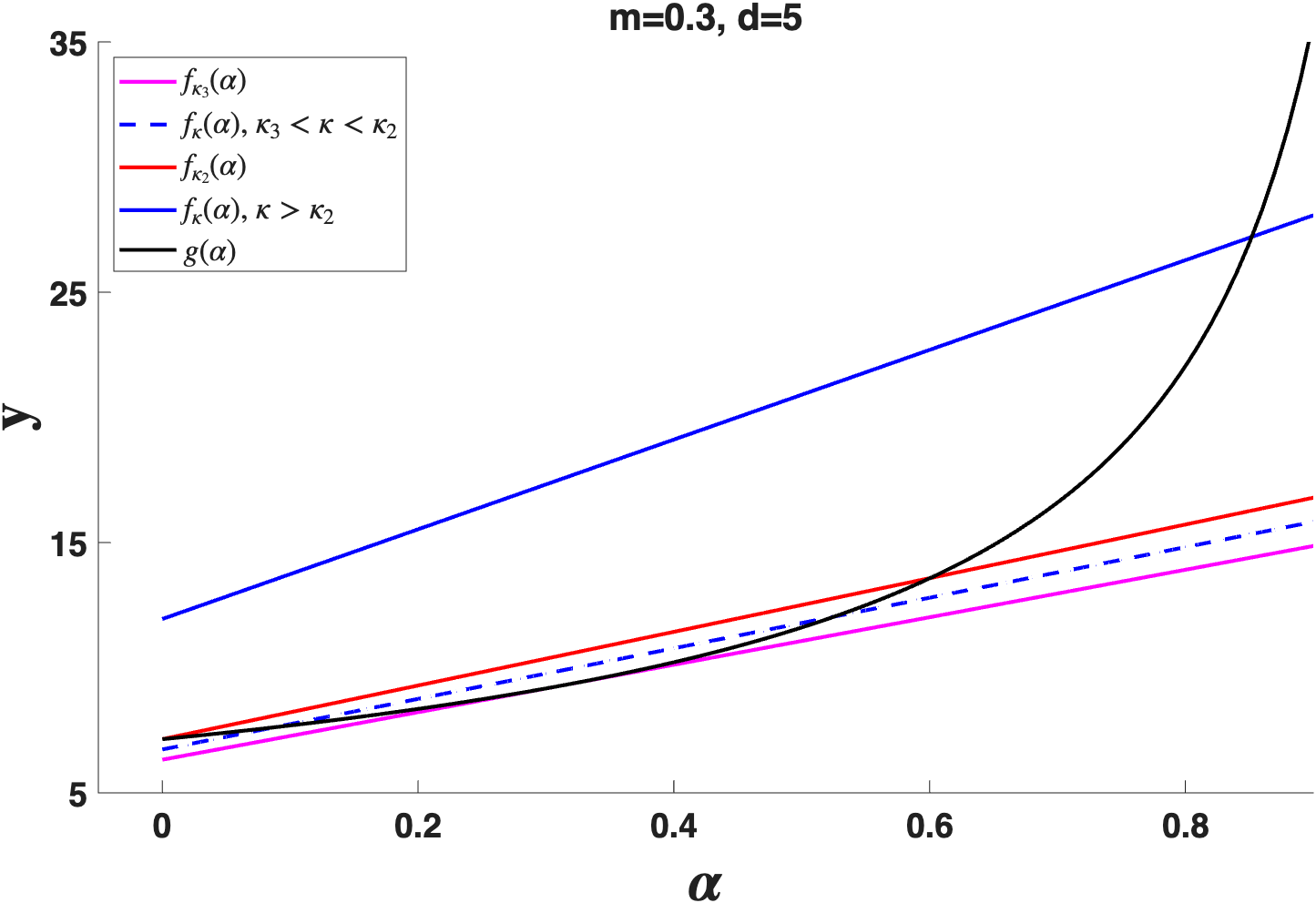} \\
 (a) & (b)
\end{tabular}
\caption{Plots of $f_\kappa$ and $g$ from \eqref{eqn:fg}. (a) Case $1-\frac{2}{\dm-1}<m<1-\frac{2}{\dm}$: for any $\kappa > \kappa_2$, equation \eqref{eqn:find-alpha} has a unique solution $\alpha_\kappa \in (0,1)$. (b) Case $0<m<1-\frac{2}{\dm-1}$: for $\kappa_3 <\kappa <\kappa_2$, there exist two solutions $0<\alphaone<\alphatwo<1$ of \eqref{eqn:find-alpha}. At $\kappa=\kappa_3$, the two solutions $\alphaone$ and $\alphatwo$ coincide. For $\kappa>\kappa_2$, \eqref{eqn:find-alpha} has a unique solution in the interval $(0,1)$. The numerical simulations correspond to (a) $m=0.25$ and $\dm=3$, and (b) $m=0.3$ and $\dm=5$.}
\label{fig:fgplot}
\end{center}
\end{figure}

In each of the cases of Lemma \ref{lem:alpha}, for every solution $\alpha_\kappa$ of \eqref{eqn:find-alpha} it corresponds an equilibrium $\mu_{\alpha_\kappa}$ in the form \eqref{eqn:mu-alpha-mod}, where the density $\brho$ is given by \eqref{eqn:equil-brho}.

The findings in this section can be put together in the following Proposition. 

\begin{proposition}[Phase transitions and equilibria]
\label{prop:bif}
Given a dimension $\dm \geq 1$, we distinguish three ranges of $0<m<1$: i) $1- \frac{2}{\dm} < m <1$, ii) $1- \frac{2}{\dm-1}<m< 1- \frac{2}{\dm}$, and iii) $0<m<1- \frac{2}{\dm-1}$. Note that for $\dm=1$ and $\dm=2$, cases ii) and iii) do not apply, while for $\dm=3$ case iii) is not applicable. Then,

\noindent a) In all three cases, at $\kappa=\kappa_1$ (see \eqref{eqn:kappa1}), an equilibrium fully supported on $\bbs^\dm$ bifurcates from the uniform distribution $\rhou$. This equilibrium, in the form \eqref{eqn:equil-fs}, is given by
\begin{equation}
\label{eqn:rhok}
\rho_\kappa(x) = \left(\frac{m}{1-m}\right)^{\frac{1}{1-m}} \left(-\lambda_\kappa- \kappa s_\kappa\cos\theta_x\right)^{\frac{1}{m-1}}, \qquad\forall x\in \bbs^\dm,
\end{equation}
where $s_\kappa$ and $\lambda_\kappa$ are uniquely determined by $\kappa$ (see \eqref{eqn:slambda} and \eqref{eqn:skappa-eta}), as given by:
\begin{equation}
\label{eqn:skappa}
s_\kappa = \frac{\int_0^\pi (\eta_\kappa-\cos\theta)^{\frac{1}{m-1}}\sin^{\dm-1}\theta \cos\theta\d\theta}{\int_0^\pi (\eta_\kappa-\cos\theta)^{\frac{1}{m-1}}\sin^{\dm-1}\theta\d\theta},
\end{equation}
$\lambda_\kappa = - \kappa s_\kappa \eta_\kappa$, with $\eta_\kappa$ being the unique solution of $H(\eta_\kappa)=\kappa^{-1}$.

In case i), the equilibria $\rho_\kappa$ exist for all $\kappa \in (\kappa_1,\infty)$. In case ii), they exist only for $\kappa \in (\kappa_1,\kappa_2)$, where $\kappa_2>\kappa_1$, and in case iii) they exist for $\kappa \in (\kappa_2,\kappa_1)$. For both cases ii) and iii), $\kappa_2$ is given by \eqref{eqn:kappa2}. \\

\noindent b) In cases ii) and iii), at $\kappa=\kappa_2$ the equilibria \eqref{eqn:rhok} transition from density-valued to measure-valued. 
\begin{itemize}
\item case ii) $1-\frac{2}{\dm-1}<m<1-\frac{2}{\dm}$. By Lemma \ref{lem:alpha}, for any $\kappa > \kappa_2$, equation \eqref{eqn:find-alpha} has a unique solution $0< \alpha_\kappa <1 $. Hence, for any $\kappa > \kappa_2$, there exists a unique equilibrium in the form \eqref{eqn:mu-alpha-mod}, given by 
\begin{equation}
\label{eqn:mu-alphak}
\mu_{\alpha_\kappa} = \alpha_\kappa \delta_{x_0} + (1-\alpha_\kappa) \brho \, \d S,
\end{equation}
where the density $\brho$ (which does not depend on $\kappa$) is given by \eqref{eqn:equil-brho}.
Here, $x_0 \in \bbs^\dm$ is a fixed (arbitrary) unit vector; see also \eqref{eqn:x_0}.
\item case iii) $0<m<1-\frac{2}{\dm-1}$. In this case, there exists a third critical value $\kappa_3<\kappa_2$ (see \eqref{eqn:kappa3} and Lemma \ref{lem:alpha}) such that: 
\begin{itemize}
    \item at $\kappa=\kappa_3$ a pair of equilibria in the form \eqref{eqn:mu-alpha} gets born. For any $\kappa_3 <\kappa <\kappa_2$, there exist two solutions $0<\alphaone<\alphatwo<1$ of \eqref{eqn:find-alpha}. To each of the two values $\alphaone$ and $\alphatwo$, it corresponds a measure-valued equilibrium in the form \eqref{eqn:mu-alphak}, with $\brho$ given by \eqref{eqn:equil-brho}. At $\kappa=\kappa_3$, the two solutions $\alphaone$ and $\alphatwo$ (and hence the corresponding equilibria) coincide. 
    
    \item as $\kappa \nearrow \kappa_2$ we have $\alphaone \searrow 0$, and its corresponding equilibrium $\mu_{\alphaone}$ changes at $\kappa=\kappa_2$ from measure-valued to an equilibrium density in the form \eqref{eqn:rhok}. On the other hand, $\mu_{\alphatwo}$ undergoes no transition at $\kappa=\kappa_2$ and retains the form \eqref{eqn:mu-alphak}. As $\kappa$ increases to infinity, $\mu_{\alphatwo}$ approaches $\delta_{x_0}$.
\end{itemize}
\end{itemize}
\end{proposition}


\subsection{Numerical illustrations} 
\label{subsect:numerics-ml1}

We illustrate Proposition \ref{prop:bif} with numerical results. For case i) we used $m=0.5$ and $\dm=2$. Figure \ref{fig:m05-se}(a) shows the bifurcation at $\kappa_1$ -- for this choice of $m$ and $\dm$ we have $\kappa_1 \approx 5.3174$. Blue colour corresponds to the uniform distribution (with centre of mass at the origin), which is unstable for $\kappa>\kappa_1$ (Proposition \ref{prop:stab-unif}). In red we plot $s_\kappa = \|c_{\rho_\kappa}\|$ for the fully supported equilibria \eqref{eqn:rhok}, which for case i) exist for all $\kappa>\kappa_1$. In Figure \ref{fig:m05-rhok}(a) we plot the equilibria $\rho_\kappa$ at several values of $\kappa$ (at $\kappa=\kappa_1$, $\rho_\kappa$ is the uniform distribution). As $\kappa$ increases, $\rho_k$ concentrates around $\theta=0$, behaviour which can also be inferred from Figure \ref{fig:m05-se}(a) since $\|c_{\rho_\kappa}\|$ approaches $1$ as $\kappa \to \infty$. 

\begin{figure}[htbp]
 \begin{center}
 \begin{tabular}{cc}
 \includegraphics[width=0.48\textwidth]{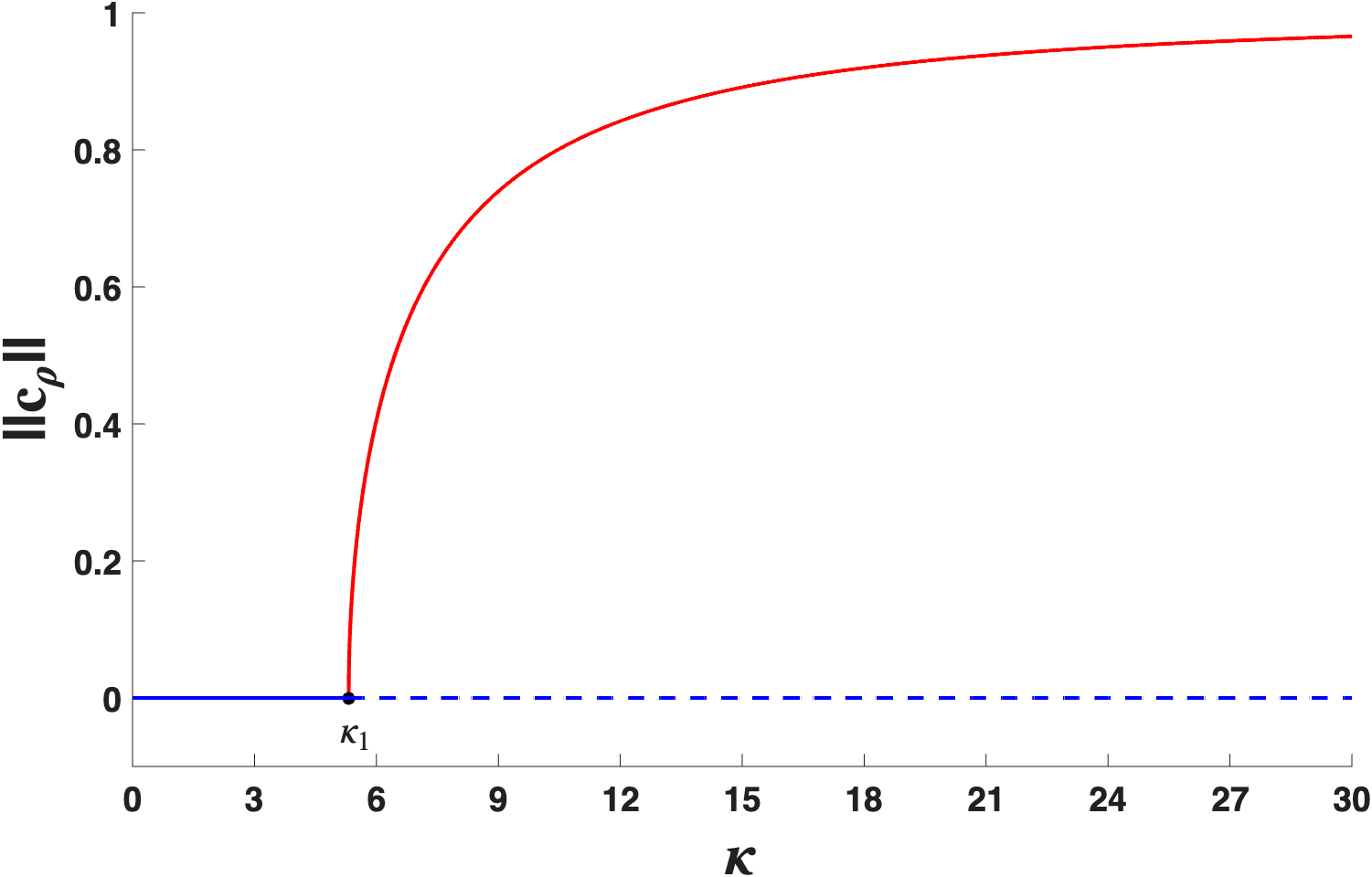} & 
 \includegraphics[width=0.48\textwidth]{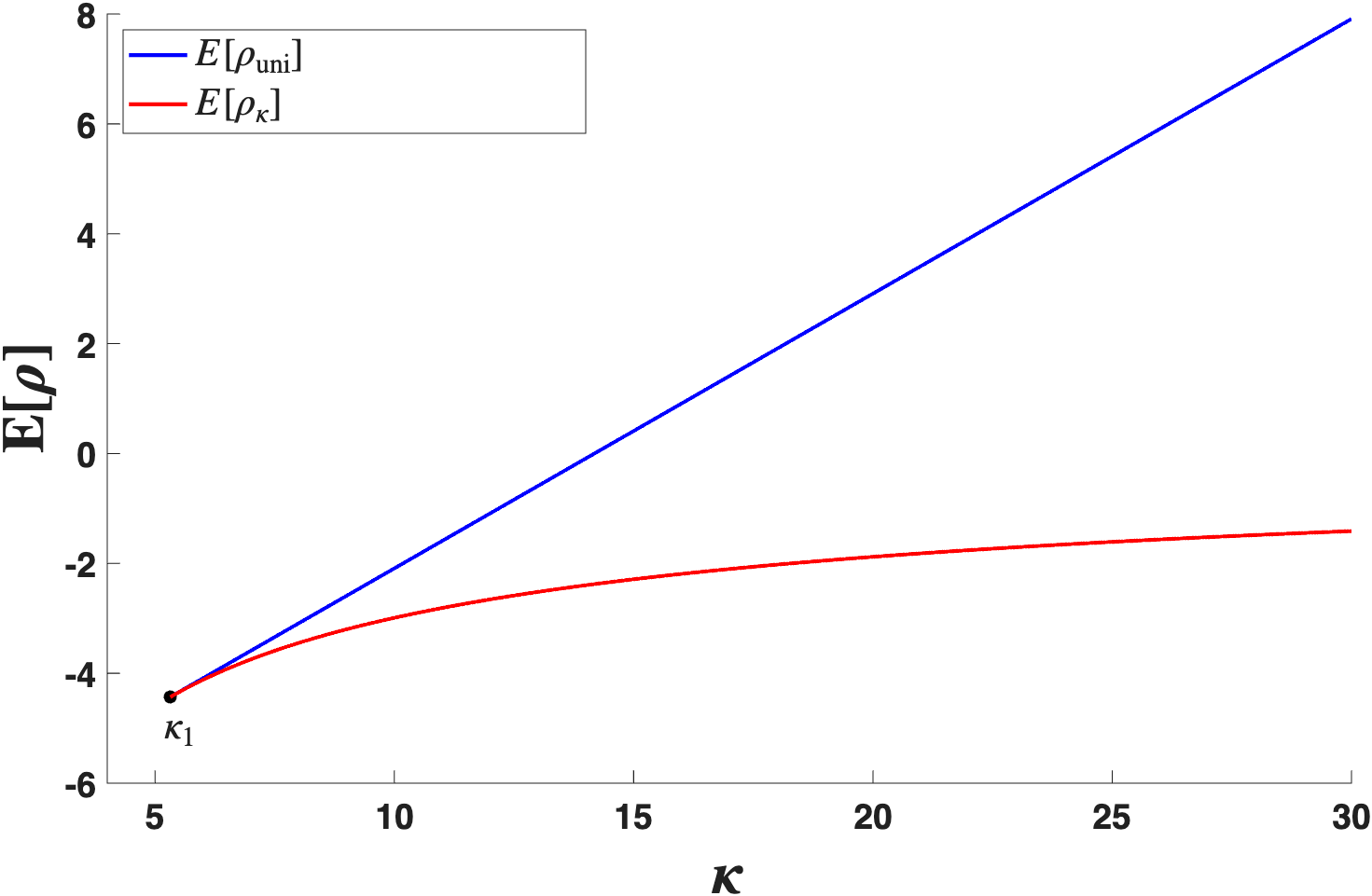} \\
 (a) & (b)
\end{tabular}
\caption{Case i) $1-2/\dm<m<1$. (a) Plot of the norm of the centre of mass of equilibria. b) Plot of the energies of equilibria for $\kappa>\kappa_1$. In both plots, blue corresponds to $\rhou$ and red to $\rho_\kappa$ -- see Proposition \ref{prop:bif} part a). At $\kappa=\kappa_1$, a fully supported equilibrium $\rho_\kappa$ in the form \eqref{eqn:rhok} emerges from the uniform distribution, and then it exists for all $\kappa>\kappa_1$. The equilibrium $\rho_\kappa$ is the global energy minimizer when $\kappa>\kappa_1$ -- see Theorem \ref{thm:global-min} part a).  The numerical simulations correspond to $m=0.5$ and $\dm=2$.}
\label{fig:m05-se}
\end{center}
\end{figure}

\begin{figure}[htbp]
 \begin{center}
 \begin{tabular}{cc}
 \includegraphics[width=0.48\textwidth]{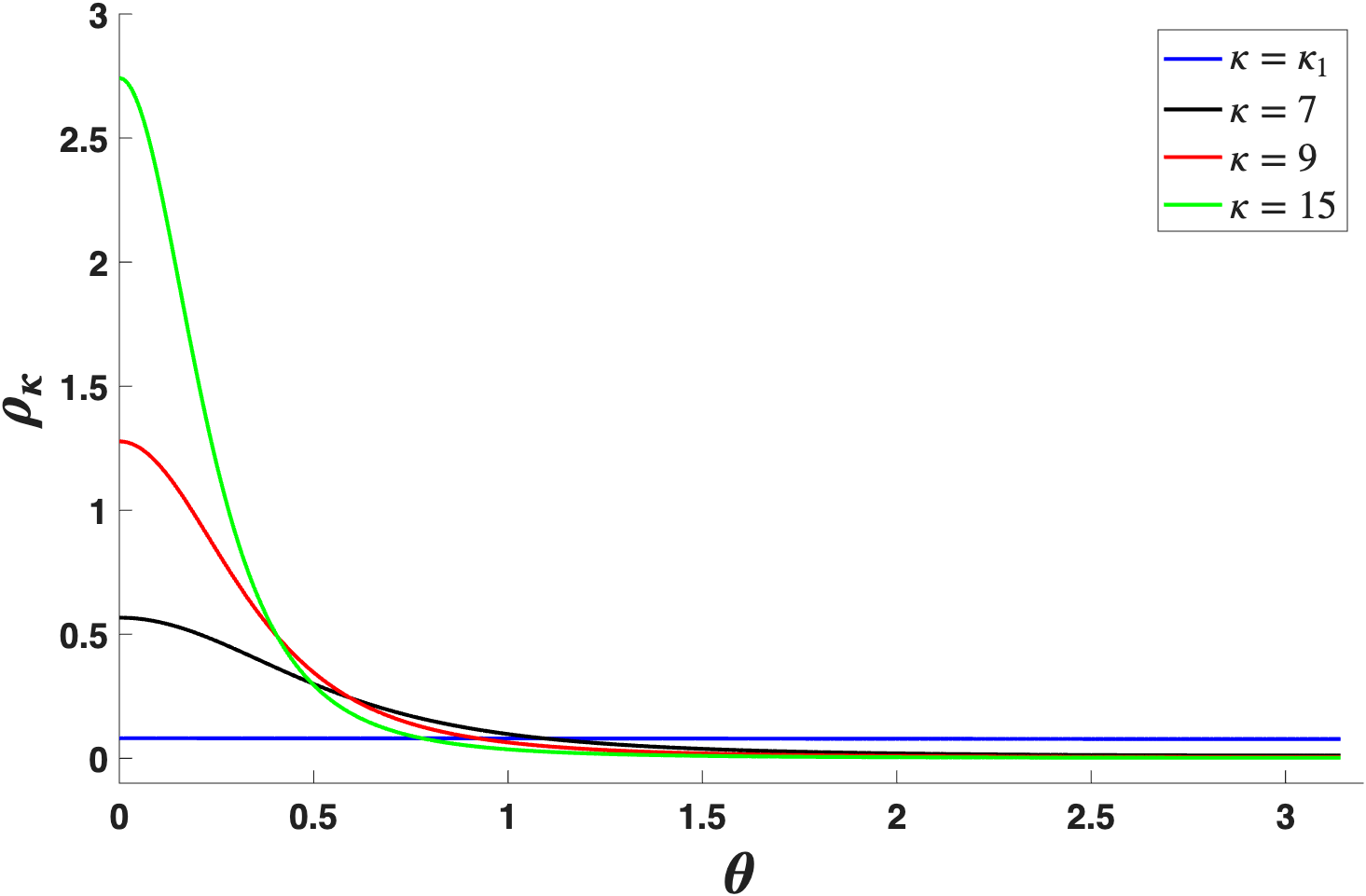} &
 \includegraphics[width=0.48\textwidth]{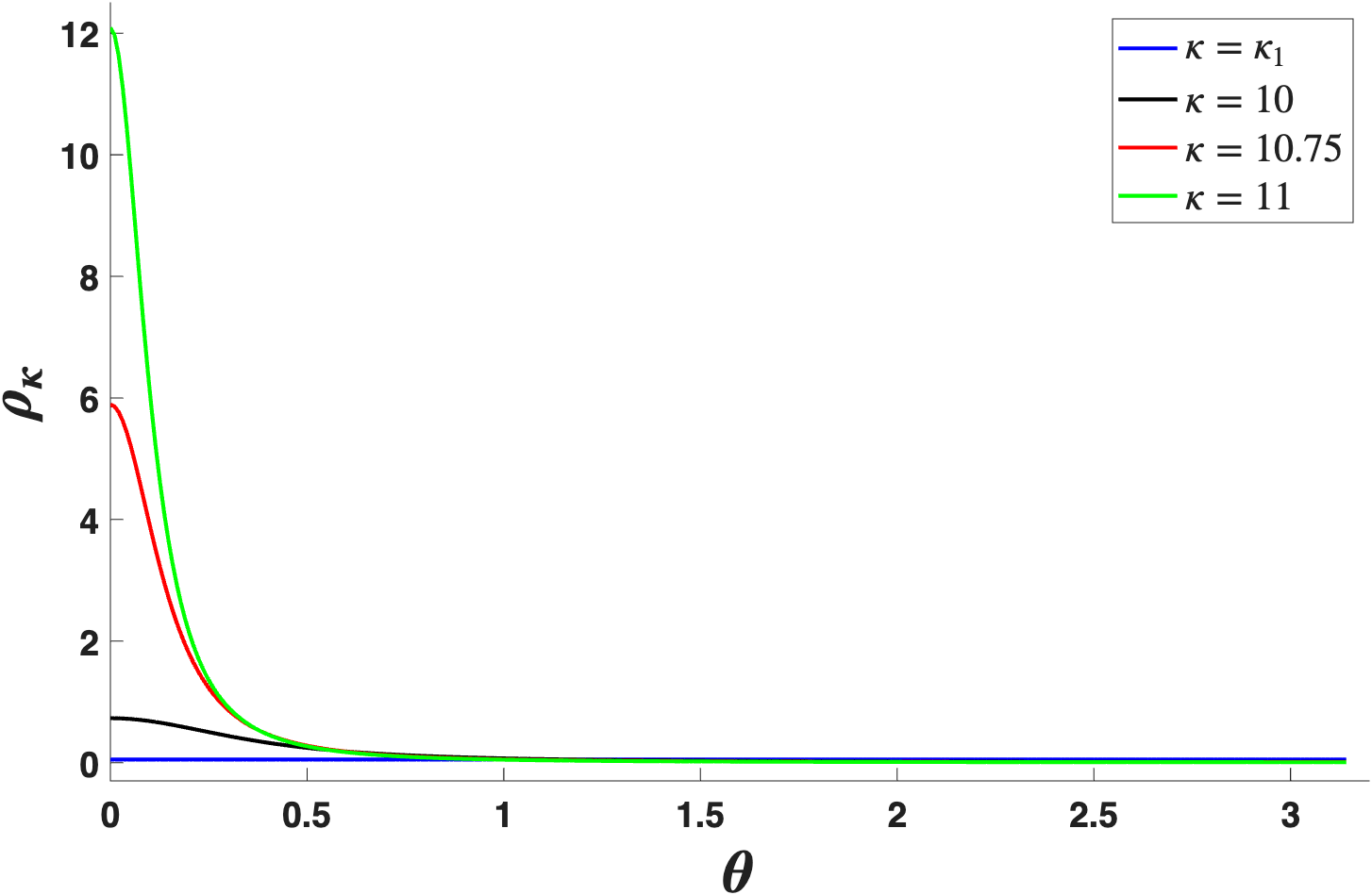} \\
  (a) & (b)
 \end{tabular}
\caption{Plot of equilibrium densities $\rho_\kappa$ from \eqref{eqn:rhok}. (a) $m=0.5$, $\dm=2$. These values correspond to case i) from Proposition \ref{prop:bif}. The equilibrium densities $\rho_\kappa$ exist for all $\kappa>\kappa_1$. (b) $m=0.25$, $\dm=3$. These values correspond to case ii) from Proposition \ref{prop:bif}. The equilibrium densities $\rho_\kappa$ exist for $\kappa_1<\kappa<\kappa_2$ in this case (here $\kappa_2 \approx 12.4453$).  For both plots, $\kappa=\kappa_1$ corresponds to the uniform distribution, and $\rho_\kappa$ concentrates around $\theta =0$ as $\kappa$ increases.}
\label{fig:m05-rhok}
\end{center}
\end{figure}

For case ii) we used $m=0.25$ and $\dm=3$. Figure \ref{fig:m025-se}(a) illustrates the bifurcation at $\kappa_1 \approx 9.3648$ and the transition to measure-valued equilibria at $\kappa_2 \approx 12.4453$. Blue corresponds to the uniform distribution, while the red line represents $\|c_{\rho_\kappa}\|$ ($\kappa_1<\kappa<\kappa_2$) -- see \eqref{eqn:rhok} or $\|c_{\mu_{\alpha_\kappa}}\|$ ($\kappa>\kappa_2$) -- see \eqref{eqn:mu-alphak} and \eqref{eqn:equil-brho}. Figure \ref{fig:m05-rhok}(b) shows the equilibria $\rho_\kappa$ at several values of $ \kappa <\kappa_2$ (so before the equilibria become measure-valued), where $\kappa=\kappa_1$ corresponds to the uniform distribution. For $\kappa \geq \kappa_2$, the regular part $\brho$ of the equilibrium $\mu_{\alpha_\kappa}$ (see \eqref{eqn:equil-brho}) is infinite at $\theta =0$, so we could not provide good illustrations for this range. As $\kappa$ increases, $\alpha_\kappa \nearrow 1$ (see Figure \ref{fig:fgplot}(a)) and the equilibria $\mu_{\alpha_\kappa}$ approach a delta Dirac.

\begin{figure}[htbp]
 \begin{center}
 \begin{tabular}{cc}
 \includegraphics[width=0.48\textwidth]{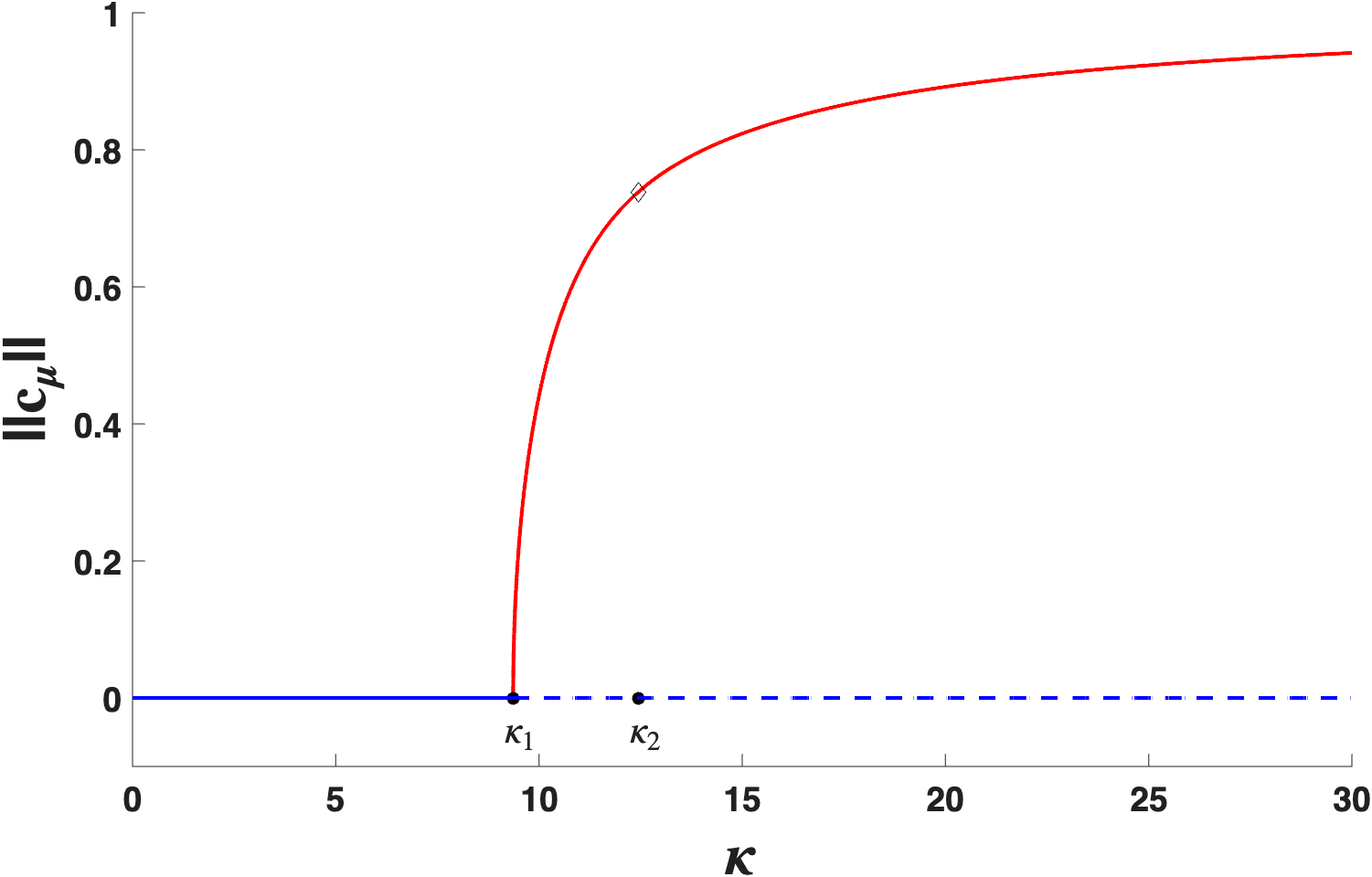} & 
 \includegraphics[width=0.48\textwidth]{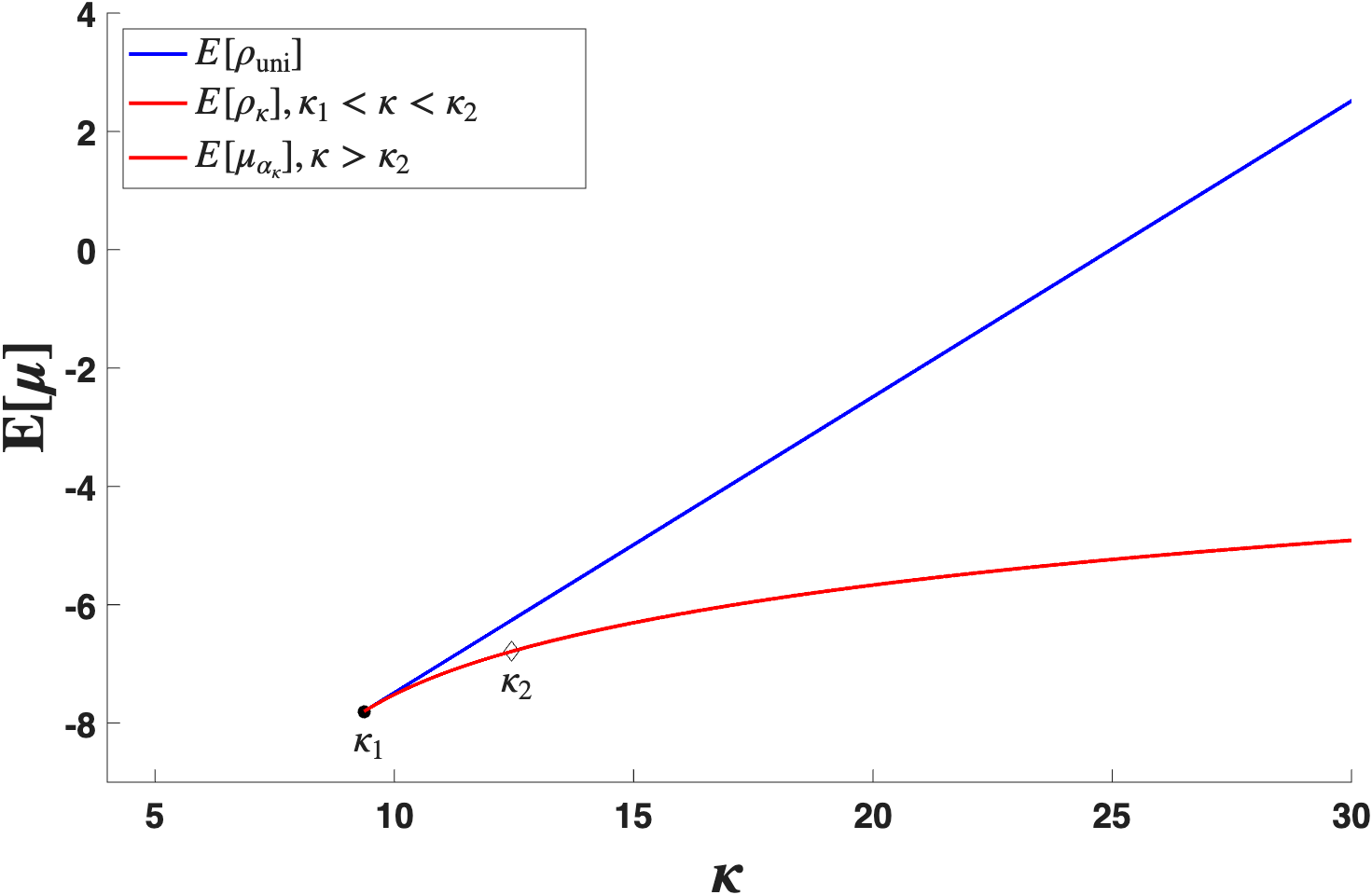} \\
 (a) & (b)
\end{tabular}
\caption{Case ii) $1-2/(\dm-1)<m<1-2/\dm$. (a) Plot of the norm of the centre of mass of equilibria. (b) Plot of the energies of equilibria for $\kappa>\kappa_1$. In both plots, blue corresponds to $\rhou$ and red to either $\rho_\kappa$ (for $\kappa_1<\kappa<\kappa_2$) or $\mu_{\alpha_\kappa}$ (for $\kappa>\kappa_2$) -- see Proposition \ref{prop:bif}. At $\kappa=\kappa_1$, a fully supported equilibrium $\rho_\kappa$ in the form \eqref{eqn:rhok} emerges from the uniform distribution. At $\kappa=\kappa_2$, $\rho_\kappa$ changes to a measure-valued equilibrium $\mu_{\alpha_\kappa}$ -- see \eqref{eqn:mu-alphak} and \eqref{eqn:equil-brho}; this transition is indicated by a black diamond. The global minimizer when $\kappa_1<\kappa<\kappa_2$ is $\rho_\kappa$, and when $\kappa>\kappa_2$, the ground state is $\mu_{\alpha_\kappa}$ -- see Theorem \ref{thm:global-min} part a). The numerical simulations correspond to $m=0.25$, $\dm=3$.}
\label{fig:m025-se}
\end{center}
\end{figure}

Finally, for case iii) we used $m=0.3$ and $\dm=5$. Figure \ref{fig:m03-se}(a) shows the bifurcation at $\kappa_1 \approx 19.9199$, the transition to measure-valued equilibria at $\kappa_2 \approx 17.8623$, and the bifurcation at $\kappa_3 \approx 15.8088$, where two measure valued-equilibria emerge. Blue corresponds to the uniform distribution, and the solid red line represents $\|c_{\mu_{\alpha_\kappa}}\|$ (which exists for all $\kappa>\kappa_3)$. Also, the dashed red line represents either $\|c_{\mu_{\tilde{\alpha}_\kappa}}\|$ for $\kappa_3<\kappa<\kappa_2$ (see \eqref{eqn:mu-alphak} and \eqref{eqn:equil-brho}), or $\|c_{\rho_{_\kappa}}\|$ for $\kappa_2<\kappa<\kappa_1$ (see \eqref{eqn:rhok}). At $\kappa = \kappa_2$ the equilibrium $\mu_{\alphaone}$ changes from measure-valued to an equilibrium density $\rho_\kappa$ in the form \eqref{eqn:rhok}; this transition is indicated by a black diamond. At $\kappa=\kappa_1$, $\rho_\kappa$ merges with the uniform distribution and vanishes. We do not provide separate illustrations of $\rho_\kappa$ here, as the qualitative behaviour is similar to Figure \ref{fig:m05-rhok}(b), except that $\rho_\kappa$ now concentrates as $\kappa$ decreases from $\kappa_1$ to $\kappa_2$. Finally, as $\kappa$ increases to infinity, the equilibrium $\mu_{\alpha_\kappa}$ on the upper branch approaches a delta Dirac (as $\alpha_\kappa \nearrow 1$). 

\begin{figure}[htbp]
 \begin{center}
 \begin{tabular}{cc}
 \includegraphics[width=0.48\textwidth]{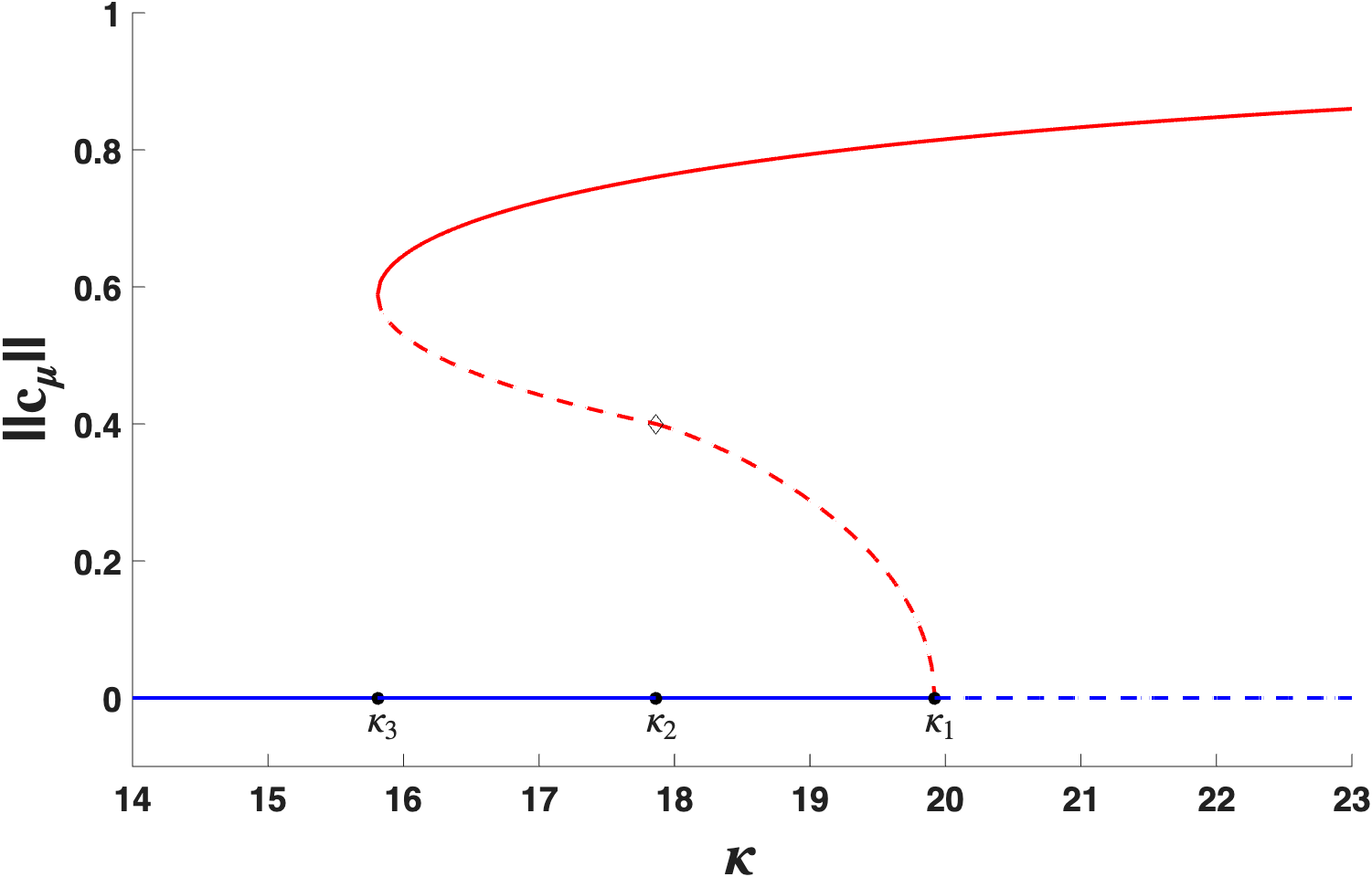} & 
 \includegraphics[width=0.48\textwidth]{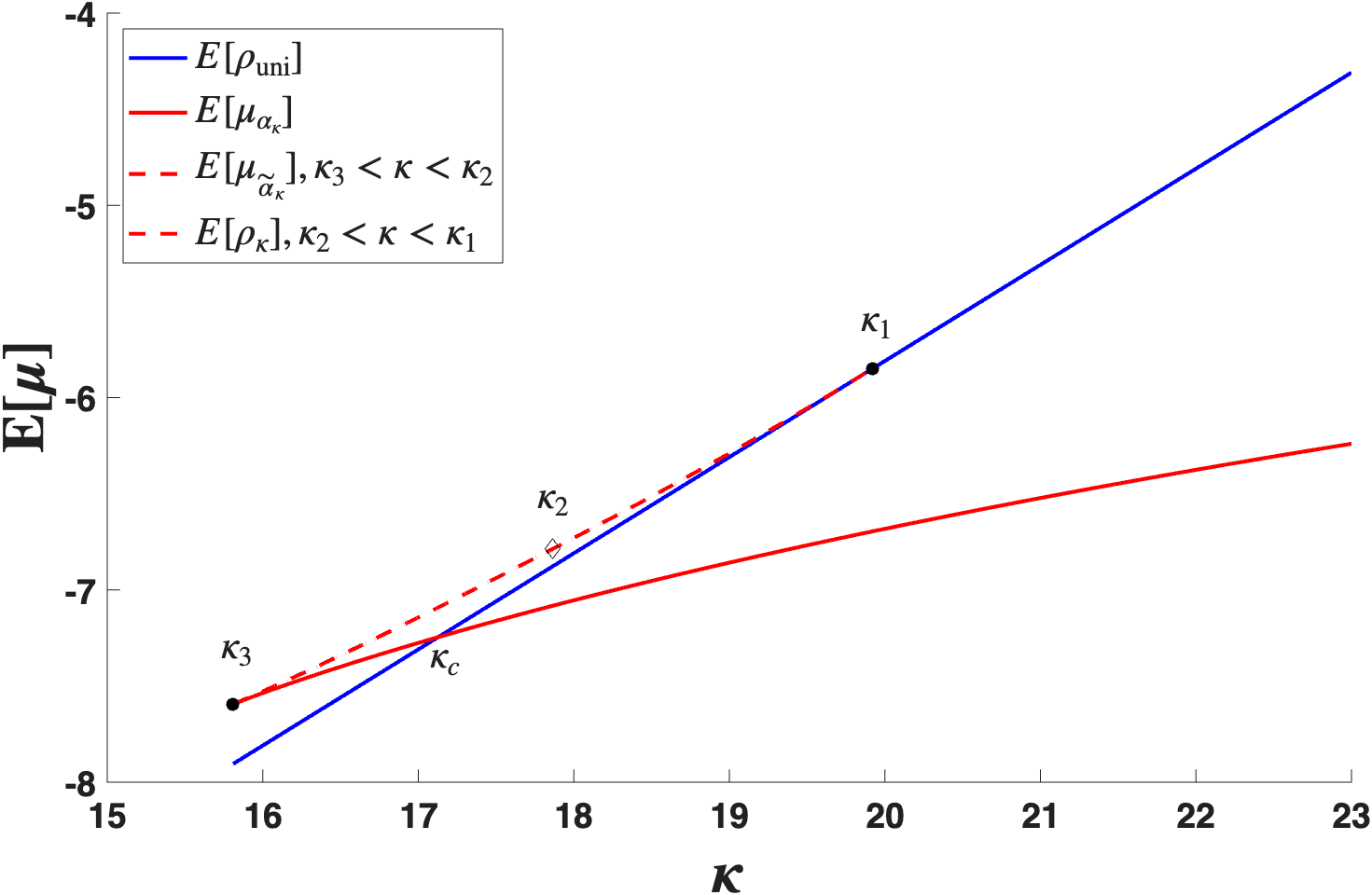} \\
 (a) & (b)
\end{tabular}
\caption{Case iii) $0<m<1-2/(\dm-1)$. (a) Plot of the norm of the centre of mass of equilibria. (b) Plot of the energies of equilibria. In both plots, blue corresponds to $\rhou$ and the solid red line relates to the measure-valued equilibrium  $\mu_{\alpha_\kappa}$ (which exists for all $\kappa>\kappa_3$). Also, the dashed red line corresponds to either $\mu_{\tilde{\alpha}_\kappa}$ (for $\kappa_3<\kappa<\kappa_2$) or $\rho_\kappa$ (for $\kappa_2<\kappa<\kappa_1$) -- see Proposition \ref{prop:bif}. At $\kappa=\kappa_3$, a pair of measure-valued equilibria in the form \eqref{eqn:mu-alphak} emerges. At $\kappa=\kappa_2$, the measure-valued equilibrium $\mu_{\tilde{\alpha}_\kappa}$ on the lower branch changes to a fully supported density-valued equilibrium $\rho_\kappa$ in the form \eqref{eqn:rhok}; this transition is indicated by a black diamond. At $\kappa=\kappa_1$, the equilibrium $\rho_\kappa$ merges with the uniform distribution and vanishes. As $\kappa \to \infty$, the equilibria $\mu_{\alpha_\kappa}$ on the upper branch approaches a delta Dirac. A fourth critical value $\kappa_c \in (\kappa_1,\kappa_3)$ exists in this case. The global energy minimizer is $\rhou$ for $\kappa<\kappa_c$ and $\mu_{\alphatwo}$ for $\kappa>\kappa_c$ -- see Theorem \ref{thm:global-min} part b). The numerical simulations correspond to $m=0.3$, $\dm=5$.}
\label{fig:m03-se}
\end{center}
\end{figure}


\section{Energy minimizers}
\label{sect:minimizers}

In this section we investigate which of the equilibria identified in Proposition \ref{prop:bif} are global energy minimizers. The result is given by the following theorem. 

\begin{theorem} [Global energy minimizers]
\label{thm:global-min}
Consider the uniform distribution together with the equilibria established in Proposition \ref{prop:bif}, for a fixed 
$\dm \geq 1$. Note that we examined three ranges of $0<m<1$ there, where for $\dm=1,2$ and $3$ only some of the ranges are applicable. Then,

\noindent a) In cases i) and ii), the uniform distribution $\rhou$ is the global energy minimizer on $\calP(\bbs^\dm)$ for all $\kappa<\kappa_1$. For $\kappa>\kappa_1$, we have
\begin{itemize}
    \item case i) $1- \frac{2}{\dm} < m <1$. The global minimizer for any $\kappa>\kappa_1$ is given by the equilibrium $\rho_\kappa$ in \eqref{eqn:rhok},
    \item case ii) $1- \frac{2}{\dm-1}<m< 1- \frac{2}{\dm}$. The global minimizer for $\kappa_1<\kappa<\kappa_2$ is given by $\rho_\kappa$ in \eqref{eqn:rhok}, while for $\kappa>\kappa_2$, the ground state is the measure-valued equilibrium $\mu_{\alpha_\kappa}$ from \eqref{eqn:mu-alphak}.
\end{itemize}
Moreover, in both cases the energy differences $E[\rhou]-E[\rho_\kappa]$ and $E[\rhou]-E[\mu_{\alpha_\kappa}]$ increase with $\kappa$ for $\kappa>\kappa_1$.
\smallskip

\noindent b) In case iii) $1- \frac{2}{\dm-1}<m< 1- \frac{2}{\dm}$, there exists $\kappa_3< \kappa_c < \kappa_1$ such that the global energy minimizer on $\calP(\bbs^\dm)$ is 
\begin{itemize}
    \item the uniform distribution $\rhou$ when $\kappa<\kappa_c$,
    \item the measure-valued equilibrium $\mu_{\alphatwo}$ in the form \eqref{eqn:mu-alphak}, as identified in Proposition \ref{prop:bif}, when $\kappa>\kappa_c$.
\end{itemize}
Moreover, the energy difference $E[\rhou]-E[\mu_{\alphatwo}]$ increases with $\kappa$ for $\kappa>\kappa_3$.
\end{theorem}

\begin{proof} We will compare the energies of the various equilibria that exist in each case, for various ranges of $\kappa$.
\smallskip

{\em Step 1: Compare $E[\rhou]$ with $E[\rho_\kappa]$.} This calculation applies to case i) when $\kappa>\kappa_1$, to case ii) when $\kappa_1<\kappa<\kappa_2$, and to case iii) when $\kappa_2<\kappa<\kappa_1$.

By \eqref{eqn:energy-s} we write
\begin{equation}
\label{eqn:Ediff-rhok}
E[\rhou]-E[\rho_\kappa] = \frac{1}{m-1} \int_{\bbs^\dm} \rhou^m(x) \dS - \frac{1}{m-1} \int_{\bbs^\dm}  \rho_\kappa(x)^m \dx +  \frac{\kappa}{2} \|c_{\rho_\kappa}\|^2.
\end{equation}
The first term in the r.h.s. above is a constant, and it does not change with $\kappa$. We will investigate how the other two terms change with $\kappa$.

Using \eqref{eqn:skappa-eta} and \eqref{eqn:rhok}, we compute
\begin{equation}
\label{eqn:Ediff-rhok-c}
\begin{aligned}
\frac{\kappa}{2} \|c_{\rho_\kappa}\|^2 - \frac{1}{m-1} \int_{\bbs^\dm} \rho_\kappa^m(x) \dS &= \frac{m}{2(1-m) (\dm w_\dm)^{m-1}} \frac{\int_0^\pi (\eta_\kappa - \cos \theta)^{\frac{1}{m-1}} \sin^{\dm-1} \theta \cos \theta \d \theta}{\left( \int_0^\pi (\eta_\kappa - \cos \theta)^{\frac{1}{m-1}} \sin^{\dm-1} \theta \d \theta \right)^m} \\[3pt]
&\quad - \frac{1}{(m-1) (\dm w_\dm)^{m-1}} \frac{\int_0^\pi (\eta_\kappa - \cos \theta)^{\frac{m}{m-1}} \sin^{\dm-1} \theta \d \theta}{\left( \int_0^\pi (\eta_\kappa - \cos \theta)^{\frac{1}{m-1}} \sin^{\dm-1} \theta \d \theta \right)^m} \\[3pt]
&:= g_1(\eta_\kappa) g_2(\eta_\kappa),
\end{aligned}
\end{equation}
where 
\begin{align*}
g_1(\eta) &= m \int_0^\pi (\eta - \cos \theta)^{\frac{1}{m-1}} \sin^{\dm-1} \theta \cos \theta \d \theta + 2 \int_0^\pi (\eta - \cos \theta)^{\frac{m}{m-1}} \sin^{\dm-1} \theta \d \theta, \\[5pt]
g_2(\eta) &= \frac{1}{2(1-m)(\dm w_\dm)^{m-1}} \frac{1}{\left( \int_0^\pi (\eta - \cos \theta)^{\frac{1}{m-1}} \sin^{\dm-1} \theta \d \theta \right)^m}.
\end{align*}

Calculate 
\begin{align*}
g_1'(\eta) &= \frac{m}{m-1} \int_0^\pi (\eta - \cos \theta)^{\frac{1}{m-1}-1} \sin^{\dm-1} \theta \cos \theta \d \theta + \frac{2m}{m-1} \int_0^\pi (\eta - \cos \theta)^{\frac{1}{m-1}} \sin^{\dm-1} \theta \d \theta \\[3pt]
&= \frac{m}{m-1} \int_0^\pi (\eta - \cos \theta)^{\frac{1}{m-1}-1} \sin^{\dm-1} \theta \cos \theta \d \theta \\[3pt]
& \quad + \frac{2m}{m-1} \int_0^\pi (\eta - \cos \theta)^{\frac{1}{m-1}-1} (\eta - \cos \theta) \sin^{\dm-1} \theta \d \theta \\[3pt]
& = \frac{2m \eta}{m-1} \int_0^\pi (\eta - \cos \theta)^{\frac{1}{m-1}-1} \sin^{\dm-1} \theta \d \theta - \frac{m}{m-1} \int_0^\pi (\eta - \cos \theta)^{\frac{1}{m-1}-1} \sin^{\dm-1} \theta \cos \theta \d \theta,
\end{align*}
and 
\begin{align*}
g_2'(\eta) = \frac{m}{2(1-m)^2 (\dm w_\dm)^{m-1}} \frac{\int_0^\pi (\eta - \cos \theta)^{\frac{1}{m-1}-1} \sin^{\dm-1} \theta \d \theta}{ \left( \int_0^\pi (\eta - \cos \theta)^{\frac{1}{m-1}} \sin^{\dm-1} \theta \d \theta \right)^{m+1}}.
\end{align*}
Then,
\begin{align*}
& (g_1(\eta) g_2(\eta))' = g_1'(\eta) g_2(\eta) + g_1(\eta) g_2'(\eta) \\[2pt]
&\quad = \frac{\frac{m}{1-m} \int_0^\pi (\eta - \cos \theta)^{\frac{1}{m-1}-1} \sin^{\dm-1} \theta \d \theta }{2(1-m)(\dm w\dm)^{m-1} \left( \int_0^\pi (\eta - \cos \theta)^{\frac{1}{m-1}} \sin^{\dm-1} \theta \d \theta \right)^m} \times \\[3pt]
&\qquad \left(  -2 \eta + \frac{\int_0^\pi (\eta - \cos \theta)^{\frac{1}{m-1}-1} \sin^{\dm-1} \theta \cos \theta \d \theta}{\int_0^\pi (\eta - \cos \theta)^{\frac{1}{m-1}-1} \sin^{\dm-1} \theta \d \theta}\right. \\[3pt]
&\quad \qquad \left. + \frac{m \int_0^\pi (\eta - \cos \theta)^{\frac{1}{m-1}} \sin^{\dm-1} \theta \cos \theta \d \theta + 2 \int_0^\pi (\eta - \cos \theta)^{\frac{m}{m-1}} \sin^{\dm-1} \theta \d \theta}{\int_0^\pi (\eta - \cos \theta)^{\frac{1}{m-1}} \sin^{\dm-1} \theta \d \theta} \right).
\end{align*}

Now write 
\begin{align*}
\int_0^\pi (\eta - \cos \theta)^{\frac{m}{m-1}} \sin^{\dm-1} \theta \d \theta &= \int_0 ^\pi (\eta - \cos \theta)^{\frac{1}{m-1}} (\eta - \cos \theta)\sin^{\dm-1} \theta \d \theta \\[3pt]
&= \eta  \int_0 ^\pi (\eta - \cos \theta)^{\frac{1}{m-1}} \sin^{\dm-1} \theta \d \theta - \int_0^\pi (\eta - \cos \theta)^{\frac{1}{m-1}} \sin^{\dm-1} \theta \cos \theta \d \theta,
\end{align*}
to simplify
\begin{equation}
\label{eqn:g1g2p-s}
\begin{aligned}
 (g_1(\eta) g_2(\eta))' &= \frac{\frac{m}{1-m} \int_0^\pi (\eta - \cos \theta)^{\frac{1}{m-1}-1} \sin^{\dm-1} \theta \d \theta }{2(1-m)(\dm w\dm)^{m-1} \left( \int_0^\pi (\eta - \cos \theta)^{\frac{1}{m-1}} \sin^{\dm-1} \theta \d \theta \right)^m} \times \\[3pt]
& \quad \left( \frac{\int_0^\pi (\eta - \cos \theta)^{\frac{1}{m-1}-1} \sin^{\dm-1} \theta \cos \theta \d \theta}{\int_0^\pi (\eta - \cos \theta)^{\frac{1}{m-1}-1} \sin^{\dm-1} \theta \d \theta} + (m-2) \frac{ \int_0^\pi (\eta - \cos \theta)^{\frac{1}{m-1}} \sin^{\dm-1} \theta \cos \theta \d \theta}{\int_0^\pi (\eta - \cos \theta)^{\frac{1}{m-1}} \sin^{\dm-1} \theta \d \theta} \right).
\end{aligned}
\end{equation}

By a direct calculation using the expression of $H(\eta)$ in \eqref{eqn:H}, we compute
\begin{equation}
\label{eqn:HpoverH}
\frac{H'(\eta)}{H(\eta)} = \frac{1}{m-1} \, \frac{\int_0^\pi (\eta-\cos \theta)^{\frac{1}{m-1}-1} \sin^{\dm-1} \theta \cos \theta \d \theta}{\int_0^\pi (\eta-\cos \theta)^{\frac{1}{m-1}} \sin^{\dm-1} \theta \cos \theta \d \theta} +  \frac{m-2}{m-1} \, \frac{\int_0^\pi (\eta-\cos \theta)^{\frac{1}{m-1}-1} \sin^{\dm-1} \theta \d \theta}{\int_0^\pi (\eta-\cos \theta)^{\frac{1}{m-1}} \sin^{\dm-1} \theta \d \theta}.
\end{equation}
Then, by combining \eqref{eqn:g1g2p-s} and \eqref{eqn:HpoverH}, we write
\begin{equation}
\label{eqn:g1g2-diff}
(g_1(\eta) g_2(\eta))' = - (1-m) h(\eta)\frac{H'(\eta)}{H(\eta)},
\end{equation}
where
\begin{equation}
\label{eqn:h}
h(\eta) := \frac{m}{2(1-m)^2 (\dm w_\dm)^{m-1}} \, \frac{\int_0^\pi (\eta - \cos \theta)^{\frac{1}{m-1}} \sin^{\dm-1} \theta \cos \theta \d \theta }{\left( \int_0^\pi (\eta - \cos \theta)^{\frac{1}{m-1}} \sin^{\dm-1} \theta \d \theta \right)^m}.
\end{equation}

 Note that $h(\eta)>0$. To show this, we only need to investigate the sign of $\int_0^\pi (\eta - \cos \theta)^{\frac{1}{m-1}} \sin^{\dm-1} \theta \cos \theta \d \theta$. Since $\frac{1}{m-1}<0$, we find $(\eta - \cos \theta)^{\frac{1}{m-1}}-(\eta +\cos \theta)^{\frac{1}{m-1}}
>0$ for all $0\leq \theta< \frac{\pi}{2}$ and $\eta>1$. Therefore, we have
\begin{align*}
&\int_0^\pi (\eta - \cos \theta)^{\frac{1}{m-1}} \sin^{\dm-1} \theta \cos \theta \d \theta\\
&=\int_0^{\pi/2} (\eta - \cos \theta)^{\frac{1}{m-1}} \sin^{\dm-1} \theta \cos \theta \d \theta+\int_{\pi/2}^\pi (\eta - \cos \theta)^{\frac{1}{m-1}} \sin^{\dm-1} \theta \cos \theta \d \theta\\
&=\int_0^{\pi/2} (\eta - \cos \theta)^{\frac{1}{m-1}} \sin^{\dm-1} \theta \cos \theta \d \theta-\int_0^{\pi/2} (\eta +\cos \theta)^{\frac{1}{m-1}} \sin^{\dm-1} \theta \cos \theta \d \theta\\
&=\int_0^{\pi/2} \left((\eta - \cos \theta)^{\frac{1}{m-1}}-(\eta +\cos \theta)^{\frac{1}{m-1}}\right) \sin^{\dm-1} \theta \cos \theta \d \theta>0,
\end{align*}
where for the second equal sign we made a change of variable $\tilde{\theta} = \pi - \theta$.

Since $h(\eta)>0$ and $H(\eta)>0,$ the sign of $(g_1(\eta)g_2(\eta))'$ is opposite to the sign of $H'(\eta)$. For cases i) and ii), $H(\eta)$ is increasing with $\eta>1$, while for case iii), $H(\eta)$ is decreasing -- see Lemma \ref{lem:H-monotone}.  Consequently, $g_1(\eta)g_2(\eta)$ is decreasing with $\eta$ in cases i) and ii), and increasing with $\eta$ in case iii).

From \eqref{eqn:Ediff-rhok} and \eqref{eqn:Ediff-rhok-c}, we have
\begin{equation}
\label{eqn:E-unimrhok}
E[\rhou]-E[\rho_\kappa] = \text{const.} + g_1(\eta_\kappa) g_2(\eta_\kappa).
\end{equation}
In cases i) and ii), $\eta_\kappa$ decreases as $\kappa$ increases from $\kappa_1$ (since $H(\eta)$ is increasing). Therefore, by the considerations above, $g_1(\eta_\kappa) g_2(\eta_\kappa)$ (and hence the energy difference $E[\rhou]-E[\rho_\kappa]$) increases with $\kappa$. Note that since $\rho_\kappa$ forms at $\kappa=\kappa_1$, the energy difference starts from value $0$. We also point out that for $\kappa<\kappa_1$ in cases i) and ii), the uniform distribution is the only critical point, and therefore the global minimizer. We have now partially shown part a) of the proposition: case i) is complete, and case ii) is shown for $\kappa_1<\kappa<\kappa_2$. For a numerical illustration we refer to Figures \ref{fig:m05-se}(b) and \ref{fig:m025-se}(b).

In case iii), $\eta_\kappa$ increases as $\kappa$ increases from $\kappa_2$ to $\kappa_1$ (since $H(\eta)$ is decreasing). Then, by the monotonicity of $g_1(\eta)g_2(\eta)$, $g_1(\eta_\kappa) g_2(\eta_\kappa)$ (and hence the energy difference $E[\rhou]-E[\rho_\kappa]$) increases as $\kappa$ ranges from $\kappa_2$ to $\kappa_1$. Since at $\kappa=\kappa_1$, $\rho_\kappa$ coincides with $\rhou$, we conclude that $E[\rhou]-E[\rho_\kappa]$ increases from a negative value at $\kappa=\kappa_2$ to value $0$ at $\kappa=\kappa_1$; see Figure \ref{fig:m03-se}(b) for a numerical illustration.
\smallskip

{\em Step 2: Compare $E[\rhou]$ with $E[\mu_{\alphatwo}]$ and $E[\mu_{\alphaone}]$.} This calculation applies to case ii) when $\kappa>\kappa_2$, and to case iii). Note that for case iii), there exist two measure-valued solutions: $\mu_{\alphatwo}$ (for $\kappa>\kappa_3$) and $\mu_{\alphaone}$ (for $\kappa_3<\kappa<\kappa_2$), where $0<\alphaone<\alphatwo<1$. Recall that both $\mu_{\alphaone}$ and $\mu_{\alphatwo}$ are in the form \eqref{eqn:mu-alphak}, where the density $\brho$ does not depend on $\kappa$ -- see \eqref{eqn:equil-brho}. Also recall notation \eqref{eqn:srho}.

By \eqref{eqn:energy-mu-s}, we have
\begin{equation}
\label{eqn:Ediff}
\begin{aligned}
E[\rhou]-E[\mu_{\alpha_\kappa}] & = \frac{1}{m-1} \int_{\bbs^\dm} \rhou^m(x) \dS - \frac{1}{m-1} \int_{\bbs^\dm} (1-\alpha_\kappa)^m \brho(x)^m \dx + \frac{\kappa}{2} \| c_{\mu_{\alpha_\kappa}} \|^2 \\[5pt]
& = \frac{1}{m-1} \int_{\bbs^\dm} \rhou^m(x) \dS - \frac{1}{m-1} \int_{\bbs^\dm} (1-\alpha_\kappa)^m \brho(x)^m \dx  \\[5pt]
& \quad +  \frac{\kappa}{2} \left( \alpha_\kappa + (1-\alpha_\kappa) \srho \right)^2.
\end{aligned}
\end{equation}

Take a derivative with $\kappa$ in \eqref{eqn:Ediff} to get
\begin{equation}
\label{eqn:dEdiff-dk}
\begin{aligned}
\frac{\d}{\d \kappa} \left( E[\rhou]-E[\mu_{\alpha_\kappa}]\right) &= \frac{m}{m-1} \left( \int_{\bbs^\dm} (1-\alpha_\kappa)^{m-1} \brho(x)^m \dx \right) \frac{\d \alpha_\kappa}{\d \kappa} \\[3pt]
&\quad +
\frac{1}{2} \left( \alpha_\kappa + (1-\alpha_\kappa) \srho\right)^2 + \kappa (1-\srho) \left( \alpha_\kappa + (1-\alpha_\kappa) \srho \right) \frac{\d \alpha_\kappa}{\d \kappa}.
\end{aligned}
\end{equation}

We will show that the terms containing $\frac{\d \alpha_\kappa}{\d \kappa}$ cancel each other.  Indeed, using \eqref{eqn:equil-brho} we have
\begin{equation}
\label{eqn:entropyk}
\frac{m}{m-1}  \int_{\bbs^\dm} (1-\alpha_\kappa)^{m-1} \brho(x)^m \dx = \frac{m}{m-1} (\dm w_\dm)^{1-m} (1-\alpha_\kappa)^{m-1} \frac{\int_0^\pi (1-\cos \theta)^{\frac{m}{m-1}} \sin^{\dm-1} \theta \d \theta }{\left( \int_0^\pi (1 -  \cos \theta)^{\frac{1}{m-1}} \sin^{\dm-1}\theta \d \theta \right)^m}.
\end{equation}
By combining \eqref{eqn:lambda-kappa} and \eqref{eqn:mlambda} -- use $\brho$ instead of $\rho$ and $\alpha_\kappa$ instead of $\alpha$ there, we find
\begin{equation}
\label{eqn:factor1}
\begin{aligned}
\kappa \left( \alpha_\kappa + (1-\alpha_\kappa)\srho \right) &= - \frac{\lambda}{1-\alpha_\kappa} \\[5pt]
& = \frac{m}{1-m} (\dm w_\dm)^{1-m} (1-\alpha_\kappa)^{m-1} \frac{1} {\left( \int_0^\pi (1 -  \cos \theta)^{\frac{1}{m-1}} \sin^{\dm-1}\theta \d \theta \right)^{m-1}},
\end{aligned}
\end{equation}
which used in \eqref{eqn:entropyk} leads to
\begin{equation}
\frac{m}{m-1}  \int_{\bbs^\dm} (1-\alpha_\kappa)^{m-1} \brho(x)^m \dx = - \kappa \left( \alpha_\kappa + (1-\alpha_\kappa) \srho \right) \cdot  \frac{\int_0^\pi (1-\cos \theta)^{\frac{m}{m-1}} \sin^{\dm-1} \theta \d \theta }{\int_0^\pi (1 -  \cos \theta)^{\frac{1}{m-1}} \sin^{\dm-1}\theta \d \theta }.
\end{equation}

Now write 
\[
\int_0^\pi (1-\cos \theta)^{\frac{m}{m-1}} \sin^{\dm-1} \theta \d \theta  = \int_0^\pi (1-\cos \theta)^{\frac{1}{m-1}} (1-\cos \theta) \sin^{\dm-1} \theta \d \theta, 
\]
and then by \eqref{eqn:srho-kgk2} we get
\begin{equation}
\label{eqn:factor2}
\begin{aligned}
\frac{\int_0^\pi (1-\cos \theta)^{\frac{m}{m-1}} \sin^{\dm-1} \theta \d \theta }{\int_0^\pi (1 -  \cos \theta)^{\frac{1}{m-1}} \sin^{\dm-1}\theta \d \theta } &= 1- \frac{\int_0^\pi (1-\cos \theta)^{\frac{1}{m-1}} \sin^{\dm-1} \theta \cos \theta \d \theta}{\int_0^\pi (1 -  \cos \theta)^{\frac{1}{m-1}} \sin^{\dm-1} \theta \d \theta} \\[5pt]
&= 1-\srho.
\end{aligned}
\end{equation}

By using \eqref{eqn:factor1} and \eqref{eqn:factor2} in \eqref{eqn:entropyk}, we then find
\begin{equation}
\label{eqn:entropyk-f}
\frac{m}{m-1}  \int_{\bbs^\dm} (1-\alpha_\kappa)^{m-1} \brho(x)^m \dx = - \kappa(1-\srho) \left( \alpha_\kappa + (1-\alpha_\kappa)\srho \right).
\end{equation}
Return now to the r.h.s of \eqref{eqn:dEdiff-dk} and note that the terms containing $\frac{\d \alpha_\kappa}{\d \kappa}$ cancel each other. Hence, we get
\begin{equation}
\label{eqn:dEdiff-dk-f}
\frac{\d}{\d \kappa} \left( E[\rhou]-E[\mu_{\alpha_\kappa}]\right) = \frac{1}{2} \left( \alpha_\kappa + (1-\alpha_\kappa) \srho\right)^2,
\end{equation}
which is positive for all $0\leq \alpha_\kappa\leq 1$. 

We conclude that $E[\rhou]-E[\mu_{\alpha_\kappa}]$ increases with $\kappa$. Together with the previous considerations, this completes the conclusion of the theorem in case ii) -- See Figure \ref{fig:m025-se}(b) for an illustration, and hence, part (a) of the theorem is now shown. 
\smallskip

Similar to \eqref{eqn:dEdiff-dk-f} - just replace $\alphatwo$ with $\alphaone$ in the derivation above, we find
\begin{equation}
\label{eqn:dEdiff-dk-tilde}
\frac{\d}{\d \kappa} \left( E[\rhou]-E[\mu_{\alphaone}]\right) = \frac{1}{2} \left( \alphaone + (1-\alphaone) \srho\right)^2,
\end{equation}
which is also positive for all $0\leq \alphaone \leq 1$. We infer that the energy difference $E[\rhou]-E[\mu_{\alphaone}]$ increases with $\kappa$, for $\kappa_3<\kappa<\kappa_2$. Note that at $\kappa=\kappa_2$, $\mu_{\alphaone}$ coincides with $\rho_\kappa$, and we know from the previous calculations that $E[\rhou]-E[\rho_\kappa]<0$ at $\kappa=\kappa_2$.  Hence, $E[\rhou]-E[\mu_{\alphaone}]$ increases from a negative value at $\kappa=\kappa_3$, and for all $\kappa_3<\kappa<\kappa_2$ we have $E[\rhou]<E[\mu_{\alphaone}]$. Also see Figure \ref{fig:m03-se}(b) for a numerical illustration.

Regarding $\mu_{\alphatwo}$, first note that it coincides with $\mu_{\alphaone}$ at $\kappa=\kappa_3$, and by the above, $\rhou$ is more energetically favourable at $\kappa_3$. Hence, $E[\rhou]-E[\mu_{\alphatwo}]$ increases from a negative value at $\kappa=\kappa_3$. By \eqref{eqn:dEdiff-dk-f}, the rate of change of $E[\rhou]-E[\mu_{\alphatwo}]$ is strictly positive, so there exists a $\kappa_c > \kappa_3$ (a fourth critical value of $\kappa$) where $E[\rhou]-E[\mu_{\alphatwo}]=0$, and past $\kappa>\kappa_c$, we have $E[\mu_{\alphatwo}] < E[\rhou]$. To conclude part (b) and hence the proof, it remains to compare the energies of the two measure-valued solutions $\mu_{\alphaone}$ and $\mu_{\alphatwo}$ (for $\kappa_3 <\kappa<\kappa_2)$ and the energies of $\rho_\kappa$ and $\mu_{\alphatwo}$ (for $\kappa_2<\kappa<\kappa_1)$.
\smallskip

{\em Step 3: Compare $E[\mu_{\alphaone}]$ and $E[\mu_{\alphatwo}]$.} This part concerns case iii) only, with $\kappa_3<\kappa<\kappa_2$. By \eqref{eqn:energy-mu-s2} and \eqref{eqn:mu-alphak} we have 
\begin{equation}
\label{eqn:Ekdiff}
\begin{aligned}
E[\mu_{\alphaone}]-E[\mu_{\alphatwo}] & = \frac{1}{m-1} \int_{\bbs^\dm} (1-\alphaone)^m \brho(x)^m \dx  
-  \frac{\kappa}{2} \left( \alphaone + (1-\alphaone) \srho \right)^2 \\
&\quad  - \frac{1}{m-1} \int_{\bbs^\dm} (1-\alphatwo)^m \brho(x)^m \dx  +  \frac{\kappa}{2} \left( \alphatwo + (1-\alphatwo) \srho \right)^2.
\end{aligned}
\end{equation}
From a calculation similar to that leading from \eqref{eqn:dEdiff-dk} to \eqref{eqn:dEdiff-dk-f}, we get
\begin{align*}
\frac{\d} {\d \kappa} \left( E[\mu_{\alphaone}]-E[\mu_{\alphatwo}]\right) = \frac{1}{2} \left( \alphatwo + (1-\alphatwo) \srho \right)^2 - \frac{1}{2} \left( \alphaone + (1-\alphaone) \srho \right)^2.
\end{align*}
Since $\alphaone<\alphatwo$, the r.h.s. above is positive, and we infer that $E[\mu_{\alphaone}]-E[\mu_{\alphatwo}]$ increases with $\kappa$. For the range $\kappa_3<\kappa<\kappa_2$ (where $\mu_{\alphaone}$ exists) we then have
$E[\mu_{\alphaone}]>E[\mu_{\alphatwo}]$.
\smallskip

{\em Step 4: Compare $E[\rho_\kappa]$ and $E[\mu_{\alphatwo}]$.} This calculations applies to case iii), for $\kappa_2<\kappa<\kappa_1$. Using \eqref{eqn:E-unimrhok} and \eqref{eqn:dEdiff-dk-f}, we get
\begin{equation}
\label{eqn:dE-rhomu}
\frac{\d}{\d \kappa} \left( E[\rho_\kappa]-E[\mu_{\alpha_\kappa}]\right) = \frac{1}{2} \left( \alpha_\kappa + (1-\alpha_\kappa) \srho\right)^2 - \frac{\d}{\d \kappa} (g_1(\eta_\kappa) g_2(\eta_\kappa)).
\end{equation}
We will show that the r.h.s. above is positive.

By chain rule we have
\begin{equation}
\label{eqn:chain-rule}
\frac{\d}{\d \kappa} (g_1(\eta_\kappa) g_2(\eta_\kappa))= \frac{\d}{\d \eta_\kappa} (g_1(\eta_\kappa) g_2(\eta_\kappa)) \frac{\d \eta_\kappa}{\d \kappa}.
\end{equation}
Now differentiate 
\[
{H(\eta_\kappa)}^{-1} = \kappa,
\]
with respect to $\kappa$, to get
\begin{equation}
\label{eqn:detakdk}
-\frac{H'(\eta_\kappa)}{H^2(\eta_\kappa)} \frac{\d \eta_\kappa}{\d \kappa} =1.
\end{equation}
Then, using \eqref{eqn:chain-rule} together with \eqref{eqn:g1g2-diff} and \eqref{eqn:detakdk},  we find
\begin{align*}
\frac{\d}{\d \kappa} (g_1(\eta_\kappa) g_2(\eta_\kappa)) &= -(1-m) h(\eta_\kappa) \frac{H'(\eta_\kappa)}{H^2(\eta_\kappa)} \left( -\frac{H^2(\eta_\kappa)}{H'(\eta_\kappa)} \right) \\
&= \frac{1}{2} \left( \frac{\int_0^\pi (\eta_\kappa-\cos \theta)^{\frac{1}{m-1}} \sin^{\dm-1}\theta \cos \theta \d \theta}{\int_0^\pi (\eta_\kappa-\cos \theta)^{\frac{1}{m-1}} \sin^{\dm-1}\theta \d \theta}\right)^2,
\end{align*}
where for the second equal sign we used the expressions of $h(\eta)$ and $H(\eta)$ from \eqref{eqn:h} and \eqref{eqn:H}.

From \eqref{eqn:skappa} we find from the above that 
\begin{equation}
\label{eqn:dkappa-g1g2}
\frac{\d}{\d \kappa} (g_1(\eta_\kappa) g_2(\eta_\kappa)) = \frac{1}{2} s_\kappa^2.
\end{equation}
Here, $\eta_\kappa$ ranges from $1$ (for $\kappa = \kappa_2$) to $\infty$ (for $\kappa = \kappa_1$). By using \eqref{eqn:dkappa-g1g2} in \eqref{eqn:dE-rhomu}, we get
\begin{equation}
\label{eqn:Ediff-rhomu}
\frac{\d}{\d \kappa} \left( E[\rho_\kappa]-E[\mu_{\alpha_\kappa}]\right) = \frac{1}{2} \left( \alpha_\kappa + (1-\alpha_\kappa) \srho \right)^2 - \frac{1}{2} s_\kappa^2.
\end{equation}

Note that $\srho = s_{\kappa_2}$ (see \eqref{eqn:s-kgk2} and \eqref{eqn:srho-kgk2}), as $\eta_{\kappa_2}=1$. Now, since $\alpha_\kappa>0$ we have
\[
\alpha_\kappa + (1-\alpha_\kappa) \srho > \srho,
\]
and to show that the r.h.s. of \eqref{eqn:Ediff-rhomu} is positive, it is enough to show that $s_\kappa$ decreases with increasing $\kappa$ (or equivalently, with increasing $\eta_\kappa$), as this implies $\srho>s_\kappa$, for all $\kappa_2<\kappa<\kappa_1$. This behaviour was in fact observed numerically -- see the red dashed line in Figure \ref{fig:m03-se}(a).

We can show that $s_\kappa$ is a decreasing function of $\eta_\kappa$ by using the expression of $H$ from \eqref{eqn:H}, and express $s_\kappa$ in terms of $H(\eta_\kappa)$ as 
\[
s_\kappa=\left(\frac{m}{(1-m)(\dm w_\dm)^{m-1}}H(\eta_\kappa)\right)\left(\int_0^\pi (\eta_\kappa-\cos\theta)^{\frac{1}{m-1}}\sin^{\dm-1}\theta \d\theta\right)^{1-m}.
\]
From Lemma \ref{lem:H-monotone}, we know that $H$ is a decreasing function for $0<m<1-\frac{2}{\dm-1}$. Since $\frac{1}{m-1}<0$ and $1-m>0$, we also infer that $\left(\int_0^\pi (\eta_\kappa-\cos\theta)^{\frac{1}{m-1}}\sin^{\dm-1}\theta \d\theta\right)^{1-m}$ is a decreasing function of $\eta_\kappa$. Therefore, as a product of two positive decreasing functions, $s_\kappa$ is also a decreasing function.

We deduce that the r.h.s. of \eqref{eqn:Ediff-rhomu} is positive for $\kappa_2<\kappa<\kappa_1$. We have $E[\rho_\kappa] - E[\mu_{\alphatwo}]>0$ at $\kappa=\kappa_2$ (as $\rho_\kappa$ coincides with $\mu_{\alphaone}$ there, and use Step 3), and from there on, the energy difference will increase to a larger positive value at $\kappa=\kappa_1$. Since $\rho_\kappa$ coincides with $\rhou$ at $\kappa=\kappa_1$, we find that $E[\rhou] > E[\mu_{\alphatwo}]$ at $\kappa_1$. By Step 2, $E[\rhou] < E[\mu_{\alphatwo}]$ at $\kappa=\kappa_3$, which implies that the fourth critical value $\kappa_c$ is between $\kappa_3$ and $\kappa_1$. By the considerations in Steps 2-4, we now conclude part b) of the proof. 

\end{proof}


\appendix
\section{Proof of Lemma \ref{Lemma:Dirac-delta-concentration}}
\label{appendix:Dirac-delta-concentration}

Assume that $\mu$ in the form \eqref{eqn:mu-decomp} is a critical point of the energy $E[\mu]$, where $\alpha>0$ and $\mu_s$ is not a Dirac delta measure concentrated at one point. By \eqref{eqn:energy-s} and \eqref{eqn:mu-decomp}, we write
\begin{equation}
\label{eqn:mu-decomp-s}
E[\mu]=\frac{1}{m-1}\int_{\bbs^\dm} (1-\alpha)^m \rho(x)^m\dx-\frac{\kappa}{2}\|\alpha c_{\mu_s}+(1-\alpha)c_\rho\|^2+\frac{\kappa}{2},
\end{equation}
where 
\[
c_{\mu_s}=\int_{\bbs^\dm}x\d {\mu_s}(x), \qquad \text{ and }\qquad c_\rho=\int_{\bbs^\dm} x \rho(x)\dx.
\]
As we can see from \eqref{eqn:mu-decomp-s}, $c_{\mu_s}$ is the only term that depends on the singular measure $\mu_s$. To show that $\mu$ cannot be a ground state, we will provide a change of $\mu_s$ which decreases the total energy, i.e., increases $\|\alpha c_{\mu_s}+(1-\alpha)c_\rho\|$. 

 First, if $\|c_\rho\|\neq0$ and the directions of $c_\rho$ and $c_{\mu_s}$ are not aligned, then a rotation that aligns the two directions decreases the energy. Indeed, assume $c_\rho$ and $c_{\mu_s}$ are not parallel. Denote by $\tilde{\mu}_s$ the singular measure obtained by rotating $\mu_s$ so that $c_\rho$ and $c_{\tilde{\mu}_s}$ are aligned, i.e., $c_\rho = \|c_\rho\| z$ and $c_{\tilde{\mu}_s} = \|c_{\tilde{\mu}_s}\| z$, for some unit vector $z \in \bbs^\dm$. Also denote 
\[
\tilde{\mu} = \alpha \tilde{\mu}_s + (1-\alpha) \rho \, \dS.
\]
Then, we have
\begin{align*}
\| \alpha c_{\mu_s}+(1-\alpha)c_\rho\| &< \alpha \|c_{\mu_s}\| +(1-\alpha) \|c_\rho\| \\
&= \alpha \|c_{\tilde{\mu}_s}\| +(1-\alpha) \|c_\rho\| \\
&= \| \alpha c_{\tilde{\mu}_s}+(1-\alpha)c_\rho\|,
\end{align*}
where we used the triangle inequality, the fact that a rotation preserves the norm of the centre of mass, and finally, that $c_\rho$ and $c_{\tilde{\mu}_s}$ are aligned. From this calculation we infer that by the rotation, $\|\alpha c_{\mu_s}+(1-\alpha)c_\rho\|$ increases, and hence, the total energy will be decreased (i.e., $E[\tilde{\mu}] \leq E[\mu]$). 

Based on the considerations above, the cases we have to consider are: either $\|c_\rho\|=0$ or $\|c_\rho\|\neq 0$, in which case directions of $c_\rho$ and $c_{\mu_s}$ are aligned. For both cases, we can assume that $c_\rho=\|c_\rho\|z$ and $c_{\mu_s}=\|c_{\mu_s}\|z$ for some unit vector $z\in\bbs^\dm$. 

Consider the map $f^t_z:\bbs^\dm\to\bbs^\dm$ defined for all $0\leq t<1$ by
\[
f^t_z(x)=\frac{x+tz}{\|x+tz\|}, \qquad\forall x\in\bbs^\dm,
\]
and the push-forward measure $(f_{z}^t)_\#\mu_s$ of $\mu_s$ under this map.
Then, we have
\begin{align*}
\langle  c_{(f_{z}^t)_\#\mu_s}, z\rangle&=\left\langle \int_{\bbs^\dm}x\d(f_{z}^t)_\#\mu_s(x), z\right\rangle
=\left\langle
\int_{\bbs^\dm}\frac{x+tz}{\|x+tz\|}\d\mu_s(x), z \right\rangle=\int_{\bbs^\dm}\frac{\langle x+tz, z\rangle}{\|x+tz\|}\d\mu_s(x),
\end{align*}
and by taking a derivative with respect to $t$ in the above, we find
\begin{equation}
\label{eqn:deriv-ip}
\frac{\d}{\d t}\langle  c_{(f_{z}^t)_\#\mu_s}, z\rangle=\int_{\bbs^\dm}\left(
\frac{1}{\|x+tz\|}-\frac{\langle x+tz, z\rangle^2 }{\|x+tz\|^3}
\right)\d\mu_s(x).
\end{equation}

Evaluate \eqref{eqn:deriv-ip} at $t=0$ to get
\[
\frac{\d}{\d t}\langle  c_{(f_{z}^t)_\#\mu_s}, z\rangle _{\big|_{t=0}}=\int_{\bbs^\dm}\left(
1-\langle x, z\rangle^2\right)\d\mu_s(x).
\]
If $\mu_s$ is not concentrated on $\{z,-z\}$, then $\frac{\d}{\d t}\langle  c_{(f_{z}^t)_\#\mu_s}, z\rangle_{\big|_{t=0}}>0$, and this implies
\begin{align*}
\frac{\d}{\d t}\bigl\|\alpha c_{(f_{z}^t)_\#\mu_s}+(1-\alpha)c_\rho\bigr\|^2_{\; \big|_{t=0}}
&=2\left\langle\alpha \frac{\d}{\d t}{c_{(f_{z}^t)_\#\mu_s}}_{\big|_{t=0}},(\alpha\|c_{\mu_s}\|+(1-\alpha)\|c_\rho\|)z\right\rangle\\[2pt]
&=2\alpha(\alpha\|c_{\mu_s}\|+(1-\alpha)\|c_\rho\|)\frac{\d}{\d t}\langle  c_{(f_{z}^t)_\#\mu_s}, z\rangle_{\big|_{t=0}} \\
&>0.
\end{align*}
Therefore, a perturbation $ \mu^t = \alpha (f_{z}^t)_\#\mu_s + (1-\alpha) \rho \dS $ of $\mu$ decreases its energy. 

Consider now the case when $\mu_s$ is concentrated on $\{z,-z\}$. Since we already assumed $\mu_s$ is not concentrated at one point, we can assume that $\alpha\mu_s=\alpha_1\delta_z+\alpha_2\delta_{-z}$, where $\alpha_1 \geq \alpha_2 >0$ and $\alpha=\alpha_1+\alpha_2$. Take a perturbation of $\alpha \mu_s$ given by
\[
\alpha\mu^t_s=\alpha_1\delta_z+\alpha_2\delta_{-(\cos t)z+(\sin t)w},\qquad\text{ for } 0\leq t\leq\frac{\pi}{2},
\]
for some $w\in\bbs^\dm$ that satisfies $\langle w,z \rangle =0$. Then, we have
\begin{align*}
\alpha c_{\mu_s^t}+(1-\alpha)c_\rho&=\alpha_1 z+\alpha_2\left(-(\cos t)z+(\sin t)w\right)+(1-\alpha)\|c_\rho\|z\\[2pt]
&=(\alpha_1-\alpha_2\cos t+(1-\alpha)\|c_\rho\|)z+\alpha_2(\sin t)w\\[2pt]
&=(\alpha\|c_{\mu_s}\|+\alpha_2(1-\cos t) +
(1-\alpha)\|c_\rho\|)z+\alpha_2(\sin t)w,
\end{align*}
where we used $\alpha_1-\alpha_2=\alpha\|c_{\mu_s}\|$. It yields
\[
\|\alpha c_{\mu_s^t}+(1-\alpha)c_\rho\|^2=(\alpha\|c_{\mu_s}\|+\alpha_2(1-\cos t) + (1-\alpha)\|c_\rho\|)^2+\alpha_2^2\sin^2t,
\]
and this function is increasing in $t$, for $0\leq t\leq \frac{\pi}{2}$. For this reason, a perturbation $\mu^t= \alpha \mu_s^t+ (1-\alpha)\rho\dS$ of $\mu$ decreases its energy, and we infer again that $\mu$ cannot be a ground state.

\section{Proof of Lemma \ref{lem:H-monotone}}
\label{appendix:H-monotone}

We will investigate the sign of $H'(\eta)$ for $\eta>1$. From integration by parts, we have
\begin{align*}
&\int_0^\pi (\eta-\cos\theta)^{\frac{1}{m-1}}\sin^{\dm-1}\theta \cos\theta \,\d\theta\\
&\qquad =\left[\frac{1}{\dm}(\eta-\cos\theta)^{\frac{1}{m-1}}\sin^{\dm}\theta \d\theta\right]_0^\pi-\frac{1}{\dm(m-1)}\int_0^\pi (\eta-\cos\theta)^{\frac{1}{m-1}-1}\sin^{\dm+1}\theta\d\theta\\
&\qquad =\frac{1}{\dm(1-m)}\int_0^\pi (\eta-\cos\theta)^{\frac{1}{m-1}-1}\sin^{\dm+1}\theta\d\theta, 
\end{align*}
which we use in \eqref{eqn:H} to rewrite $H$ as
\[
H(\eta)=  \frac{1}{\dm m} (\dm w_\dm)^{m-1} \left(\int_0^\pi (\eta-\cos\theta)^{\frac{1}{m-1}-1}\sin^{\dm+1}\theta\d\theta\right)
\left(\int_0^\pi (\eta-\cos\theta)^{\frac{1}{m-1}}\sin^{\dm-1}\theta \d\theta\right)^{m-2}.
\]

Then, the derivative of $\ln H(\eta)$ can be written as 
\begin{align*}
\frac{\d}{\d\eta}\ln H(\eta) & = \frac{H'(\eta)}{H(\eta)} \\
& =\left(\frac{2-m}{m-1}\right)\left(\frac{\int_0^\pi (\eta-\cos\theta)^{\frac{1}{m-1}-2}\sin^{\dm+1}\theta\d\theta}{\int_0^\pi (\eta-\cos\theta)^{\frac{1}{m-1}-1}\sin^{\dm+1}\theta\d\theta}\right) \\[3pt]
& \quad +\left(\frac{m-2}{m-1}\right)\left(
\frac{\int_0^\pi (\eta-\cos\theta)^{\frac{1}{m-1}-1}\sin^{\dm-1}\theta \d\theta}{\int_0^\pi (\eta-\cos\theta)^{\frac{1}{m-1}}\sin^{\dm-1}\theta \d\theta}
\right).
\end{align*}
Since $\frac{m-2}{m-1}>0$ for any $0<m<1$, the sign of $H'(\eta)$ is the same as the sign of
\begin{equation}
\label{eqn:H1}
\begin{aligned}
\mathcal{H}_1(\eta, m, \dm):=&\left(\int_0^\pi (\eta-\cos\theta)^{\frac{1}{m-1}-1}\sin^{\dm-1}\theta \d\theta\right)\left(
\int_0^\pi (\eta-\cos\theta)^{\frac{1}{m-1}-1}\sin^{\dm+1}\theta\d\theta
\right)\\[3pt]
&\quad - \left(\int_0^\pi (\eta-\cos\theta)^{\frac{1}{m-1}-2}\sin^{\dm+1}\theta\d\theta\right)\left(\int_0^\pi (\eta-\cos\theta)^{\frac{1}{m-1}}\sin^{\dm-1}\theta \d\theta\right).
\end{aligned}
\end{equation}

We will study the sign of $\mathcal{H}_1$ using tables of integrals that involve associated Legendre functions of the first kind, available in \cite{gradshteyn2014table}. From Section 3.666, equation (2) of \cite{gradshteyn2014table}, we have:
\[
\int_0^\pi (\cosh \beta+\sinh\beta \cos \theta)^{\mu+\nu}\sin^{-2\nu}\theta\d \theta=\frac{\sqrt{\pi}}{2^\nu}\sinh^{\nu}\beta~\Gamma\left(\frac{1}{2}-\nu\right)P_\mu^\nu(\cosh \beta)\quad\forall\mathrm{Re}(\nu)<\frac{1}{2},
\]
where $P_\mu^\nu$ denotes the associated Legendre function of the first kind. By factoring out $\sinh \beta$ from the l.h.s., the equation can be written as
\begin{equation}
\label{eqn:int-Legendre}
\int_0^\pi (\coth \beta+ \cos \theta)^{\mu+\nu}\sin^{-2\nu}\theta\d \theta=\frac{\sqrt{\pi}}{2^\nu}\sinh^{-\mu}\beta~\Gamma\left(\frac{1}{2}-\nu\right)P_\mu^\nu(\cosh \beta).
\end{equation}

Denote $\mu_0=\frac{1}{m-1}-2+\frac{\dm+1}{2}$, $\nu_0=-\frac{\dm+1}{2}$ and substitute $\eta=\coth\beta$ in \eqref{eqn:H1}, and use \eqref{eqn:int-Legendre} to simplify each integral in $\mathcal{H}_1(\eta, m, \dm)$ as follows:
\begin{align*}
\begin{cases}
\displaystyle\int_0^\pi (\eta-\cos\theta)^{\frac{1}{m-1}-1}\sin^{\dm-1}\theta \d\theta&= \displaystyle\frac{\sqrt{\pi}}{2^{\nu_0+1}}\sinh^{-\mu_0}\beta~\Gamma\left(\frac{1}{2}-(\nu_0+1)\right)P_{\mu_0}^{\nu_0+1}(\cosh \beta),
\vspace{0.2cm}
\\
\displaystyle\int_0^\pi (\eta-\cos\theta)^{\frac{1}{m-1}-1}\sin^{\dm+1}\theta\d\theta&=\displaystyle\frac{\sqrt{\pi}}{2^{\nu_0}}\sinh^{-(\mu_0+1)}\beta~\Gamma\left(\frac{1}{2}-\nu_0\right)P_{\mu_0+1}^{\nu_0}(\cosh \beta),
\vspace{0.2cm}
\\
\displaystyle\int_0^\pi (\eta-\cos\theta)^{\frac{1}{m-1}-2}\sin^{\dm+1}\theta\d\theta&=\displaystyle\frac{\sqrt{\pi}}{2^{\nu_0}}\sinh^{-\mu_0}\beta~\Gamma\left(\frac{1}{2}-\nu_0\right)P_{\mu_0}^{\nu_0}(\cosh \beta),
\vspace{0.2cm}
\\
\displaystyle\int_0^\pi (\eta-\cos\theta)^{\frac{1}{m-1}}\sin^{\dm-1}\theta \d\theta&=\displaystyle\frac{\sqrt{\pi}}{2^{\nu_0+1}}\sinh^{-(\mu_0+1)}\beta~\Gamma\left(\frac{1}{2}-(\nu_0+1)\right)P_{\mu_0+1}^{\nu_0+1}(\cosh \beta).
\vspace{0.2cm}
\end{cases}
\end{align*}
Hence, we can write $\mathcal{H}_1(\eta, m, \dm)$ as
\begin{align*}
\mathcal{H}_1(\eta, m, \dm)&=\frac{\pi}{2^{2\nu_0+1}}\sinh^{-(2\mu_0+1)}\beta \; \Gamma\left(\frac{1}{2}-\nu_0\right)\Gamma\left(-\frac{1}{2}-\nu_0\right)\\
&\qquad\times\left(P_{\mu_0}^{\nu_0+1}(\cosh \beta) P_{\mu_0+1}^{\nu_0}(\cosh \beta) -P_{\mu_0}^{\nu_0}(\cosh \beta)P_{\mu_0+1}^{\nu_0+1}(\cosh \beta)\right).
\end{align*}

The sign of $\mathcal{H}_1(\eta, m, \dm)$ is the same as the sign of 
\[
\mathcal{H}_2(\beta, m, \dm):= P_{\mu_0}^{\nu_0+1}(\cosh \beta) P_{\mu_0+1}^{\nu_0}(\cosh \beta)-P_{\mu_0}^{\nu_0}(\cosh \beta)P_{\mu_0+1}^{\nu_0+1}(\cosh \beta),
\]
where $\beta$ ranges in $(0,\infty)$.

Equation (1) in Section 8.715 of \cite{gradshteyn2014table} provides the following integral representation of the associated Legendre function of the first kind:
\[
P^\nu_\mu(\cosh\beta)=\frac{\sqrt{2}\sinh^\nu\beta}{\sqrt{\pi}\Gamma\left(\frac{1}{2}-\nu\right)}\int_0^\beta\frac{\cosh\left(\mu+\frac{1}{2}\right)t }{\left(\cosh \beta-\cosh t\right)^{\nu+\frac{1}{2}}}\d t.
\]
By using the integral expression above, we can rewrite $\mathcal{H}_2(\beta, m, \dm)$ as
\begin{align*}
\mathcal{H}_2(\beta, m, \dm)&=\frac{2\sinh^{2\nu_0+1}\beta}{\pi\Gamma\left(\frac{1}{2}-\nu_0\right)\Gamma\left(-\frac{1}{2}-\nu_0\right)}\\[3pt]
&\quad \times\bigg(
\int_0^\beta \frac{\cosh\left(\mu_0+\frac{1}{2}\right)t}{\left(\cosh \beta-\cosh t\right)^{\nu_0+\frac{3}{2}}}\d t 
\int_0^\beta \frac{\cosh\left(\mu_0+\frac{3}{2}\right)t}{\left(\cosh \beta-\cosh t\right)^{\nu_0+\frac{1}{2}}}\d t
\\[3pt]
&\qquad-\int_0^\beta \frac{\cosh\left(\mu_0+\frac{1}{2}\right)t}{\left(\cosh \beta-\cosh t\right)^{\nu_0+\frac{1}{2}}}\d t 
\int_0^\beta \frac{\cosh\left(\mu_0+\frac{3}{2}\right)t}{\left(\cosh \beta-\cosh t\right)^{\nu_0+\frac{3}{2}}}\d t
\bigg).
\end{align*}

Finally, we will investigate the sign of
\begin{align*}
\mathcal{H}_3(\beta, m, \dm):=&
\int_0^\beta \frac{\cosh\left(\mu_0+\frac{1}{2}\right)t}{\left(\cosh \beta-\cosh t\right)^{\nu_0+\frac{3}{2}}}\d t \int_0^\beta \frac{\cosh\left(\mu_0+\frac{3}{2}\right)t}{\left(\cosh \beta-\cosh t\right)^{\nu_0+\frac{1}{2}}}\d t
\\[3pt]
&\quad-\int_0^\beta \frac{\cosh\left(\mu_0+\frac{1}{2}\right)t}{\left(\cosh \beta-\cosh t\right)^{\nu_0+\frac{1}{2}}}\d t
\int_0^\beta \frac{\cosh\left(\mu_0+\frac{3}{2}\right)t}{\left(\cosh \beta-\cosh t\right)^{\nu_0+\frac{3}{2}}}\d t,
\end{align*}
for $\beta \in (0,\infty)$.
The sign of 
\[
A:= \mu_0+1 = \frac{1}{m-1}+\frac{\dm-1}{2}
\]
will be crucial in determining the sign of $\mathcal{H}_3$ (and hence, of $H$).

We express $\mathcal{H}_3$ using $A$ instead of $\mu_0$ as follows:
\begin{align*}
\mathcal{H}_3(\beta, m, \dm)=&
\int_0^\beta \frac{\cosh\left(A-\frac{1}{2}\right)t}{\left(\cosh \beta-\cosh t\right)^{\nu_0+\frac{3}{2}}}\d t \int_0^\beta \frac{\cosh\left(A+\frac{1}{2}\right)t}{\left(\cosh \beta-\cosh t\right)^{\nu_0+\frac{1}{2}}}\d t \\[3pt]
&\quad-\int_0^\beta \frac{\cosh\left(A-\frac{1}{2}\right)t}{\left(\cosh \beta-\cosh t\right)^{\nu_0+\frac{1}{2}}}\d t
\int_0^\beta \frac{\cosh\left(A+\frac{1}{2}\right)t}{\left(\cosh \beta-\cosh t\right)^{\nu_0+\frac{3}{2}}}\d t\\[3pt]
=& \int_0^\beta \frac{\cosh\left|A-\frac{1}{2}\right|t}{\left(\cosh \beta-\cosh t\right)^{\nu_0+\frac{3}{2}}}\d t \int_0^\beta \frac{\cosh\left|A+\frac{1}{2}\right|t}{\left(\cosh \beta-\cosh t\right)^{\nu_0+\frac{1}{2}}}\d t \\[3pt]
&\quad-\int_0^\beta \frac{\cosh\left|A-\frac{1}{2}\right|t}{\left(\cosh \beta-\cosh t\right)^{\nu_0+\frac{1}{2}}}\d t
\int_0^\beta \frac{\cosh\left|A+\frac{1}{2}\right|t}{\left(\cosh \beta-\cosh t\right)^{\nu_0+\frac{3}{2}}}\d t,
\end{align*}
where in the second equality we used that $\cosh(\cdot)$ is an even function. 

Now define, for $t >0$:
\[
\d \xi:=\left(\frac{\cosh\left|A+\frac{1}{2}\right|t}{\left(\cosh \beta-\cosh t\right)^{\nu_0+\frac{3}{2}}}\right)\d t, \quad f(t)=\cosh\beta-\cosh t,\quad g(t)=\frac{\cosh\left|A-\frac{1}{2}\right|t}{\cosh\left|A+\frac{1}{2}\right|t}.
\]
Hence, we write
\[
\mathcal{H}_3(\beta, m, \dm)=\int_0^\beta f(t)\d\xi\int_0^\beta g(t)\d\xi-\int_0^\beta f(t)g(t)\d\xi\int_0^\beta \d\xi.
\]

The function $f$ is decreasing on $t\geq 0$. Also, $g'$ can be calculated from
\[
\frac{g'(t)}{g(t)} =\left|A-\frac{1}{2}\right|\tanh\left|A-\frac{1}{2}\right|t-\left|A+\frac{1}{2}\right|\tanh\left|A+\frac{1}{2}\right|t,
\]
which is positive for $A<0$, negative for $A>0$ and zero for $A=0$. We distinguish between the following cases:

\noindent(Case 1: $A=0$) In this case, $g(t)=1$ and $\mathcal{H}_3\equiv0$.\\

\noindent(Case 2: $A>0$) Since $f$ and $g$ are both decreasing, $\mathcal{H}_3<0$ from the Chebyshev inequality.\\

\noindent(Case 3: $A<0$) For this case, since $f$ is decreasing and $g$ is increasing, $\mathcal{H}_3>0$ from the Chebyshev inequality.

Therefore, we can finally conclude that $H$ is increasing for $1-\frac{2}{\dm-1}<m<1$, decreasing for $0<m<1-\frac{2}{\dm-1}$, and constant for $m=1-\frac{2}{\dm-1}$, which concludes the proof of Lemma \ref{lem:H-monotone}.


\section{Proof of Lemma \ref{lem:Hlim}}
\label{appendix:Hlim}

We show that
\[
\lim_{ \eta\rightarrow \infty}H(-\eta)= \kappa_1^{-1},
\] 
where $\kappa_1$ is given by \eqref{eqn:kappa1} and the function $H$ by \eqref{eqn:H}. By factoring out $\eta$ from each of the terms in \eqref{eqn:H}, we have
\begin{equation}
\label{eqn:H-meta-1}
H(\eta) = \frac{1-m}{m} (\dm w_\dm)^{m-1} \eta
\left(\int_0^\pi \left(1- \frac{\cos\theta}{\eta} \right)^{\frac{1}{m-1}}\sin^{\dm-1}\theta \cos\theta\d\theta\right)
\left(\int_0^\pi \left(1-\frac{\cos\theta}{\eta} \right)^{\frac{1}{m-1}}\sin^{\dm-1}\theta\d\theta\right)^{m-2}.
\end{equation}

We will use Newton's generalized binomial theorem: for any $r\in \bbr$ and $x \in \bbr$ with $|x|<1$, it holds that
\begin{align*}
    (1+x)^r = \sum_{k=0}^\infty \binom{r}{k} x^k,
\end{align*}
where $ \binom{r}{k}= \frac{r(r-1)....(r-k+1)}{k!}$, with the convention $\binom{r}{0}=1$. Then, as the binomial series is uniformly convergent, we get from \eqref{eqn:H-meta-1}:
\begin{align}
H(\eta) &= \frac{1-m}{m} (\dm w_\dm)^{m-1}\eta
\left(\sum_{k=0}^{\infty}(-1)^k\binom{\frac{1}{m-1}}{k}\frac{1}{\eta^k} \int_0^\pi \cos^{k+1}\theta \sin^{d-1}\theta \d\theta\right) \\
&\qquad \times \left(\sum_{k=0}^{\infty}(-1)^k\binom{\frac{1}{m-1}}{k}\frac{1}{\eta^k} \int_0^\pi \cos^k\theta \sin^{d-1}\theta \d\theta\right)^{m-2}\\
&= \frac{1-m}{m} (\dm w_\dm)^{m-1} 
\left(\sum_{k=1}^{\infty}(-1)^k \binom{\frac{1}{m-1}}{k}\frac{1}{\eta^{k-1}} \int_0^\pi \cos^{k+1}\theta \sin^{d-1}\theta \d\theta\right) \\
&\qquad \times \left(\sum_{k=0}^{\infty}(-1)^k \binom{\frac{1}{m-1}}{k}\frac{1}{\eta^k} \int_0^\pi \cos^k\theta \sin^{d-1}\theta \d\theta\right)^{m-2},
\end{align}
where for the second equal sign we used that the $k=0$ term in the first binomial sum, vanishes. 

Then, in the limit $\eta$ tends to infinity, we find
\begin{align} 
\lim_{\eta\rightarrow\infty} H(\eta) &= \frac{1-m}{m} (dw_d)^{m-1} \left(-\binom{\frac{1}{m-1}}{1}\int_0^\pi \cos^2\theta\sin^{d-1}\theta \d\theta\right)\left(\binom{\frac{1}{m-1}}{0}\int_0^\pi\sin^{d-1}\theta\right)^{m-2}\\
&= \frac{1-m}{m} (dw_d)^{m-1} \frac{1}{1-m} \times \frac{1}{\dm+1} \left(\int_0^\pi\sin^{d-1}\theta\right)^{m-1}\\
&= \frac{|\bbs^\dm|^{m-1}}{m(d+1)},
\end{align}
where for the last equal sign we used 
\[
|\bbs^{\dm}| = |\bbs^{\dm-1}| \int_0^\pi \sin^{\dm-1} \theta\d\theta,
\]
and $|\bbs^{\dm-1}| = \dm w_\dm$. This proves the claim.


\section{Derivation of $\kappa_2$}
\label{appendix:kappa2}

We will derive the explicit expression of $\kappa_2$ given in Remark \ref{kappa2-exp}. We calculate $H(1)$ first (see \eqref{eqn:H} and \eqref{eqn:kappa2}):
\begin{equation}
\label{eqn:H1-calc}
\begin{aligned}
H(1)&=\frac{1-m}{m} (\dm w_\dm)^{m-1} \frac{\int_0^\pi (1-\cos\theta)^{\frac{1}{m-1}}\sin^{\dm-1}\theta \cos\theta\d\theta}
{\left(\int_0^\pi (1-\cos\theta)^{\frac{1}{m-1}}\sin^{\dm-1}\theta \d\theta\right)^{2-m}}\\
&=\frac{1-m}{m} (\dm w_\dm)^{m-1} \frac{\int_0^\pi 2^{\frac{1}{m-1}}\sin^{\frac{2}{m-1}}\left(\frac{\theta}{2}\right)\sin^{\dm-1}\theta \cos\theta\d\theta}
{\left(\int_0^\pi 2^{\frac{1}{m-1}}\sin^{\frac{2}{m-1}}\left(\frac{\theta}{2}\right)\sin^{\dm-1}\theta \d\theta\right)^{2-m}}\\
&=2\left(\frac{1-m}{m}\right) (\dm w_\dm)^{m-1} \frac{\int_0^\pi \sin^{\frac{2}{m-1}}\left(\frac{\theta}{2}\right)\sin^{\dm-1}\theta \cos\theta\d\theta}
{\left(\int_0^\pi\sin^{\frac{2}{m-1}}\left(\frac{\theta}{2}\right)\sin^{\dm-1}\theta \d\theta\right)^{2-m}}.
\end{aligned}
\end{equation}

Using the identity
\begin{equation}
\label{eqn:Gamma-prop}
\int_0^\pi\cos^a\left(\frac{\theta}{2}\right)\sin^b\left(\frac{\theta}{2}\right)\d\theta=\frac{\Gamma\left(\frac{a+1}{2}\right)\Gamma\left(\frac{b+1}{2}\right)}{\Gamma\left(\frac{a+b+2}{2}\right)},
\end{equation}
the property $\Gamma(z+1)=z\Gamma(z)$ for all $z>0$, and the elementary trig identities
\[
\sin(\theta) = 2 \sin \left( \frac{\theta}{2}\right) \cos \left( \frac{\theta}{2} \right),\qquad \cos \theta = 2 \cos^2 \left( \frac{\theta}{2} \right)-1,
\]
we can simplify the numerator and denominator in \eqref{eqn:H1-calc} as
\begin{align*}
&\int_0^\pi \sin^{\frac{2}{m-1}}\left(\frac{\theta}{2}\right)\sin^{\dm-1}\theta \cos\theta\d\theta\\
=&2^{\dm-1}\int_0^\pi \left(\sin^{\frac{2}{m-1}+\dm-1}\left(\frac{\theta}{2}\right)\cos^{\dm+1}\left(\frac{\theta}{2}\right)-\sin^{\frac{2}{m-1}+\dm+1}\left(\frac{\theta}{2}\right)\cos^{\dm-1}\left(\frac{\theta}{2}\right)\right)\d\theta\\
=&2^{\dm-1}\left(\frac{\Gamma\left(\frac{1}{m-1}+\frac{\dm}{2}\right)\Gamma\left(\frac{\dm+2}{2}\right)-\Gamma\left(\frac{1}{m-1}+\frac{\dm+2}{2}\right)\Gamma\left(\frac{\dm}{2}\right)}{\Gamma\left(\frac{1}{m-1}+\dm+1\right)}\right)\\
=&2^{\dm-1}\left(\frac{\Gamma\left(\frac{1}{m-1}+\frac{\dm}{2}\right)\Gamma\left(\frac{\dm}{2}\right)}{\Gamma\left(\frac{1}{m-1}+\dm+1\right)}\right)\left(\frac{\dm}{2}-\left(\frac{1}{m-1}+\frac{\dm}{2}\right)\right)\\
=&\frac{2^{\dm-1}\Gamma\left(\frac{1}{m-1}+\frac{\dm}{2}\right)\Gamma\left(\frac{\dm}{2}\right)}{(1-m)\Gamma\left(\frac{1}{m-1}+\dm+1\right)},
\end{align*}
and
\begin{align*}
\int_0^\pi \sin^{\frac{2}{m-1}}\left(\frac{\theta}{2}\right)\sin^{\dm-1}\theta \d\theta
&=2^{\dm-1}\int_0^\pi \sin^{\frac{2}{m-1}+\dm-1}\left(\frac{\theta}{2}\right)\cos^{\dm-1}\left(\frac{\theta}{2}\right)\d\theta\\
&=\frac{2^{\dm-1}\Gamma\left(\frac{1}{m-1}+\frac{\dm}{2}\right)\Gamma\left(\frac{\dm}{2}\right)}{\Gamma\left(\frac{1}{m-1}+\dm\right)}.
\end{align*}

Then, we can simplify $H(1)$ from \eqref{eqn:H1-calc} as follows:
\begin{align*}
H(1)&=2\left(\frac{1-m}{m}\right) (\dm w_\dm)^{m-1} \frac{\left(
\frac{2^{\dm-1}\Gamma\left(\frac{1}{m-1}+\frac{\dm}{2}\right)\Gamma\left(\frac{\dm}{2}\right)}{(1-m)\Gamma\left(\frac{1}{m-1}+\dm+1\right)}\right)}
{
\left(
\frac{2^{\dm-1}\Gamma\left(\frac{1}{m-1}+\frac{\dm}{2}\right)\Gamma\left(\frac{\dm}{2}\right)}{\Gamma\left(\frac{1}{m-1}+\dm\right)}
\right)^{2-m}
}\\
&=\frac{2^{1+(\dm-1)(m-1)}}{m\left(\frac{1}{m-1}+\dm\right)} (\dm w_\dm)^{m-1} \left(
\frac{\Gamma\left(\frac{1}{m-1}+\frac{\dm}{2}\right)\Gamma\left(\frac{\dm}{2}\right)}{\Gamma\left(\frac{1}{m-1}+\dm\right)}\right)^{m-1}.
\end{align*}
Finally, use $|\bbs^\dm|=\frac{\dm w_\dm\Gamma\left(\frac{1}{2}\right)\Gamma\left(\frac{\dm}{2}\right)}{\Gamma\left(\frac{\dm+1}{2}\right)}$ to write the above expression as
\[
H(1)=\frac{2^{1+(\dm-1)(m-1)}|\bbs^\dm|^{m-1}}{m\left(\frac{1}{m-1}+\dm\right)}\left(\frac{\Gamma\left(\frac{1}{m-1}+\frac{\dm}{2}\right)\Gamma\left(\frac{\dm+1}{2}\right)}{\Gamma\left(\frac{1}{2}\right)\Gamma\left(\frac{1}{m-1}+\dm\right)}\right)^{m-1}.
\]
The explicit expression for $\kappa_2$ listed in Remark \ref{kappa2-exp}, follows then from $\kappa_2=H(1)^{-1}$.


\section{Calculation of $\srho$}
\label{appendix:s}
We calculate from \eqref{eqn:srho-kgk2} using simple trig identities:
\begin{align*}
\srho &=1-\frac{\int_0^\pi (1-\cos\theta)^{\frac{1}{m-1}}\sin^{\dm-1}\theta (1-\cos\theta)\d\theta}{\int_0^\pi (1-\cos\theta)^{\frac{1}{m-1}}\sin^{\dm-1}\theta\d\theta}\\[5pt]
&=1-\frac{\int_0^\pi \left(2\sin^2\frac{\theta}{2}\right)^{\frac{1}{m-1}+1}\left(2\sin\frac{\theta}{2}\cos\frac{\theta}{2}\right)^{\dm-1} \d\theta}{\int_0^\pi \left(2\sin^2\frac{\theta}{2}\right)^{\frac{1}{m-1}}\left(2\sin\frac{\theta}{2}\cos\frac{\theta}{2}\right)^{\dm-1}\d\theta}\\[5pt]
&=1-2 \, \frac{\int_0^\pi \sin^{\frac{2}{m-1}+\dm+1}\frac{\theta}{2}\cos^{\dm-1}\frac{\theta}{2}\d\theta}{\int_0^\pi \sin^{\frac{2}{m-1}+\dm-1}\frac{\theta}{2}\cos^{\dm-1}\frac{\theta}{2}\d\theta}.
\end{align*}
Then, using the Gamma function and its properties (see \eqref{eqn:Gamma-prop}) we find
\begin{align*}
\srho &=1-2 \, \frac{\frac{\Gamma\left(\frac{1}{m-1}+\frac{\dm}{2}+1\right)\Gamma\left(\frac{\dm}{2}\right)}{\Gamma\left(\frac{1}{m-1}+\dm+1\right)}}{\frac{\Gamma\left(\frac{1}{m-1}+\frac{\dm}{2}\right)\Gamma\left(\frac{\dm}{2}\right)}{\Gamma\left(\frac{1}{m-1}+\dm\right)}} \\[3pt]
&=1-2 \, \frac{\frac{1}{m-1}+\frac{\dm}{2}}{\frac{1}{m-1}+\dm} \\
&=1-\frac{2+\dm(m-1)}{1+\dm(m-1)}\\[3pt]
&=\frac{1}{(1-m)\dm-1}.
\end{align*}


\bibliographystyle{abbrv}
\def\url#1{}
\bibliography{lit-H.bib}

\end{document}